\documentclass[12pt]{amsart}
\usepackage{amsmath,amsfonts,graphicx}
\usepackage{amssymb} 
\usepackage{amsthm}
\newtheorem{thm}{Theorem}[section]
\newtheorem{prp}[thm]{Proposition}
\newtheorem{lmm}[thm]{Lemma}
\newtheorem{cor}[thm]{Corollary}
\newtheorem{dfn}[thm]{Definition}
\newtheorem{rmk}[thm]{Remark}

\renewcommand{\labelenumi}{(\theenumi)}

\def\RR{{\mathbb R}}
\def\CC{{\mathbb C}}
\def\NN{{\mathbb N}}
\def\ZZ{{\mathbb Z}}
\def\PP{{\mathbb P}}

\begin{document}
\title{Homotopy Shadowing}
\author{Yutaka Ishii and John Smillie}
\date{\it Version, July 7, 2010.}
\address{Department of Mathematics, Kyushu University, Motooka, Fukuoka 819--0395, Japan. Email: {\tt yutaka{\char'100}math.kyushu-u.ac.jp}}
\address{Department of Mathematics, Cornell University, Malott Hall, Ithaca, NY 14853-4201, USA. Email: {\tt smillie{\char'100}math.cornell.edu}}

\begin{abstract}
Michael Shub proved in 1969 that the topological conjugacy class of an expanding endomorphism on a compact manifold is determined by its homotopy type. In this article we generalize this result in two directions. In one direction we consider certain expanding maps on metric spaces. In a second direction we consider maps which are hyperbolic with respect to product cone fields on a product manifold. A key step in the proof is to establish a shadowing theorem for pseudo-orbits with some additional homotopy information.
\end{abstract}

\maketitle

\begin{center}
\end{center}

\tableofcontents

\begin{section}{Introduction}
\label{sec:intro}

This paper deals with the problem of showing that two different hyperbolic dynamical systems are topologically conjugate. Structural stability says that two hyperbolic maps which are sufficiently close are topologically conjugate. Our objective though is to obtain conjugacies between systems which are not assumed to be close to one another.

Our approach will be to build models for the dynamical systems and show that an appropriate notion of homotopy equivalence of models establishes topological conjugacy of the original systems. Since we want to deal with the restrictions of dynamical systems to sets which are not invariant it is necessary for us to work with partially defined dynamical systems. It will turn out to be natural to take this one step further and allow our systems to be multiply valued as well as partially defined (see Section \ref{sec:pdmds}).

We start by giving an example of a situation where our techniques apply. Let $f$ be defined on a manifold $M$ and let $\Lambda$ be an invariant set. Let $U$ be a neighborhood of $\Lambda$ and assume that $\Lambda$ is the maximal invariant set in $U$. We moreover assume that either (i) the partially defined map $f$ is {\it expanding} on $U$ or (ii) $U$ can be written as a product of two Riemannian manifolds $U_x\times U_y$ and $f$ is {\it hyperbolic with respect to product cone fields} on $U$ associated to the product structure. In this paper we show that the homotopy type of the restriction of $f$ to $U$ determines $f$ restricted to $\Lambda$ up to topological conjugacy.

\begin{thm} 
\label{thm:exp}
If $f$ is either expanding or hyperbolic with respect to product cone fields on a neighborhood $U$ of $\Lambda$ then the topological conjugacy class of $f : \Lambda \to \Lambda$ depends only on the homotopy type of the restriction of $f$ to $U$.
\end{thm}

We will give precise definitions and related results in Sections~\ref{sec:expsyst} and \ref{sec:hypsyst}.

In the case of expanding maps of compact manifolds we can take $\Lambda=U=M$. In this case Shub's theorem~\cite{S} shows that the action induced by $f$ on the fundamental group $\pi_1(M)$ of $M$ determines the dynamics of $f$. Our result yields Shub's theorem as a special case. 

In the case that $f$ is expanding on $U$ and $U$ is  arcwise connected the theory of iterated monodromy groups~\cite{N} yields combinatorial models of $\Lambda$ which can be determined from the action of $f$ on the fundamental group of $U$. There is a connection between our first theorem and the results of~\cite{N} though our approach is more dynamical and less group theoretic. The relation between the results of  this paper and the theory of iterated monodromy groups will be investigated in~\cite{HPS}. 

In the hyperbolic case the result above was motivated by our interest in complex H\'enon diffeomorphisms though our result as stated makes no reference to any complex structure. 

There are other settings where homotopy information about neighborhoods is used to analyze $\Lambda$. One of these is the theory of the Conley index (see~\cite{FR} and the references contained there). This theory deals with determining invariants of $f$ restricted to $\Lambda$ from homotopy invariants of its neighborhood $U$. On the one hand this theory is more general in that it makes no assumptions about the hyperbolicity of $f$. On the other hand it does not give conditions under which the conjugacy type of $\Lambda$ is determined.

We will also prove another statement which describes the dynamics of a hyperbolic system in terms of its {\it associated expanding system}. In Section~\ref{sec:associated} we construct a partially defined expanding map $f_x : U_x\to U_x$ from a hyperbolic system $f : U\to U$ which we call the associated expanding system of $f$. The map that we construct depends on certain choices but the homotopy equivalence class of the map is well defined (see Proposition~\ref{prp:independence}). We construct a map $f_x : \Lambda_x\to \Lambda_x$ from the homotopy model associated to $f_x : U_x\to U_x$ and show that the topological conjugacy class of the resulting map is determined. 

\begin{thm} 
\label{thm:hyp}
If $f$ is hyperbolic with respect to product cone fields on a neighborhood $U$ of $\Lambda$ then the restriction of $f$ to $\Lambda$ is topologically conjugate to the inverse limit of the restriction of $f_x$ to $\Lambda_x$.
\end{thm}

We will give precise definitions and the statement in Sections~\ref{sec:hypsyst} and~\ref{sec:associated}.

Our interest in these questions was partially motivated by complex dynamics in one and two variables. In one-dimensional complex dynamics there are natural metrics on neighborhoods $U$ of $\Lambda$ which can be used to show that $f$ is expanding. In two complex dimensions Hubbard and Oberste-Vorth~\cite{HO} gave conditions on cone fields in a neighborhood $U$ of $\Lambda$ which show that $f$ is hyperbolic. Additional conditions on product cone fields were given in~\cite{I1}. In both cases it seems natural to use these neighborhoods as tools to analyze the sets $\Lambda$ and construct conjugacies and semi-conjugacies between systems.

This situation arises for complex H\'enon maps $f : \CC^2 \to \CC^2$ given by
$$f=f_{c, b} : (x, y)\longmapsto (p_c(x)-by, x),$$
where $p_c(x)=x^2+c$. Hubbard and Oberste-Vorth~\cite{HO} proved that if $p_c$ is expanding and $|b|$ is sufficiently small then $f_{c, b}$ is hyperbolic on its Julia set $J_f$ and the H\'enon map $f$ on $J_f$ is topologically conjugate to the inverse limit of $p_c$ on its Julia set. However, the conjugacy is not easy to calculate explicitly except for the horseshoe Julia set. As a consequence of Theorem~\ref{thm:hyp} together with~\cite{I1} it follows
\begin{cor}
\label{cor:Henon}
For the complex H\'enon map $f=f_{c, b}$, we have the following.
\begin{enumerate}
 \renewcommand{\labelenumi}{(\roman{enumi})}
  \item If $|c|>2(1+|b|)^2$, then $f : J_f\to J_f$ is topologically conjugate to a horseshoe.
  \item If $c=0$ and $|b|<(\sqrt{2}-1)/2$, then $f : J_f\to J_f$ is topologically conjugate to the solenoid.
  \item If $c=-1$ and $|b|<0.02$, then $f : J_f\to J_f$ is topologically conjugate to the inverse limit of the basilica.
\end{enumerate}
Moreover, there are explicit maps on orbits which realize these conjugacies.
\end{cor}

The estimate (i) has been essentially obtained in~\cite{O}, but there is a trivial arithmetic error in~\cite{O} which leads to a different condition for $c$. Modulo this error the proof given establishes the result (i) stated above.

We remark that not all hyperbolic H\'enon maps have the property of restricting to a map on the Julia set which is topologically conjugate to the inverse limit of any expanding one-dimensional polynomial map. In~\cite{I1} such examples are constructed for cubic H\'enon maps. In particular the product cone field hypothesis of Theorem~\ref{thm:hyp} does not hold for all hyperbolic H\'enon maps.

In the case of polynomial maps of $\CC$ natural models (in our sense) of Julia sets can be obtained from Hubbard trees~\cite{D, M}. In~\cite{I2} the results of this paper have been applied to the construction of Hubbard trees for a class of hyperbolic H\'enon maps including the non-perturbative one~\cite{I1} mentioned in the previous paragraph.

We will use the technique of {\it homotopy pseudo-orbits} to prove both theorems. An $\varepsilon$ pseudo-orbit is a sequence $(x_i)$ for which $d(f(x_i),x_{i+1})<\varepsilon$. A homotopy pseudo-orbit is a sequence $(x_i)$ for which we have chosen a path from $f(x_i)$ to $x_{i+1}$ (see Section \ref{sec:hpo}). In the proof of the structural stability theorem a large role is played by a shadowing theorem. We establish a shadowing theorem for homotopy pseudo-orbits (called a {\it homotopy shadowing theorem}) which says that any homotopy pseudo-orbit is homotopic to a unique orbit (see Sections \ref{sec:expsyst} and \ref{sec:hypsyst}).  In this way our argument is very much in the spirit of classical dynamical systems and we hope that our approach has the advantage of seeming  natural to dynamicists. Our results give explicit conjugacies and lead to numerical algorithms which can be implemented by computer. See \cite{Mu} for results in this direction.

In this paper we deal with both the expanding case and the hyperbolic case. For the reader who is interested in one of these cases but not the other we suggest the following. Those readers interested only in the case of expanding dynamics should read Sections 2, 3, 4, 6, 7, 10. Those readers interested only in the case of hyperbolic dynamics should read Sections 2, 3, 5, 6, 8, 10.

\vspace{0.3cm}

\noindent 
{\it Acknowledgment.} The authors are grateful to Masayuki Asaoka for pointing out an error in the proof of Proposition~\ref{prp:buildup} and supplying a correct proof.

\end{section}

\begin{section}{Partially defined and multivalued dynamical systems}
\label{sec:pdmds}

\subsection{Multivalued dynamical systems}
\label{sub:mds}

We are interested in finding a general setting in which shadowing ideas can be applied. We observe that the notion of shadowing as it is usually applied does not require that the dynamical system be everywhere defined. For example if we have a pseudo-orbit  in a neighborhood $U$ of a hyperbolic invariant set $\Lambda$ with local product structure then the shadowing principle allows us to shadow this pseudo-orbit with an actual orbit in $\Lambda$. The fact that the map is not everywhere defined in $U$ is not a hindrance. What is important is that the pseudo-orbit is defined for all time and that the orbit that it shadows is defined for all time. 

This discussion suggests that it would be useful to discuss the notion of shadowing in the context of partially defined dynamical systems. We will explain why a further extension to multiple valued dynamical systems is also useful. We will show that for a partially defined dynamical system $f:U\to U$ with appropriate hyperbolicity properties the dynamics on the invariant set $\Lambda$ depends only on the homotopy type of $U$ and the homotopy class of $f$. On the other hand replacing $U$ by a homotopy equivalent space can turn $f$ from a single valued map to a multivalued map. To give a familiar example say that $U$ is a disjoint union of of two disks $D_1$ and $D_2$. Assume that each disk is mapped across both $D_1$ and $D_2$ by an expanding partially defined map $f$ so that $\Lambda$ is a horseshoe. If we replace $U$ by the homotopy equivalent space consisting of two points $x_1$ and $x_2$ then each point should map to both points. Thus it will be useful for us to extend the definition of dynamical system so that it includes ``maps'' which may be undefined at certain points as well as maps which may take on multiple values at certain points. Keep in mind that we use such multivalued maps as tools for studying classical (single valued) dynamical systems. 

The following definitions give us a general set-up for defining ``multivalued dynamical systems''. In Sections~\ref{sec:expsyst} and~\ref{sec:hypsyst} we will consider the expansion and hyperbolicity hypotheses we need in order to make shadowing arguments work for multivalued dynamical systems.

\begin{dfn} 
\label{dfn:mds}
A pair of spaces $X^0$ and $X^1$ together with a pair of maps $\iota, \sigma : X^1\to X^0$ between them is called a {\it multivalued dynamical system}.
\end{dfn}

When we denote this formally we describe it as a quadruple $\mathcal{X}=(X^0, X^1; \iota, \sigma)$. We denote this quadruple by $\iota, \sigma : X^1\to X^0$. We make the convention that if $\sigma : X\to X$ is a classical dynamical system then we view it as a multivalued dynamical system $\iota, \sigma : X\to X$ where $\iota : X\to X$ is the identity.

 \begin{figure}
\setlength{\unitlength}{0.9mm}
\begin{picture}(50,130)(50,0)

\linethickness{1.2pt}
\put(0, 0){\line(1, 0){50}}
\put(50, 0){\line(0, 1){50}}
\put(0, 0){\line(0, 1){50}}
\put(0, 50){\line(1, 0){50}}

\put(100, 0){\line(1, 0){50}}
\put(150, 0){\line(0, 1){50}}
\put(100, 0){\line(0, 1){50}}
\put(100, 50){\line(1, 0){50}}

\qbezier(90,120)(98,107)(98,70)
\qbezier(78,120)(86,107)(86,70)
\qbezier(72,120)(64,107)(64,70)
\qbezier(60,120)(52,107)(52,70)
\put(60,120){\line(1, 0){12}}
\put(78,120){\line(1, 0){12}}
\put(52,70){\line(1, 0){12}}
\put(86,70){\line(1, 0){12}}

\small
\put(17, 55){$X^0=B$}
\put(59, 125){$X^1=B\cap f^{-1}(B)$}
\put(117, 55){$X^0=B$}
\put(20, 15){$\iota(X^1)$}
\put(130, 24){$\sigma(X^1)$}
\put(50, 60){$\iota$}
\put(98, 60){$\sigma$}

\linethickness{0.3pt}
\put(50, 70){\line(1, 0){50}}
\put(100, 70){\line(0, 1){50}}
\put(50, 70){\line(0, 1){50}}
\put(50, 120){\line(1, 0){50}}

\put(62, 67){\vector(-1, -1){15}}
\put(88, 67){\vector(1, -1){15}}
\put(20, 20){\line(-1, 1){10}}
\put(28, 20){\line(1, 1){10}}
\put(129, 23){\line(-1, -1){12}}
\put(129, 27){\line(-1, 1){12}}

\qbezier(40,50)(48,37)(48, 0)
\qbezier(28,50)(36,37)(36, 0)
\qbezier(22,50)(14,37)(14, 0)
\qbezier(10,50)( 2,37)( 2, 0)
\put(10,50){\line(1, 0){12}}
\put(28,50){\line(1, 0){12}}
\put(2, 0){\line(1, 0){12}}
\put(36, 0){\line(1, 0){12}}

\qbezier(100, 40)(113, 48)(150, 48)
\qbezier(100, 28)(113, 36)(150, 36)
\qbezier(100, 22)(113, 14)(150, 14)
\qbezier(100, 10)(113, 2)(150, 2)
\put(100, 10){\line(0, 1){12}}
\put(100, 28){\line(0, 1){12}}
\put(150, 2){\line(0, 1){12}}
\put(150, 36){\line(0, 1){12}}

\end{picture}
\vspace{1cm}
\begin{center}
Figure 1. A multivalued dynamical system $\iota, \sigma : X^1\to X^0$.
\end{center}
\end{figure}

An example of a multivalued dynamical system is given by a horseshoe map. Let $B\subset \RR^2$ be a square in the plane and let $f : B\to \RR^2$ be a standard horseshoe map. Let $X^0=B$ and $X^1=B\cap f^{-1}(B)$. Let $\sigma : X^1\to X^0$ be the restriction of $f$ to $B\cap f^{-1}(B)$ and let $\iota : X^1\to X^0$ be the inclusion of $B\cap f^{-1}(B)$ into $B$. Then, $\iota, \sigma : X^1\to X^0$ becomes a multivalued dynamical system (see Figure 1).

 \begin{figure}
\setlength{\unitlength}{1mm}
\begin{picture}(100, 100)(25, 0)

\put(0, 0){\line(1, 0){30}}
\put(30, 0){\line(0, 1){30}}
\put(0, 0){\line(0, 1){30}}
\put(0, 30){\line(1, 0){30}}

\put(60, 0){\line(1, 0){30}}
\put(90, 0){\line(0, 1){30}}
\put(60, 0){\line(0, 1){30}}
\put(60, 30){\line(1, 0){30}}

\put(120, 0){\line(1, 0){30}}
\put(150, 0){\line(0, 1){30}}
\put(120, 0){\line(0, 1){30}}
\put(120, 30){\line(1, 0){30}}

\put(30, 50){\line(1, 0){30}}
\put(60, 50){\line(0, 1){30}}
\put(30, 50){\line(0, 1){30}}
\put(30, 80){\line(1, 0){30}}

\put(90, 50){\line(1, 0){30}}
\put(120, 50){\line(0, 1){30}}
\put(90, 50){\line(0, 1){30}}
\put(90, 80){\line(1, 0){30}}

\put(0, 38){\vector(1, -2){13.2}}
\put(48.5, 61){\vector(1, -2){24.7}}
\put(47.5, 61){\vector(-2, -3){33}}
\put(108.5, 61){\vector(1, -2){24.7}}
\put(107.5, 61){\vector(-2, -3){33}}
\put(150, 35.2){\vector(-2, -3){15.7}}

\put(48, 62.3){\circle*{1.5}}
\put(108, 62.3){\circle*{1.5}}
\put(13.8, 10){\circle*{1.5}}
\put(73.8, 10){\circle*{1.5}}
\put(133.8, 10){\circle*{1.5}}

\small
\put(14, 33){$X^0$}
\put(74, 33){$X^0$}
\put(134, 33){$X^0$}
\put(44, 83){$X^1$}
\put(104, 83){$X^1$}
\put(45, 65){$x_{i-1}$}
\put(107, 65){$x_i$}
\put(31, 40){$\iota$}
\put(60, 40){$\sigma$}
\put(91, 40){$\iota$}
\put(120, 40){$\sigma$}
\put(0.5, 5){$\sigma(x_{i-2})=\iota(x_{i-1})$}
\put(62, 5){$\sigma(x_{i-1})=\iota(x_i)$}
\put(122.5, 5){$\sigma(x_i)=\iota(x_{i+1})$}

\end{picture}
\vspace{1cm}
\begin{center}
Figure 2a. An orbit $x$ : elements $x_i$ of $X^1$ are drawn as points.
\end{center}
\end{figure}

 \begin{figure}
\setlength{\unitlength}{1mm}
\begin{picture}(100, 50)(25, 0)

\qbezier(0, 26)(8, 19)(13.3, 11)
\qbezier(14.3, 11)(43.8, 60)(73.3, 11)
\qbezier(74.3, 11)(103.8, 60)(133.3, 11)
\qbezier(134.3, 11)(139.6, 19)(147.6, 26)
\put(13.8, 10){\circle*{1.5}}
\put(73.8, 10){\circle*{1.5}}
\put(133.8, 10){\circle*{1.5}}
\put(12.3, 12.5){\vector(2, -3){1}}
\put(72.3, 12.5){\vector(2, -3){1}}
\put(132.3, 12.5){\vector(2, -3){1}}

\small
\put(41, 37){$x_{i-1}$}
\put(103, 37){$x_i$}
\put(0.5, 5){$\sigma(x_{i-2})=\iota(x_{i-1})$}
\put(62, 5){$\sigma(x_{i-1})=\iota(x_i)$}
\put(122.5, 5){$\sigma(x_i)=\iota(x_{i+1})$}

\end{picture}
\vspace{1cm}
\begin{center}
Figure 2b. An orbit $x$ : elements $x_i$ of $X^1$ are drawn as arrows.
\end{center}
\end{figure}

The key property of multivalued dynamical systems is that they give rise to orbits.

\begin{dfn} 
\label{dfn:orbit}
We define an {\it orbit} for a multivalued dynamical system $\iota, \sigma : X^1\to X^0$ to be a sequence $(x_i)$ of points in $X^1$ such that $\sigma(x_i)=\iota(x_{i+1})$. 
\end{dfn}

An orbit gives us an infinite sequence of points in $X^1$ and an infinite sequence of points in $X^0$. In Figure 2a we adopt the convention of drawing a distinct copy of $X^1$ for each point $x_i$ in $X^1$ and a distinct copy of $X^0$ for each point $\sigma(x_{i-1})=\iota(x_i)$.

It can also be useful to think of the elements of $X^1$ as arrows between points in $X^0$. We interpret $x\in X^1$ as an arrow going from the point $\iota(x)$ to the point $\sigma(x)$. If we do this then an orbit is a sequence of arrows where the head of one ends at the tail of the next (see Figure 2b).

The orbit can be finite $(x_1, \ldots, x_n)$, forward infinite $(x_i)_{i\geq 0}$, backward infinite $(x_i)_{i\leq 0}$ or bi-infinite $(x_i)_{i\in\ZZ}$. We will use the notation $(x_i)$ in each of these cases hoping that our meaning can be determined from the context. 

Spaces of orbits will be useful objects to consider.

\begin{dfn} 
\label{dfn:orbsp}
Let $X^{+\infty}$ denote the space of forward orbits $\{(x_i)_{i\geq 0} : \sigma(x_i)=\iota(x_{i+1})\}$. Let $X^{\pm\infty}$ denote the space of bi-infinite orbits $\{(x_i)_{i\in \ZZ} : \sigma(x_i)=\iota(x_{i+1})\}$. Let $X^{\infty}$ denote either the space of forward orbits or the space of bi-infinite orbits.
\end{dfn}

We view these sets as topological spaces where the topology is the  topology that they inherit as subsets of product spaces.

We can think of our multivalued dynamical system as giving rise to an abstract (single valued) dynamical system by considering the shift map on the space of orbits. 

\begin{dfn} 
\label{dfn:shift}
If $\iota, \sigma : X^1\to X^0$ is a multivalued dynamical system then let $\hat \sigma : X^{+\infty}\to X^{+\infty}$ and $\hat \sigma : X^{\pm\infty}\to X^{\pm\infty}$ denote the {\it shift maps} defined so that $\hat \sigma((x_i))=(y_i)$ where $y_i=x_{i+1}$.
\end{dfn}

Thus the shift maps on $X^{+\infty}$ and $X^{\pm\infty}$ are dynamical systems in the ordinary sense of the word. The shift map on $X^{\pm\infty}$ is invertible.

In the case of a classical dynamical system $f : X\to X$ the space $X^{+\infty}$ is just $X$. If the map $f$ is invertible then the space $X^{\pm\infty}$ is also $X$. If $f$ is not invertible then $X^{\pm\infty}$ can be identified with the inverse limit $\varprojlim (X, f)$ which is also known as the natural extension of $f$.

\subsection{Partially defined maps}
\label{sub:pdm}

A {\it partially defined map} is a multivalued dynamical system for which $\iota$ is injective. The injectivity of $\iota$ gives a forward determinism in that for each $x\in X^0$ there is a unique forward orbit starting with $x$ though this orbit may have finite length. If we interpret $\sigma\iota^{-1}(x)$ as the set of possible values of our multivalued dynamical system then in the case of partially defined maps our ``multivalued dynamical system'' actually takes on 0 or 1 value at each point.

We can obtain examples of partially defined maps by restricting classical dynamical systems to subsets of the domain which are not invariant. As an example of such a system let $f : \CC\PP^1\to \CC\PP^1$ be a rational map. Let $U$ be a neighborhood of the Julia set $J_f$ of $f$ which is chosen so that $f^{-1}(U)\subset U$ and every point not in the Julia set eventually leaves $U$. We construct a multivalued dynamical system $\iota, \sigma: X^1\to X^0$. Let $X^0$ be the set $U$. Let $X^1$ be the set $f^{-1}(U)$ and let $\iota$ be the inclusion map from $f^{-1}(U)$ to $U$ and let $\sigma$ denote the restriction of $f$ to $X^1$. In this example a forward orbit for our partially defined dynamical system is an orbit for $f$ which remains in $U$ and we can identify $X^{+\infty}$ with the set $\bigcap_{n=1}^\infty f^{-n}(U)=J_f$. A bi-infinite orbit is a forward orbit together with a choice of a prehistory for the orbit so we can identify $X^{\pm\infty}$ with the inverse limit $\varprojlim(J_f,f)$.

Our notion of multivalued dynamical system allows us to restrict an invertible map to an arbitrary set. For example if $f$ is a H\'enon diffeomorphism of $\CC^2$ and $B$ is a set which contains all bounded orbits then let $X^0=B$ and let $X^1=B\cap f^{-1}(B)$. Let $\sigma:X^1\to X^0$ be the map $f|_{B\cap f^{-1}(B)} : B\cap f^{-1}(B)\to  B$ and let $\iota$ be the inclusion of $B\cap f^{-1}(B)$ into $B$. Then we have a multivalued dynamical system $\iota, \sigma : B\cap f^{-1}(B)\to B$ (see Figure 1 again for the horseshoe case). In this case both  $\iota$ and $\sigma$ are injective. A point in $X^0$ determines a unique orbit and we can identify $X^{\pm\infty}$ with $\bigcap_{n\in \ZZ}f^n(B)$.

\subsection{Subshifts of finite type}
\label{sub:sft}

Consider a multivalued dynamical system $\iota, \sigma: X^1\to X^0$ where the spaces $X^0$ and $X^1$ are finite sets. We can identify such a multivalued dynamical system with a directed graph where $X^0$ is the set of vertices, $X^1$ is the set of arrows, the map $\iota$ maps each arrow to its tail and the map $\sigma$ maps each arrow to its head. An orbit for such a multivalued dynamical system is a sequence of arrows  $x_1 x_2 \cdots x_n$ where the head of $x_i$ coincides with the tail of $x_{i+1}$ so that they form an oriented path in the graph. The spaces $X^{+\infty}$ and $X^{\pm\infty}$ together with the shift map $\hat\sigma$ are the corresponding one-sided and bi-infinite subshifts of finite type.

Subshifts of finite type give efficient ways to represent certain dynamical systems. Their efficiency is due to the fact that they are genuinely multivalued that is to say that $\iota$ is not injective. Other genuinely multivalued dynamical systems will arise when we take, for example, a neighborhood of an expanding Julia set and replace it by a homotopy equivalent one-complex. In this situation we are also producing efficient models for complicated dynamics.

\subsection{Hybrid examples}
\label{sub:hybrid}

We can construct models which mix the properties of the previous examples. Say that we have a directed graph where to each vertex $j$ we have assigned a topological space $Y_j$ and to each edge from $j$ to $j'$ we have assigned a space $Z_{j,j'}$ and a pair of maps $\iota_{j,j'}: Z_{j,j'} \to Y_j$ and $\sigma_{j,j'} : Z_{j,j'}\to Y_{j'}$. Then we can define a multivalued dynamical system $\iota, \sigma : X^1\to X^0$ by taking $X^1$ to be the disjoint union of the spaces $Z_{j,j'}$ and taking $X^0$ to be the disjoint union of the spaces $Y_j$ and defining $\iota(x)$ to be $\iota_{j,j'}(x)$ when $x\in Z_{j,j'}$ and defining $\sigma(x)$ to be $\sigma_{j,j'}(x)$ when $x\in Z_{j,j'}$.

Multivalued dynamical systems like these arise in the work of Ishii~\cite{I2} and of Bedford and Smillie~\cite{BS} on complex H\/enon maps where they are called systems of crossed mappings. The notion of crossed mapping first appeared in \cite{HO}. We can think of such systems of crossed mappings as arising when we want to consider a ``restriction'' of a map to a collection of subsets of the domain which are not necessarily disjoint. Say that $f:\CC^2\to\CC^2$ is a complex H\'enon map and $Y_1, \ldots, Y_j$ are subsets of $\CC^2$. Then we can construct a hybrid system as above by taking $Z_{j,j'}$ to be $Y_j\cap f^{-1}(Y_j')$ and letting $\iota_{j,j'}:Y_j\cap f^{-1}(Y_{j'})\to Y_j$ be the inclusion and letting $f_{j,j'}: Y_j\cap f^{-1}(Y_{j'})\to Y_{j'}$ be the restriction of the map $f$.

\subsection{Spaces of orbits of finite length}
\label{sub:finite}

We have defined spaces of infinite orbits previously. It is also useful to have at our disposal spaces of orbits of finite length.

\begin{dfn}
\label{dfn:finite}
Let $X^n$ be the space of sequences $(x_1,\ldots, x_n)\in (X^1)^n$ such that $\sigma(x_i)=\iota(x_{i+1})$ for $i=1, \ldots, n-1$.
\end{dfn}

As in the case of the spaces of infinite orbits, $X^{+\infty}$ and $X^{\pm\infty}$, there are shift maps associated to spaces of finite length orbits. Unlike the example of the shift map on $X^\infty$ or $X^{\pm\infty}$ this shift map does not give spaces of orbits of finite length the structure of a standard dynamical system but it does give them the structure of a multivalued dynamical system.

A multivalued dynamical system $\iota, \sigma : X^1\to X^0$ gives rise to a sequence of multivalued dynamical systems given by pairs of maps $\iota, \sigma : X^{n+1}\to X^n$ for $n\geq 0$ where $\iota((x_1, \ldots, x_{n-1}, x_n))=(x_1, \ldots, x_{n-1})$ and $\sigma((x_1, x_2, \ldots, x_n))=(x_2, \ldots, x_n)$. In the context of directed graphs these systems correspond to the  directed graphs for higher block presentations of the original system. Note that these systems $\iota, \sigma : X^{n+1}\to X^n$ produce the same spaces of infinite orbits as the original system $\iota, \sigma : X^1\to X^0$.

In the case of the restriction of a rational map to a neighborhood of the Julia set we can identify the spaces $X^n$ with a decreasing sequence of neighborhoods of the Julia set.

We also remark that the space $X^n$ can be described by means of a universal property. If there is a space $Y$ and a pair of maps $\phi_i : Y\to X^n$ ($i=1, 2$) with $\iota \phi_1=\sigma \phi_2$, then there exists a map $\psi : Y\to X^{n+1}$ so that $\phi_1=\sigma \psi$ and $\phi_2=\iota \psi$. This situation is often described by saying that $X^{n+1}$ is a ``pullback''.

\end{section}

\begin{section}{Homotopy semi-conjugacies and homotopy equivalence}
\label{sec:hsc}

We are interested in extending the notions of semi-conjugacy and conjugacy to multivalued dynamical systems. The notion of semi-conjugacy will form the basis for the notion of homotopy semi-conjugacy.

We begin by recalling the classical notion.

\begin{dfn} 
\label{dfn:semiconj1}
Let $f:X\to X$ and $g:Y\to Y$ be (classical) dynamical systems then a map $h : X\to Y$ is a {\it (classical) semi-conjugacy} if $hf=gh$.
\end{dfn}

\begin{rmk} 
\label{rmk:nonsurj}
We do not require semi-conjugacies to be surjective. 
\end{rmk}

We have a corresponding concept for multivalued dynamical systems. Let $\mathcal{X}=(X^0, X^1; \iota, f)$ and $\mathcal{Y}=(Y^0, Y^1; \iota, g)$ be two multivalued dynamical systems. 

\begin{dfn} 
\label{dfn:semiconj2}
We say that $\mathcal{X}$ is {\it semi-conjugate} to $\mathcal{Y}$ if there are maps $h^0 : X^0\to Y^0$ and $h^1 : X^1\to Y^1$ so that $g h^1=h^0 f$ and $\iota h^1=h^0\iota$ hold. The pair $h=(h^0, h^1)$ is called a semi-conjugacy from $\mathcal{X}$ to $\mathcal{Y}$.
\end{dfn}

The pair of the identity maps ${\rm id}_{X^0} : X^0\to X^0$ and ${\rm id}_{X^1} : X^1\to X^1$ is an example of a semi-conjugacy from a multivalued dynamical system $\mathcal{X}=(X^0, X^1; \iota, f)$ to itself. We call ${\rm id}_{\mathcal{X}}=({\rm id}_{X^0}, {\rm id}_{X^1})$ the {\it identity semi-conjugacy} of $\mathcal{X}$. 

\begin{prp} 
\label{prp:semiconj}
A semi-conjugacy from $\mathcal{X}$ to $\mathcal{Y}$ takes orbits of $\mathcal{X}$ to orbits of $\mathcal{Y}$ and induces a semi-conjugacy (in the classical sense) between $\hat{f} : X^\infty\to X^\infty$ and $\hat{g} : Y^\infty\to Y^\infty$.
\end{prp}

Let $\pi_1:X^n\to X^1$ be the projection which maps a point $(x_1, x_2, \ldots,  x_n)$ to $x_1$.

\begin{prp} 
\label{prp:higher}
The pair of maps $(\pi_1\circ\iota,\iota\circ\pi_1)$ induce a semi-conjugacy from $\iota,f:X^{n+1}\to X^n$ to $\iota,f: X^1\to X^0$.
\end{prp}

This semi-conjugacy induces semi-conjugacy from  $\hat f:X^\infty\to X^\infty$ to  $\hat f:X^\infty\to X^\infty$ which is the identity map.

We are interested in showing that homotopy information allows us to build semi-conjugacies. The following definition captures the homotopy information that we need.

\begin{dfn} 
\label{dfn:hsc}
$\mathcal{X}$ is said to be {\it homotopy semi-conjugate} to $\mathcal{Y}$ if there exist $h^0 : X^0\to Y^0$ and $h^1 : X^1\to Y^1$ so that $h^0f$ is homotopic to $gh^1$ by $G=G_t$ ($G_0=h^0f$ and $G_1=gh^1$) and $h^0\iota$ is homotopic to $\iota h^1$ by $H=H_t$ ($H_0=h^0\iota$ and $H_1=\iota h^1$). We call the quadruple $h=(h^0, h^1; G, H)$ a {\it homotopy semi-conjugacy} from $\mathcal{X}$ to $\mathcal{Y}$.
\end{dfn}

The pair of the identity maps ${\rm id}_{X^0} : X^0\to X^0$ and ${\rm id}_{X^1} : X^1\to X^1$ and a pair of constant homotopies $\iota$ and $f$ becomes a homotopy semi-conjugacy ${\rm id}_{\mathcal{X}}=({\rm id}_{X^0}, {\rm id}_{X^1}; f, \iota)$ from $\mathcal{X}=(X^0, X^1; \iota, f)$ to itself.

\begin{dfn} 
\label{dfn:identity}
We call ${\rm id}_{\mathcal{X}}=({\rm id}_{X^0}, {\rm id}_{X^1}; f, \iota)$ the {\it identity semi-conjugacy} of $\mathcal{X}$.
\end{dfn}

Given a homotopy $I(x)=I_t(x)$ ($0\leq t\leq 1$), we will write $I(x)^{-1}=I_{1-t}(x)$. Let $\cdot$ denote the concatenation of two homotopies or of two paths.

Let $h=(h^0, h^1; G, H)$ and $k=(k^0, k^1; G', H')$ be two homotopy semi-conjugacies from $\mathcal{X}=(X^0, X^1; \iota, f)$ to $\mathcal{Y}=(Y^0, Y^1; \iota, g)$. The next definition will be useful in telling us when two homotopy semi-conjugacies produce the same conjugacy. 

\begin{dfn} 
\label{dfn:homotopic}
$h$ is said to be {\it homotopic} to $k$ if there exist $S=S_t : X^1\to Y^1$ with $S_0=h^1$ and $S_1=k^1$ and $T=T_s : X^0\to Y^0$ with $T_0=h^0$ and $T_1=k^0$ so that (i) $gS(x)\cdot G'(x)^{-1}$ is homotopic to $G(x)^{-1}\cdot Tf(x)$ and (ii) $H(x)\cdot \iota S(x)$ is homotopic to $T\iota (x)\cdot H'(x)$ for each $x\in X^1$. The pair $(T, S)$ is called a {\it homotopy} from $h$ to $k$.
\end{dfn}

Let $h=(h^0, h^1; G, H)$ be a homotopy semi-conjugacy from $\mathcal{X}$ to $\mathcal{Y}$ and let $k=(k^0, k^1; G', H')$ be one from $\mathcal{Y}$ to $\mathcal{Z}$. We define their composition $kh : \mathcal{X}\to \mathcal{Z}$ as 
$$kh\equiv (k^0h^0, k^1h^1; k^0G\cdot G'h^1, k^0H\cdot H'h^1).$$

\begin{dfn} 
\label{dfn:equivalent}
$\mathcal{X}$ and $\mathcal{Y}$ are said to be {\it homotopy equivalent} if there exist homotopy semi-conjugacies $h$ from $\mathcal{X}$ to $\mathcal{Y}$ and $k$ from $\mathcal{Y}$ to $\mathcal{X}$ so that  $kh$ is homotopic to the identity semi-conjugacy ${\rm id}_{\mathcal{X}}$ of $\mathcal{X}$ and $hk$ is homotopic to the identity semi-conjugacy ${\rm id}_{\mathcal{Y}}$ of $\mathcal{Y}$.
\end{dfn}

One can extend the above notions for semi-conjugacies to semi-conjugacies with ``lag''. The first two definitions extend Definitions~\ref{dfn:semiconj2} and~\ref{dfn:identity}.

\begin{dfn} 
\label{dfn:lag}
A {\it semi-conjugacy of lag $n$} from $\iota, f : X^1\to X^0$ to $\iota, g : Y^1\to Y^0$ is a semi-conjugacy from $\iota, f : X^{n+1}\to X^n$ to $\iota, g : Y^1\to Y^0$.
\end{dfn}

\begin{dfn}
The {\it identity semi-conjugacy of lag $n$} is the semi-conjugacy of lag $n$ from $\iota, f : X^1\to X^0$ to itself which is given in Proposition~\ref{prp:higher}. 
\end{dfn}

Given $m>0$ note that a semi-conjugacy from $\iota, f : X^{n+1}\to X^n$ to $\iota, g : Y^1\to Y^0$ induces a natural semi-conjugacy from $\iota, f : X^{m+n+1}\to X^{m+n}$ to $\iota, g : Y^{m+1}\to Y^m$. In particular we can compose semi-conjugacies of lag $n$ and $m$ to get a semi-conjugacy of lag $m+n$.

\begin{dfn} 
\label{dfn:shift-eq}
A {\it shift equivalence} between $\iota, f : X^1\to X^0$ and $\iota, g : Y^1\to Y^0$ is a pair of semi-conjugacies of lag $n$ and $m$ such that the compositions in either directions give the identity semi-conjugacies.
\end{dfn}

The following generalizes Definition~\ref{dfn:hsc}.

\begin{dfn} 
\label{dfn:hsc-lag}
A {\it homotopy semi-conjugacy of lag $n$} from $\iota, f : X^1\to X^0$ to $\iota, g : Y^1\to Y^0$ is a homotopy semi-conjugacy from $\iota, f : X^{n+1}\to X^n$ to $\iota, g : Y^1\to Y^0$.
\end{dfn}

Finally the following generalizes Definition~\ref{dfn:shift-eq}.

\begin{dfn} 
\label{dfn:equiv-lag}
A {\it homotopy shift equivalence} between $\iota, f : X^1\to X^0$ and $\iota, g : Y^1\to Y^0$ is a pair of homotopy semi-conjugacies of lag $n$ and $m$ such that the compositions in either direction are homotopic to the identity semi-conjugacies.
\end{dfn}

\end{section}

\begin{section}{Expanding systems: Definitions, examples and results}
\label{sec:expsyst}

In this section we will define expanding multivalued dynamical systems and state our results in the expanding case. The proofs of these results depend on the notion of homotopy pseudo-orbit which is introduced in Section~\ref{sec:hpo}, the existence and uniqueness of the solution to the shadowing problem in the expanding case established in Section~\ref{sec:exp-shadowing} and the functorial properties of homotopy semi-conjugacies which are established in Section~\ref{sec:functorial}.

Let $X^0$ and $X^1$ be two complete length spaces with metrics $d^0$ and $d^1$ respectively.

\begin{dfn} 
\label{dfn:expsyst}
A multivalued dynamical system $\iota, f : X^1\to X^0$ is said to be {\it expanding} if (i) there exist $\delta>0$ and $\lambda>1$ so that $d^0(f(x), f(y))\geq \lambda d^0(\iota(x), \iota(y))$ whenever $d^1(x, y)<\delta$, and (ii) $f : X^1\to X^0$ is a covering map.
\end{dfn}

We call an expanding multivalued dynamical system an {\em expanding system} for short.

An example of an expanding system is an expanding map on a smooth manifold. Let $M$ be a compact manifold and $f : M\to M$ be an expanding map. Let $X^0\equiv M$ with its given metric and $X^1\equiv M$ with the metric pulled back by $f$. Then $\iota, f : X^1\to X^0$ becomes an expanding system since $f$ is a local isometry and $\iota$ is a contraction. The covering property for $f$ holds automatically in this case.

A second example of an expanding system is a directed graph. Here we choose the metric for which the distance between distinct edges is one and we set $\delta=1/2$.

Another important example of an expanding system is a rational map of $\CC\PP^1$. Let $f : \CC\PP^1 \to \CC\PP^1$ which has the property that any critical point of $f$ is attracted by some attractive cycle. Hence there is a neighborhood $U$ of the attractive cycles of $f$ so that $U$ contains all critical values of $f$ and $f(U)$ is compactly contained in $U$. Then, by letting $X^0\equiv \CC\PP^1 \setminus U$ equipped with the Poincar\'e distance in $X^0$, letting $X^1\equiv f^{-1}(X^0)$ equipped with the Poincar\'e distance in $X^1$ and $\iota : X^1\to X^0$ be the inclusion, we see that $\iota, f : X^1\to X^0$ becomes an expanding system.

The first main result of this paper is precisely stated as

\begin{thm}
\label{thm:exp-equiv}
A homotopy equivalence between expanding systems $\iota, f : X^1\to X^0$ and $\iota, g : Y^1\to Y^0$ induces a topological conjugacy between $\hat f:X^{+\infty}\to X^{+\infty}$ and $\hat g:Y^{+\infty}\to Y^{+\infty}$.
\end{thm}

Though we will not use this notion here we observe that in the hypothesis of this theorem we can replace homotopy equivalence by homotopy shift equivalence (see Definition~\ref{dfn:equiv-lag}).

To prove Theorem~\ref{thm:exp-equiv} we need the following

\begin{thm} 
\label{thm:exp-funct}
A homotopy semi-conjugacy $h$ from a multivalued dynamical system $\iota, f : X^1\to X^0$ to an expanding system $\iota, g : Y^1\to Y^0$ induces a unique semi-conjugacy $h^{\infty}$ from $\hat f:X^{+\infty}\to X^{+\infty}$ to $\hat g:Y^{+\infty}\to Y^{+\infty}$. 
\end{thm}

Thus, we have a natural correspondence $h\mapsto h^{\infty}$. The proof of Theorem~\ref{thm:exp-funct} is given in Subsection~\ref{sub:exp-unique} and the proof of Theorem~\ref{thm:exp-equiv} is given in Section~\ref{sec:functorial}.

We have in mind the following interpretation for this theorem. Let us assume that our expanding systems $\mathcal{X}=(X^0, X^1; \iota, f)$ and $\mathcal{Y}=(Y^0, Y^1; \iota, g)$ have the additional property that $X^0$, $X^1$, $Y^0$ and $Y^1$ have the homotopy type of finite CW complexes. In this case the homotopy semi-conjugacies are determined (up to the relationship of being homotopic) by a finite amount of information. 

Following the paradigm of subshifts of finite type we can use homotopy semi-conjugacies of lag $n$ to define a notion of elementary shift equivalence and shift equivalence between expanding systems. If we apply the previous result to semi-conjugacies of lag $n$ we see that an elementary shift equivalence between expanding systems yields a conjugacy between the corresponding shift spaces.
 
 The semi-conjugacy produced in Theorem \ref{thm:exp-funct} can depend on the particular homotopy we choose in constructing the homotopy semi-conjugacy. This phenomenon occurs in the example of conjugacies between expanding circle maps of degree three and themselves. The expanding circle map of degree three has two fixed points and there are two conjugacies from this map to itself, the identity and one which switches the fixed points. The particular conjugacy which arises depends on which homotopy we use. In many other cases the semi-conjugacy does not in fact depend on the homotopies that we use. 

In many situations such as the case of rational maps the spaces $X^n$ are $K(\pi, 1)$'s so that the homotopy types of maps between them are determined by the maps on fundamental groups. The work of Nekrashevych~\cite{N} builds an elegant theory of expanding maps based on information about the fundamental groups of neighborhoods of Julia sets.

Since the dynamical systems are determined by the homotopy types of the maps we can replace the spaces $X^n$ by simpler spaces which capture their homotopy type. In the case of rational maps we can replace the spaces $X^n$ by one complexes. This gives another approach to the theory of Hubbard trees (in the uniformly expanding case) and can be used to show that the Hubbard tree determines the Julia set.

\end{section}

\begin{section}{Hyperbolic systems: Definitions, examples and results}
\label{sec:hypsyst}

In this section we will define hyperbolic product multivalued dynamical systems and state our results in this case. The proofs of these results depend on the notion of homotopy pseudo-orbit which is introduced in Section ~\ref{sec:hpo}, the existence and uniqueness of the solution to the shadowing problem in the hyperbolic case which are established  in Section~\ref{sec:hyp-shadowing} and the functorial properties of homotopy semi-conjugacies which are established in Section~\ref{sec:functorial}.

\subsection{Definition of hyperbolic systems}
\label{sub:hypsyst}

Let $M^0_x$ and $M^0_y$ be compact connected and orientable smooth manifolds of dimensions $m_x$ and $m_y$ respectively. 
\begin{center}
{\it From here on, we will always assume that $M^0_y$ is simply connected.} 
\end{center}
Write $X^0\equiv M^0_x\times M^0_y$ and let $\pi^0_x : X^0\to M^0_x$ and $\pi^0_y : X^0\to M^0_y$ be projections. Take an open subset $X^1\subset X^0$. Let $f : X^1\to X^0$ be a smooth diffeomorphism onto its image and $\iota : X^1\to X^0$ be the inclusion.

\begin{dfn} 
\label{dfn:crossed}
A multivalued dynamical system $\iota, f : X^1\to X^0$ is called a {\it crossed mapping} if
$$\rho_f\equiv (\pi^0_x \circ f, \pi^0_y \circ \iota) : X^1 \longrightarrow X^0$$
is proper. The {\it degree} of the crossed mapping $\iota, f : X^1\to X^0$ is defined as the degree of the proper map $\rho_f$.
\end{dfn}

It is not difficult to see that $\iota, f : X^1\to X^0$ is a crossed mapping of degree $d$ iff the map $\pi^0_x\circ f : \iota^{-1}(M^0_x(y_0))\to M^0_x$ is proper of degree $d$ for all $y_0\in M^0_y$, where $M^0_x(y_0)\equiv M_x^0\times \{y_0\}\subset X^0$.

Let $|\cdot |_{M^0_x}$ and $|\cdot |_{M^0_y}$ be infinitesimal metrics in the tangent bundles $TM^0_x$ and $TM^0_y$ respectively. For $p\in X^0$ we put
$$C^0_h(p)\equiv \{v=(v_x, v_y)\in T_pX^0 : |v_x|_{M^0_x}>|v_y|_{M^0_y} \}$$
and define $\|v\|^0_h\equiv |v_x|_{M^0_x}$ for $v=(v_x, v_y)\in C^0_h(p)$. Similarly we put 
$$C^0_v(p)\equiv \{v=(v_x, v_y)\in T_pX^0 : |v_x|_{M^0_x}<|v_y|_{M^0_y} \}$$
and define $\|v\|^0_v\equiv |v_y|_{M^0_y}$ for $v=(v_x, v_y)\in C^0_v(p)$.

\begin{dfn} 
\label{dfn:cone}
We call $(\{C^0_h(p)\}_{p\in X^0}, \| \cdot \|^0_h)$ the {\it horizontal cone field} in $X^0$. We also call $(\{C^0_v(p)\}_{p\in X^0}, \| \cdot \|^0_v)$ the {\it vertical cone field} in $X^0$. 
\end{dfn}

A crossed mapping $\iota, f : X^1\to X^0$ is said to {\it expand the horizontal cone field} if there exists $\lambda>1$ so that for any $p\in X^1$, we have 
$$D\iota^{-1}(C^0_h(\iota(p)))\subset Df^{-1}(C^0_h(f(p))) \quad \mathrm{and} \quad \lambda\|D\iota (v)\|^0_h\leq \|Df(v)\|^0_h$$ 
for any $v\in T_pX^1$ with $D\iota(v)\in C^0_h(\iota(p))$. Similarly, a crossed mapping $\iota, f : X^1\to X^0$ is said to {\it contract the vertical cone field} if there exists $\lambda>1$ so that for any $p\in X^1$, we have 
$$Df^{-1}(C^0_v(f(p)))\subset D\iota^{-1}(C^0_v(\iota(p))) \quad \mathrm{and} \quad \lambda\|Df(v)\|^0_v\leq \|D\iota (v)\|^0_v$$ 
for any $v\in T_pX^1$ with $Df(v)\in C^0_v(f(p))$.

\begin{dfn} 
\label{dfn:hypsyst}
A crossed mapping $\iota, f : X^1\to X^0$ of degree $d$ is called a {\it hyperbolic system} of degree $d$ if it expands the horizontal cone field and contracts the vertical cone field.
\end{dfn}

Let $l^0_x$ and $l^0_y$ be the arc lengths induced from the infinitesimal metrics in $M^0_x$ and $M^0_y$ respectively and put $l^0(\gamma)\equiv l^0_x(\pi^0_x(\gamma))+l^0_y(\pi^0_y(\gamma))$ for a path $\gamma$ in $X^0$. Let $d^0_x$ and $d^0_y$ be the induced distances in $M^0_x$ and $M^0_y$ respectively and put $d^0(p, q)\equiv d^0_x(\pi^0_x(p), \pi^0_x(q))+d^0_y(\pi^0_y(p), \pi^0_y(q))$. 

\begin{rmk} 
\label{rmk:bounded}
For $A\subset M^0_y$, let ${\rm diam}_y\, A$ be the diameter of $A$ with respect to the distance $d^0_y$ induced from $|\cdot|^0_y$. Since $M^0_y$ is assumed to be compact, we have 
$$(\ast) \quad C\equiv {\rm diam}_y (\pi^0_y(f(X^1)))<+\infty.$$
In the following discussion, the compactness of $M^0_y$ is not essential but the condition ($\ast$) is. See also Subsection~\ref{sub:example-hyp} for the case of complex H\'enon maps where $M^0_y$ is not compact but the condition ($\ast$) is satisfied.
\end{rmk}

From here on, we use the notation $M^0_y(x)\equiv \{x\} \times M^0_y$ and $M^0_x(y)\equiv M^0_x\times \{y\}$. 

\begin{lmm} 
\label{lmm:covering}
If $\iota, f : X^1\to X^0$ is a hyperbolic system of degree $d$, then $\rho=\rho_f : X^1\to X^0$ is a covering map of degree $d$.
\end{lmm}

\begin{proof}
For $(x', y')\in X^0$, take its simply connected neighborhood $U\subset X^0$. For $(x_0, y_0)\in U$ the inverse image $\rho^{-1}_f(x, y)$ is the intersection of $f^{-1}(M^0_y(x_0))$ and $\iota^{-1}(M^0_x(y_0))$. Since $M^0_y(x_0)$ is contained in the vertical cone field and $M^0_x(y_0)$ is contained in the horizontal cone field at each point, $f^{-1}(M^0_y(x_0))$ intersects transversally with $\iota^{-1}(M^0_x(y_0))$ at $d$ points by the definition of a hyperbolic system. Since a transversal intersection persists by a small perturbation, these distinct $d$ points persist when $(x_0, y_0)$ moves over $U$. This shows that $\rho_f : X^1\to X^0$ is a covering map of degree $d$. 
\end{proof}

Thanks to this lemma and the assumption that $M^0_y$ is simply conneccted, when $\iota, f : X^1\to X^0$ is a hyperbolic system we see that $\rho_f$ naturally induces a product structure in $X^1$. To see this, we first note that $\rho^{-1}_f(M^0_y(x_0))$ consists of $d$ mutually disjoint submanifolds of $X^1$ which are all diffeomorphic to the simply connected submanifold $M^0_y(x_0)$ of $X^0$ for any $x_0\in M^0_x$. It then follows that $X^1$ is diffeomorphic to a space of the form $M^1_x\times M^1_y$, where $M^1_x\equiv X^1\cap \{y=y_0\}$ for any $y_0\in M^0_y$ and $M^1_y\equiv M^0_y$. Thus, one can define projections $\pi_x^1 : X^1\to M^1_x$ and $\pi_y^1 : X^1\to M^1_y$. In these product coordinates for $X^1$ the map $\rho_f : X^1\to X^0$ becomes $\rho_f(x, y)=(\pi^0_x\circ f(x, y_0), y)$. We let 
$$g\equiv \pi_x^0\circ f(\ \cdot \ , y_0) : M^1_x\longrightarrow M^0_x$$
and $\iota' : M^1_y\to M^0_y$ be $\iota'(y)\equiv y$. Then, since $f$ maps any vertical straight disk $M^1_y(x)$ in $M^1$ into certain vertical straight disk $M^0_y(x')$ in $M^0$ in the product coordinates, it follows that $\pi^0_x\circ f(x, y)$ does not depend on $y$. In particular, this is equal to $g(x)=\pi^0_x\circ f(x, y_0)$. Thus, we can write down as $f(x, y)=(g(x), h(x, y))$ for some $h$ in the product coordinates.

Infinitesimal metrics in $M_x^1$ and $M_y^1$ can be defined as 
$$|v_x|_{M^1_x}\equiv |Dg(v_x)|_{M^0_x}$$ 
for $v_x\in T_pM^1_x$ and
$$|v_y|_{M^1_y}\equiv |D\iota'(v_y)|_{M^0_y}$$
for $v_y\in T_pM^1_y$. Then, one can define the horizontal cone field $(\{C^1_h(p)\}_{p\in X^1}, \|\cdot\|^1_h)$ and the vertical cone field $(\{C^1_v(p)\}_{p\in X^1}, \|\cdot\|^1_v)$ in $X^1$ as before. This allows us to define the notions of length of paths and the distance in $X^1$ etc.

\begin{lmm} 
\label{lmm:expansion}
Let $\iota, f : X^1\to X^0$ be a hyperbolic system. Then, we have
$$\|v\|^1_h=\|Df(v)\|^0_h \quad \mathrm{\it and} \quad \|v\|^1_h\geq \lambda \|D\iota(v)\|^0_h.$$
Similarly we have
$$\|v\|^1_v\geq \lambda \|Df(v)\|^0_v \quad \mathrm{\it and} \quad \|v\|^1_v=\|D\iota(v)\|^0_v.$$
\end{lmm}

\begin{proof}
By the definition of the norms, we first have $\|v\|^1_h=|v_x|_{M_x^1}$. In the product coordinates in $X^1=M^1_x\times M^1_y$ the derivative $Df$ has of the form:
$$Df=
\begin{pmatrix}
Dg & 0 \\
\ast & \ast
\end{pmatrix},$$
thus $\|Df(v)\|^0_h=|Dg(v_x)|_{M^0_x}$. By the definition of $|\cdot |_{M^1_x}$ we obtain the first equality $\|v\|^1_h=\|Df(v)\|^0_h$. 

The expansion of the horizontal cone field gives $\lambda \|D\iota(v)\|^0_h\leq \|Df(v)\|^0_h$. This combined with the first equality gives $\|v\|^1_h\geq \lambda \|D\iota(v)\|^0_h$.

The proof for the vertical direction is similar, hence omitted.
\end{proof}

The following notions play the role of ``approximate'' stable and unstable manifolds in the shadowing process.

\begin{dfn} 
\label{dfn:submfd}
For $n=0, 1$, an $m_x$-dimensional (not necessarily conneccted) submanifold $D$ in $X^n$ is said to be {\it horizontal-like of degree $d$} if $\pi^n_x : D\to M^n_x$ is a proper map of degree $d$ and for each point $p\in D$, the tangent space $T_pD$ is contained in $C^n_h(p)$. Similarly the notion of a {\it vertical-like submanifold} is defined.
\end{dfn}

\begin{lmm} 
\label{lmm:submfd}
Let $\iota, f : X^1\to X^0$ be a hyperbolic system of degree $d$. Then, for a horizontal-like submanifold $H$ of degree $k$ in $X^1$, $f(H)$ becomes a horizontal-like submanifold of degree $dk$ in $X^0$. For a vertical-like submanifold $V$ of degree $k$ in $X^1$, $\iota(V)$ becomes a vertical-like submanifold of degree $k$ in $X^0$.
\end{lmm}

\begin{proof}
Since $H$ is a horizontal-like submanifold of degree $k$ in $X^1$, each fiber of $\pi_x^1 : X^1\to M^1_x$ intersects with $H$ at $k$ points. By the formula of $\rho$ we see that $\rho^{-1}(\{x=x_0\})=f^{-1}(\{x=x_0\})$. Thus, $f$ maps exactly $d$ fibers of $\pi_x^1 : X^1\to M^1_x$ into a fiber of $\pi_x^0 : X^0\to M^0_x$. It then follows that $f(H)$ intersects with a fiber of $\pi_x^0 : X^0\to M^0_x$ at $dk$ points. The discussion for a vertical-like submanifold is trivial since $\iota$ is the inclusion.
\end{proof}

\begin{lmm}
\label{lmm:intersection}
Let $H$ be a horizontal-like submanifold of degree $k_h$ and $V$ be a vertical-like submanifold of degree $k_v$ in $X^0$. Then their intersection $H\cap V$ consists of $k_hk_v$ points. Moreover, they are all transverse intersections.
\end{lmm}

\begin{proof}
Since $H$ is horizontal-like and $V$ is vertical-like, the intersection number is always positive for each point in $H\cap V$. Thus, the conclusion  follows from a standard homological argument.
\end{proof}

Combining the previous two lemmas, we have

\begin{cor} 
\label{cor:intersection}
Let $\iota, f : X^1\to X^0$ be a hyperbolic system of degree $d$. Then, for a horizontal-like submanifold $H$ of degree one and a vertical-like submanifold $V$ of degree one in $X^1$, $f(H)\cap \iota(V)$ consists of $d$ points. Moreover, they are all transverse intersections.
\end{cor}

Next we lift up the structure of hyperbolic system of $\iota, f : X^1\to X^0$ to higher levels $\iota, f : X^{n+1}\to X^n$.

\begin{prp} 
\label{prp:buildup}
Let $\iota, f : X^1\to X^0$ be a hyperbolic system. Then, $\iota, f : X^2\to X^1$ becomes a hyperbolic system. 
\end{prp}

In the following proof, the commutativity $\iota f=f\iota$ is essentially used.

\begin{proof}
Take $(x_0, y_0)\in X^1$. Since $\iota$ and $f$ are injective,
\begin{align*}
{\rm card}(\rho^{-1}_f(x_0, y_0))
= & {\rm card}(f^{-1}(M^1_y(x_0))\cap \iota^{-1}(M^1_x(y_0))) \\
= & {\rm card}(\iota(M^1_y(x_0))\cap f(M^1_x(y_0))) \\
= & d 
\end{align*}
by Corollary~\ref{cor:intersection}. Since the intersections $\iota(M^1_y(x_0))\cap f(M^1_x(y_0))$ are transverse, this number is stable under a small perturbation of $(x_0, y_0)\in X^1$. This shows that $\rho_f : X^2\to X^1$ is a covering. 

We put
$$C^m_h(p, \alpha)\equiv \{ v\in T_pX^m : \|v\|^m_v\leq \alpha \|v\|^m_h \}$$ 
and
$$C^m_v(p, \alpha)\equiv \{ v\in T_pX^m : \|v\|^m_h\leq \alpha \|v\|^m_v \}$$ 
for $p\in X^m$ ($m=0, 1$) and $\alpha>0$. In particular, we have $C^m_h(p)=C^m_h(p, 1)$ and $C^m_v(p)=C^m_v(p, 1)$. Note also that $C^1_h(p, 1)\cap C^1_v(p, \alpha)=\emptyset$ and $C^1_h(p, \alpha)\cap C^1_v(p, 1)=\emptyset$ hold for $\alpha<1$.

To finish the proof of Proposition~\ref{prp:buildup}, it is enough to show

\vspace{0.3cm}

\noindent
{\bf Claim.}
{\it Assume that $\iota, f : X^1\to X^0$ is a hyperbolic system. Take $(p_1, p_2)\in X^2$ and $(v_1, v_2)\in T_pX^2$. 
\begin{enumerate}
\renewcommand{\labelenumi}{(\roman{enumi})}
\item If $v_1\in C^1_h(p_1, 1)$, then $v_2\in C^1_h(p_2, \lambda^{-1})$ and $\|v_2\|^1_h\geq \lambda \|v_1\|^1_h$ hold.
\item If $v_2\in C^1_v(p_2, 1)$, then $v_1\in C^1_v(p_1, \lambda^{-1})$ and $\|v_1\|^1_v\geq \lambda \|v_2\|^1_v$ hold.
\end{enumerate}
}

\vspace{0.3cm}

Note that $f(p_1)=\iota(p_2)$ and $Df(v_1)=D\iota(v_2)$ hold. 

\vspace{0.3cm}

\noindent
{\it Proof of Claim.} We first show
\begin{enumerate}
\renewcommand{\labelenumi}{(\alph{enumi})}
\item $D\iota(v)\in C^0_h(\iota(p), 1)$ implies $v\in C^1_h(p, \lambda^{-1})$, and 
\item $Df(v)\in C^0_v(f(p), 1)$ implies $v\in C^1_v(p, \lambda^{-1})$
\end{enumerate}
for $p\in X^1$ and $v\in T_p X^1$. Indeed, suppose that $D\iota(v)\in C^0_h(\iota(p), 1)$ holds. We then have the following estimate:
$$\|D\iota(v)\|^0_v\leq \|D\iota(v)\|^0_h\leq \lambda^{-1} \|Df(v)\|^0_h$$
by the definition of the expansion/contraction of the cone fields and the definition of $C^0_h(\iota(p), 1)$. This together with Lemma~\ref{lmm:expansion} implies $\|v\|^1_v\leq \lambda^{-1}\|v\|^1_h$, hence $v\in C^1_h(p, \lambda^{-1})$. The proof of (b) is similar.

Now let us prove Claim (i). Assume $v_1\in C^1_h(p_1, 1)$. Since $v_1\notin C^1_v(p_1, \lambda^{-1})$, we have $D\iota(v_2)=Df(v_1)\notin C^0_v(f(p_1), 1)$ by (b) above. In particular, $D\iota(v_2)\in C^0_h(f(p_1), 1)=C^0_h(\iota(p_2), 1)$. Then, (a) above implies $v_2\in C^1_h(p_2, \lambda^{-1}$). Since $D\iota(v_2)\in C^0_h(\iota(p_2), 1)$, we have
$$\|v_2\|^1_h=\|Df(v_2)\|^0_h\geq\lambda \|D\iota(v_2)\|^0_h=\lambda\|Df(v_1)\|^0_h=\lambda\|v_1\|^1_h$$
by the estimate above and Lemma~\ref{lmm:expansion}. The proof of (ii) is similar. Thus, we have shown Claim and hence Proposition~\ref{prp:buildup}.
\end{proof}

By applying this proposition repeatedly, one can induce a product structure $X^n=M^n_x\times M^n_y$, define the projections $\pi_x^n : X^n\to M^n_x$ and $\pi_y^n : X^n\to M^n_y$, the infinitesimal metrics $|\cdot |^n_x$ in $M_x^n$ and $|\cdot |^n_y$ in $M_y^n$, the horizontal cone field $(\{C^n_h(p)\}_{p\in X^n}, \|\cdot\|^n_h)$ and the vertical cone field $(\{C^n_v(p)\}_{p\in X^n}, \|\cdot\|^n_v)$ in $X^n$, the notions of horizontal-like submanifolds and vertical-like submanifolds in $X^n$, lengths of paths $l^n_x$ and $l^n_y$, and distances $d^n_x$ and $d^n_y$ in $X^n$ etc for all $n\geq 0$. Thus, in summary,

\begin{prp} 
\label{prp:induction}
Let $\iota, f : X^1\to X^0$ be a hyperbolic system. Then, the conclusions in  Lemmas~\ref{lmm:covering} to \ref{lmm:intersection}, Corollary~\ref{cor:intersection} and Proposition~\ref{prp:buildup} above hold for the setting of the multivalued dynamical system $\iota, f : X^{n+1}\to X^n$ for all $n\geq 0$.
\end{prp}

In particular, the following claim will be quite useful in the sequel.

\begin{cor} 
\label{cor:expansion}
Let $\iota, f : X^1\to X^0$ be a hyperbolic system. Then, for $n\geq 0$ a path $u$ in a horizontal-like submanifold in $X^{n+1}$ satisfies $\lambda \cdot l^n_x(\iota u)\leq l^n_x(fu)$. Similarly a path $s$ in a vertical-like submanifold in $X^{n+1}$ satisfies $\lambda \cdot l^n_y(f s)\leq l^n_y(\iota s)$.
\end{cor}

\begin{proof}
This immediately follows from Lemma~\ref{lmm:expansion}.
\end{proof}

\subsection{Examples of hyperbolic systems}
\label{sub:example-hyp}

One of the most important examples of a hyperbolic system is a certain polynomial diffeomorphism of $\CC^2$ such as a complex H\'enon map~\cite{I1}. 

Let $M^0_x$ and $M^0_y$ be connected, bounded, open subsets of $\CC$. As before, we moreover assume that $M^0_y$ is simply connected. Let $f : \CC^2\to\CC^2$ be a polynomial diffeomorphism of $\CC^2$. We put $X^0\equiv M^0_x\times M^0_y$ and $X^1\equiv X^0\cap f^{-1}(X^0)$, and we let $\iota : X^1\to X^0$ be the inclusion. This defines a multivalued dynamical system $\iota, f : X^1\to X^0$.

Let $|\cdot |_{M^0_x}$ and $|\cdot |_{M^0_y}$ be Poincar\'e metrics in $M^0_x$ and $M^0_y$ respectively. Define a cone field in terms of the ``slope'' with respect to the Poincar\'e metrics in $M^0_x$ and $M^0_y$ as
$$C^h_p\equiv \bigl\{v=(v_x, v_y) \in T_p X^0 : |v_x|_{M^0_x}>|v_y|_{M^0_y} \bigr\}.$$
A metric in this cone is given by $\|v\|_h\equiv |D\pi^0_x (v)|_{M^0_x}$. Similarly we put
$$C^v_p\equiv \bigl\{v=(v_x, v_y) \in T_p X^0 : |v_x|_{M^0_x}<|v_y|_{M^0_y} \bigr\}.$$
A metric in this cone is given by $\|v\|_v\equiv |D\pi^0_y (v)|_{M^0_y}$.

\begin{dfn}
We call $(\{C^h_p\}_{p\in X^0}, \| \cdot \|_h)$ the {\it horizontal Poincar\'e cone field}. We call $(\{C^v_p\}_{p\in X^0}, \| \cdot \|_v)$ the {\it vertical Poincar\'e cone field}. 
\end{dfn}

\begin{rmk} 
\label{rmk:ok-Henon}
In this setting $M^0_x$ and $M^0_y$ are not compact. However, the condition ($\ast$) holds when $\iota, f : X^1\to X^0$ is a crossed mapping.
\end{rmk}

Let $\mathcal{F}_h=\{M^0_x(y)\}_{y\in M^0_y}$ be the horizontal foliation of $X^0$ with leaves $M^0_x(y)$ ($y\in M^0_y$), and let $\mathcal{F}_v=\{M^0_y(x)\}_{x\in M^0_x}$ be the vertical foliation of $X^0$ with leaves $M^0_y(x)$ ($x\in M^0_x$).

\begin{dfn} 
\label{dfn:ntc}
We say that $\iota, f : X^1\to X^0$ satisfies the {\it no-tangency condition (NTC)} if ${\iota}^{-1}(\mathcal{F}_h)$ and $f^{-1}(\mathcal{F}_v)$ have no tangencies in $X^1$. 
\end{dfn}

The following statement has been proved in~\cite{I1}.

\begin{thm} 
\label{thm:criterion}
Let $f$ be a polynomial diffeomorphism of $\CC^2$ and assume that $\iota, f : X^1\to X^0$ is a crossed mapping of degree $d\geq 2$. Then, $\iota, f : X^1\to X^0$ is a hyperbolic system if and only if it satisfies the (NTC).
\end{thm}

There is in fact a checkable criterion for a multivalued dynamical system to be hyperbolic. To do this, given two open subsets $V$ and $W$ of $\CC$ let us write $\partial_v(V\times W)=\partial V\times W$ and $\partial_h(V\times W)=V\times \partial W$. Let ${\rm dist} (A, B)$ be the Euclidean distance between two sets $A$ and $B$ in $\CC$. Let $f : \CC^2\to \CC^2$ be a polynomial diffeomorphism of $\CC^2$. 

\begin{dfn} 
\label{dfn:bcc}
$f$ is said to satisfy the {\it boundary compatibility condition (BCC)} if 
\begin{enumerate}
 \renewcommand{\labelenumi}{(\roman{enumi})}
  \item ${\rm dist} (\pi^0_x\circ f(\partial_vX^0), M^0_x)>0$ and
  \item ${\rm dist} (\pi^0_y\circ f^{-1} (\partial_hX^0), M^0_y)>0$
\end{enumerate}
hold. 
\end{dfn}

Let us define
\begin{equation*}
\mathcal{C}_f\equiv \bigcup_{y\in M^0_y}\left\{{\rm critical \ points \ of \ } \pi^0_x\circ f : M^0_x \times\{y\} \to \CC \right\}
\end{equation*}
and call it the {\it dynamical critical set} of $f$.

\begin{dfn} 
\label{dfn:occ}
We say that $f$ satisfies the {\it off-criticality condition (OCC)} if
$${\rm dist} (\pi^0_x\circ f(\mathcal{C}_f), M^0_x)>0$$
holds.
\end{dfn}
 
It is easy to see that the (BCC) implies that $\iota, f : X^1\to X^0$ is a crossed mapping and the (OCC) implies the (NTC). Thus, 

\begin{cor} 
\label{cor:criterion}
If a polynomial diffeomorphism $f : \CC^2\to \CC^2$ satisfies the (BCC) and the (OCC), then $\iota, f : X^1\to X^0$ is a hyperbolic system.
\end{cor}

\subsection{Statement of main result}
\label{sub:result-hyp}

The second main result of this paper is precisely stated as

\begin{thm}
\label{thm:hyp-equiv}
A homotopy equivalence between hyperbolic systems $\iota, f : X^1\to X^0$ and $\iota, g : Y^1\to Y^0$ induces a topological conjugacy between $\hat f:X^{\pm\infty}\to X^{\pm\infty}$ and $\hat g:Y^{\pm\infty}\to Y^{\pm\infty}$.
\end{thm}

Though we will not use this notion here we observe that in the hypothesis of this theorem we can replace homotopy equivalence by homotopy shift equivalence (see Definition~\ref{dfn:equiv-lag}).

To prove Theorem~\ref{thm:hyp-equiv} we need the following 

\begin{thm} 
\label{thm:hyp-funct}
A homotopy semi-conjugacy $h$ from a multivalued dynamical system $\iota, f : X^1\to X^0$ to a hyperbolic system $\iota, g : Y^1\to Y^0$ induces a unique semi-conjugacy $h^{\infty}$ from $\hat f:X^{\pm\infty}\to X^{\pm\infty}$ to $\hat g:Y^{\pm\infty}\to Y^{\pm\infty}$. 
\end{thm}

Thus, we have a natural correspondence $h\mapsto h^{\infty}$. The proof of Theorem~\ref{thm:hyp-funct} is given in Subsection~\ref{sub:hyp-unique} and the proof of Theorem~\ref{thm:hyp-equiv} is given in Section~\ref{sec:functorial}.

\end{section}

\begin{section}{Homotopy pseudo-orbits and homotopies between them}
\label{sec:hpo}

A pseudo-orbit of a dynamical system $f : X\to X$ is a sequence of points $(x_i)$ in $X$ so that $f(x_i)$ and $x_{i+1}$ are ``close''. The shadowing lemma says that for hyperbolic maps pseudo-orbits are close to actual orbits. Our hypotheses do not give closeness information instead they give homotopy information. We introduce homotopy pseudo-orbits to capture this information.

\begin{figure}
\setlength{\unitlength}{1mm}
\begin{picture}(100, 100)(25, 0)

\put(0, 0){\line(1, 0){30}}
\put(30, 0){\line(0, 1){30}}
\put(0, 0){\line(0, 1){30}}
\put(0, 30){\line(1, 0){30}}

\put(60, 0){\line(1, 0){30}}
\put(90, 0){\line(0, 1){30}}
\put(60, 0){\line(0, 1){30}}
\put(60, 30){\line(1, 0){30}}

\put(120, 0){\line(1, 0){30}}
\put(150, 0){\line(0, 1){30}}
\put(120, 0){\line(0, 1){30}}
\put(120, 30){\line(1, 0){30}}

\put(30, 50){\line(1, 0){30}}
\put(60, 50){\line(0, 1){30}}
\put(30, 50){\line(0, 1){30}}
\put(30, 80){\line(1, 0){30}}

\put(90, 50){\line(1, 0){30}}
\put(120, 50){\line(0, 1){30}}
\put(90, 50){\line(0, 1){30}}
\put(90, 80){\line(1, 0){30}}

\qbezier(0, 37)(2, 33)(7, 23)
\put(5.5, 26){\vector(1, -2){1}}
\qbezier(47, 63)(35, 35)(23, 7)
\put(24.5, 10){\vector(-1, -2){1}}
\qbezier(47, 63)(57, 43)(67, 23)
\put(65.5, 26){\vector(1, -2){1}}
\qbezier(107, 63)(95, 35)(83, 7)
\put(84.5, 10){\vector(-1, -2){1}}
\qbezier(107, 63)(117, 43)(127, 23)
\put(125.5, 26){\vector(1, -2){1}}
\qbezier(152, 28)(149, 21)(143, 7)
\put(144.5, 10){\vector(-1, -2){1}}

\qbezier(8, 22)(10, 22)(10, 20)
\qbezier(10, 20)(10, 18)(12, 18)
\qbezier(12, 18)(14, 18)(14, 16)
\qbezier(14, 16)(14, 14)(16, 14)
\qbezier(16, 14)(18, 14)(18, 12)
\qbezier(18, 12)(18, 10)(20, 10)
\qbezier(20, 10)(22, 10)(22, 8)
\qbezier(68, 22)(70, 22)(70, 20)
\qbezier(70, 20)(70, 18)(72, 18)
\qbezier(72, 18)(74, 18)(74, 16)
\qbezier(74, 16)(74, 14)(76, 14)
\qbezier(76, 14)(78, 14)(78, 12)
\qbezier(78, 12)(78, 10)(80, 10)
\qbezier(80, 10)(82, 10)(82, 8)
\qbezier(128, 22)(130, 22)(130, 20)
\qbezier(130, 20)(130, 18)(132, 18)
\qbezier(132, 18)(134, 18)(134, 16)
\qbezier(134, 16)(134, 14)(136, 14)
\qbezier(136, 14)(138, 14)(138, 12)
\qbezier(138, 12)(138, 10)(140, 10)
\qbezier(140, 10)(142, 10)(142, 8)

\put(14.2, 15){\vector(1, -1){1}}
\put(74.2, 15){\vector(1, -1){1}}
\put(134.2, 15){\vector(1, -1){1}}

\put(7.2, 22.8){\circle*{1.5}}
\put(22.5, 7){\circle*{1.5}}
\put(67.2, 22.8){\circle*{1.5}}
\put(82.5, 7){\circle*{1.5}}
\put(127.2, 22.8){\circle*{1.5}}
\put(142.5, 7){\circle*{1.5}}
\put(47, 63){\circle*{1.5}}
\put(107, 63){\circle*{1.5}}

\small
\put(45, 66){$x_{i-1}$}
\put(106, 66){$x_i$}
\put(34, 40){$\iota$}
\put(60, 40){$\sigma$}
\put(94, 40){$\iota$}
\put(120, 40){$\sigma$}
\put(16, 16){$\alpha_{i-1}$}
\put(76, 16){$\alpha_{i}$}
\put(136, 16){$\alpha_{i+1}$}
\put(14, 33){$X^0$}
\put(74, 33){$X^0$}
\put(134, 33){$X^0$}
\put(44, 83){$X^1$}
\put(104, 83){$X^1$}
\put(18, 2){$\iota(x_{i-1})=\alpha_{i-1}(1)$}
\put(43, 17.5){$\sigma(x_{i-1})=\alpha_i(0)$}
\put(82, 2){$\iota(x_i)=\alpha_i(1)$}
\put(103.8, 17.5){$\sigma(x_i)=\alpha_{i+1}(0)$}

\end{picture}
\vspace{1cm}
\begin{center}
Figure 3a. A homotopy pseudo--orbit $(x, \alpha)$ : elements $x_i$ of $X^1$ are drawn as points.
\end{center}
\end{figure} 

 \begin{figure}
\setlength{\unitlength}{1mm}
\begin{picture}(100, 50)(25, 0)

\qbezier(8, 22)(10, 22)(10, 20)
\qbezier(10, 20)(10, 18)(12, 18)
\qbezier(12, 18)(14, 18)(14, 16)
\qbezier(14, 16)(14, 14)(16, 14)
\qbezier(16, 14)(18, 14)(18, 12)
\qbezier(18, 12)(18, 10)(20, 10)
\qbezier(20, 10)(22, 10)(22, 8)
\qbezier(68, 22)(70, 22)(70, 20)
\qbezier(70, 20)(70, 18)(72, 18)
\qbezier(72, 18)(74, 18)(74, 16)
\qbezier(74, 16)(74, 14)(76, 14)
\qbezier(76, 14)(78, 14)(78, 12)
\qbezier(78, 12)(78, 10)(80, 10)
\qbezier(80, 10)(82, 10)(82, 8)
\qbezier(128, 22)(130, 22)(130, 20)
\qbezier(130, 20)(130, 18)(132, 18)
\qbezier(132, 18)(134, 18)(134, 16)
\qbezier(134, 16)(134, 14)(136, 14)
\qbezier(136, 14)(138, 14)(138, 12)
\qbezier(138, 12)(138, 10)(140, 10)
\qbezier(140, 10)(142, 10)(142, 8)
\put(14.2, 15){\vector(1, -1){1}}
\put(74.2, 15){\vector(1, -1){1}}
\put(134.2, 15){\vector(1, -1){1}}

\qbezier(0, 28)(3, 26.5)(6, 24)
\qbezier(24, 8)(48, 43)(67, 23)
\qbezier(84, 8)(108, 43)(127, 23)
\qbezier(144, 8)(147, 12)(150, 15)
\put(5, 25){\vector(3, -2){1}}
\put(65.5, 24.5){\vector(3, -2){1}}
\put(125.5, 24.5){\vector(3, -2){1}}

\put(7, 23){\circle*{1.5}}
\put(23, 7){\circle*{1.5}}
\put(67, 23){\circle*{1.5}}
\put(83, 7){\circle*{1.5}}
\put(127, 23){\circle*{1.5}}
\put(143, 7){\circle*{1.5}}

\small
\put(48, 32){$x_{i-1}$}
\put(111, 32){$x_i$}
\put(16, 16){$\alpha_{i-1}$}
\put(76, 16){$\alpha_{i}$}
\put(136, 16){$\alpha_{i+1}$}
\put(16, 2){$\iota(x_{i-1})=\alpha_{i-1}(1)$}
\put(42, 18){$\sigma(x_{i-1})=\alpha_i(0)$}
\put(78, 2){$\iota(x_i)=\alpha_i(1)$}
\put(102, 18){$\sigma(x_i)=\alpha_{i+1}(0)$}

\end{picture}
\vspace{1cm}
\begin{center}
Figure 3b. A homotopy pseudo--orbit $(x, \alpha)$ : elements $x_i$ of $X^1$ are drawn as arrows.
\end{center}
\end{figure}

Assume that $X^0$ and $X^1$ are equipped with metrics and let $l(\alpha_i)$ denote the length of the path $\alpha_i$ in $X^0$. Let $\iota, \sigma :X^1\to X^0$ be a multivalued dynamical system.

\begin{dfn} 
\label{dfn:hpo}
A {\it homotopy pseudo-orbit} $(x, \alpha)$ is a sequence $x=(x_i)$ of points $x_i\in X^1$ together with a sequence $\alpha=(\alpha_i)$ of paths $\alpha_i : [0, 1]\to X^0$ so that $\alpha_i(0)=\sigma(x_{i-1})$, $\alpha_i(1)=\iota(x_i)$ and $l(\alpha_i)\leq C$ for some $C\geq 0$ independent of $i$. 
\end{dfn}

See Figures 3a and 3b. Here, we allow the index $i$ to take values in either $\NN\cup\{0\}$ (or $\NN$) or $\ZZ$ and we obtain one-sided or bi-infinite homotopy pseudo-orbits respectively. Note that when we speak of a one-sided homotopy pseudo-orbit $(x, \alpha)=((x_i), (\alpha_i))$, the index for the points $x_i$ starts at $i=0$ and the index for the paths $\alpha_i$ starts at $i=1$.

When $\alpha$ consists of constant homotopies, then the homotopy pseudo-orbit $(x, \alpha)$ becomes an orbit (see Figures 2a and 2b again). In this case the sequence of homotopies $\alpha$ may be omitted from the notation $(x, \alpha)$ and we may simply write $x$ if there is no risk of confusion.

\begin{figure}
\setlength{\unitlength}{1mm}
\begin{picture}(100, 100)(25, 0)

\qbezier(6, 71)(30, 95)(54, 71)
\qbezier(56, 71)(80, 95)(104, 71)
\qbezier(135, 83)(118, 83)(106, 71)
\put(55, 70){\circle*{1.5}}
\put(105, 70){\circle*{1.5}}
\put(53, 72){\vector(1, -1){1}}
\put(103, 72){\vector(1, -1){1}}
\put(134, 83){\vector(1, 0){1}}

\qbezier(44, 22)(46, 22)(46, 20)
\qbezier(46, 20)(46, 18)(48, 18)
\qbezier(48, 18)(50, 18)(50, 16)
\qbezier(50, 16)(50, 14)(52, 14)
\qbezier(52, 14)(54, 14)(54, 12)
\qbezier(54, 12)(54, 10)(56, 10)
\qbezier(56, 10)(58, 10)(58, 8)
\qbezier(104, 22)(106, 22)(106, 20)
\qbezier(106, 20)(106, 18)(108, 18)
\qbezier(108, 18)(110, 18)(110, 16)
\qbezier(110, 16)(110, 14)(112, 14)
\qbezier(112, 14)(114, 14)(114, 12)
\qbezier(114, 12)(114, 10)(116, 10)
\qbezier(116, 10)(118, 10)(118, 8)
\put(46.2, 19){\vector(1, -1){1}}
\put(54.2, 11){\vector(1, -1){1}}
\put(106.2, 19){\vector(1, -1){1}}
\put(114.2, 11){\vector(1, -1){1}}
\put(50.3, 15){\circle*{1.5}}
\put(110.3, 15){\circle*{1.5}}

\qbezier(0, 8)(19, 23)(43, 23)
\qbezier(60, 8)(79, 23)(103, 23)
\qbezier(120, 8)(131, 17)(146, 20)
\put(41.3, 23){\vector(1, 0){1}}
\put(101.3, 23){\vector(1, 0){1}}
\put(145, 19.8){\vector(3, 1){1}}

\put(43, 23){\circle*{1.5}}
\put(59, 7){\circle*{1.5}}
\put(103, 23){\circle*{1.5}}
\put(119, 7){\circle*{1.5}}

\small
\put(26, 87){$x_{i-1}$}
\put(78, 87){$x_i$}
\put(128, 87){$x_{i+1}$}
\put(42, 65){$\sigma(x_{i-1})=\iota(x_i)$}
\put(94, 65){$\sigma(x_i)=\iota(x_{i+1})$}
\put(16, 13){$h^1(x_{i-1})$}
\put(76, 13){$h^1(x_i)$}
\put(134, 12){$h^1(x_{i+1})$}
\put(48, 20){$G(x_{i-1})^{-1}$}
\put(44.5, 7){$H(x_i)$}
\put(108, 20){$G(x_i)^{-1}$}
\put(100, 7){$H(x_{i+1})$}

\put(75, 55){\vector(0, -10){25}}
\put(77, 42){$h=(h^0, h^1)$}

\end{picture}
\vspace{1cm}
\begin{center}
Figure 4a. An orbit $x$ and a homotopy pseudo-orbit $h(x)$.
\end{center}
\end{figure}

 \begin{figure}
\setlength{\unitlength}{1mm}
\begin{picture}(100, 110)(25, 0)

\qbezier(38, 92)(40, 92)(40, 90)
\qbezier(40, 90)(40, 88)(42, 88)
\qbezier(42, 88)(44, 88)(44, 86)
\qbezier(44, 86)(44, 84)(46, 84)
\qbezier(46, 84)(48, 84)(48, 82)
\qbezier(48, 82)(48, 80)(50, 80)
\qbezier(50, 80)(52, 80)(52, 78)
\qbezier(98, 92)(100, 92)(100, 90)
\qbezier(100, 90)(100, 88)(102, 88)
\qbezier(102, 88)(104, 88)(104, 86)
\qbezier(104, 86)(104, 84)(106, 84)
\qbezier(106, 84)(108, 84)(108, 82)
\qbezier(108, 82)(108, 80)(110, 80)
\qbezier(110, 80)(112, 80)(112, 78)
\put(44.2, 85){\vector(1, -1){1}}
\put(104.2, 85){\vector(1, -1){1}}

\qbezier(10, 86)(20, 93)(37, 93)
\qbezier(54, 78)(73, 93)(97, 93)
\qbezier(114, 78)(125, 87)(140, 90)
\put(35.3, 93){\vector(1, 0){1}}
\put(95.3, 93){\vector(1, 0){1}}
\put(139, 89.8){\vector(3, 1){1}}

\put(37, 93){\circle*{1.5}}
\put(53, 77){\circle*{1.5}}
\put(97, 93){\circle*{1.5}}
\put(113, 77){\circle*{1.5}}

\qbezier(32, 33)(32, 31)(34, 31)
\qbezier(34, 31)(36, 31)(36, 29)
\qbezier(36, 29)(36, 27)(38, 27)
\qbezier(38, 27)(40, 27)(40, 25)
\qbezier(40, 25)(40, 23)(42, 23)
\qbezier(42, 23)(44, 23)(44, 21)
\qbezier(44, 21)(44, 19)(46, 19)
\qbezier(46, 19)(48, 19)(48, 17)
\qbezier(48, 17)(48, 15)(50, 15)
\qbezier(50, 15)(52, 15)(52, 13)
\qbezier(52, 13)(52, 11)(54, 11)
\qbezier(54, 11)(56, 11)(56, 9)
\qbezier(56, 9)(56, 7)(58, 7)

\qbezier(92, 33)(92, 31)(94, 31)
\qbezier(94, 31)(96, 31)(96, 29)
\qbezier(96, 29)(96, 27)(98, 27)
\qbezier(98, 27)(100, 27)(100, 25)
\qbezier(100, 25)(100, 23)(102, 23)
\qbezier(102, 23)(104, 23)(104, 21)
\qbezier(104, 21)(104, 19)(106, 19)
\qbezier(106, 19)(108, 19)(108, 17)
\qbezier(108, 17)(108, 15)(110, 15)
\qbezier(110, 15)(112, 15)(112, 13)
\qbezier(112, 13)(112, 11)(114, 11)
\qbezier(114, 11)(116, 11)(116, 9)
\qbezier(116, 9)(116, 7)(118, 7)
\put(44.2, 20){\vector(1, -1){1}}
\put(104.2, 20){\vector(1, -1){1}}
\put(36, 29){\vector(0, -1){1}}
\put(54, 11){\vector(1, 0){1}}
\put(96, 29){\vector(0, -1){1}}
\put(114, 11){\vector(1, 0){1}}

\qbezier(10, 19)(20, 30)(32, 33)
\qbezier(58, 7)(73, 28)(92, 33)
\qbezier(118, 7)(125, 18)(140, 26)
\put(30, 32.5){\vector(3, 1){1}}
\put(90, 32.5){\vector(3, 1){1}}
\put(139, 25.5){\vector(2, 1){1}}

\put(32, 33){\circle*{1.5}}
\put(40, 25){\circle*{1.5}}
\put(50, 15){\circle*{1.5}}
\put(58, 7){\circle*{1.5}}
\put(92, 33){\circle*{1.5}}
\put(118, 7){\circle*{1.5}}
\put(100, 25){\circle*{1.5}}
\put(110, 15){\circle*{1.5}}

\small
\put(46, 86){$\alpha_{i}$}
\put(106, 86){$\alpha_{i+1}$}
\put(18, 86){$x_{i-1}$}
\put(73, 85){$x_{i}$}
\put(128, 83){$x_{i+1}$}
\put(24, 97){$\sigma(x_{i-1})=\alpha_i(0)$}
\put(43, 71){$\iota(x_i)=\alpha_i(1)$}
\put(86, 97){$\sigma(x_i)=\alpha_{i+1}(0)$}
\put(99, 71){$\iota(x_{i+1})=\alpha_{i+1}(1)$}
\put(16, 20){$h^1(x_{i-1})$}
\put(74, 20){$h^1(x_i)$}
\put(132, 17){$h^1(x_{i+1})$}
\put(38, 29){$G(x_{i-1})^{-1}$}
\put(43, 7){$H(x_i)$}
\put(98, 29){$G(x_i)^{-1}$}
\put(100, 7){$H(x_{i+1})$}
\put(46, 21){$h^0(\alpha_i)$}
\put(106, 21){$h^0(\alpha_i)$}

\put(75, 65){\vector(0, -10){25}}
\put(77, 52){$h=(h^0, h^1)$}

\end{picture}
\vspace{1cm}
\begin{center}
Figure 4b. A homotopy pseudo-orbit $(x, \alpha)$ and a homotopy pseudo-orbit $h(x, \alpha)$.
\end{center}
\end{figure}

Next we see how a (homotopy pseudo-)orbit is mapped by a homotopy semi-conjugacy.

\begin{lmm} 
\label{lmm:orbhpo}
A homotopy semi-conjugacy $h=(h^0, h^1; G, H)$ from $\mathcal{X}$ to $\mathcal{Y}$ takes an orbit $x\in X^{\infty}$ of $\mathcal{X}$ to a homotopy pseudo-orbit $h(x)\equiv(h^1(x), G(x)^{-1}\cdot H(x))$ of $\mathcal{Y}$.
\end{lmm}

Here, $G(x)^{-1}\cdot H(x)\equiv (G(x_{i-1})^{-1}\cdot H(x_i))$ is a sequence of homotopies. See Figure 4a. More generally, by writing a sequence of homotopies as $G(x)^{-1}\cdot h^0(\alpha)\cdot H(x)\equiv (G(x_{i-1})^{-1}\cdot h^0(\alpha_i)\cdot H(x_i))$ for a homotopy pseudo-orbit $(x, \alpha)$ of $\mathcal{X}$ and a homotopy semi-conjugacy $h=(h^0, h^1; G, H)$ from $\mathcal{X}$ to $\mathcal{Y}$ we can show (see Figure 4b)

\begin{lmm} 
\label{lmm:hpohpo}
A homotopy semi-conjugacy $h=(h^0, h^1; G, H)$ from $\mathcal{X}=(X^0, X^1; \iota, f)$ to $\mathcal{Y}=(Y^0, Y^1; \iota, g)$ takes a homotopy pseudo-orbit $(x, \alpha)$ of $\mathcal{X}$ to another homotopy pseudo-orbit $h(x, \alpha)\equiv (h^1(x), G(x)^{-1}\cdot h^0(\alpha)\cdot H(x))$ of $\mathcal{Y}$.
\end{lmm}

\begin{proof}
Since $G_0(x_{i-1})=h^0f(x_{i-1})=h^0\alpha_i(0)$ and $H_0(x_i)=h^0\iota(x_i)=h^0\alpha_i(1)$, the concatenation $G(x)^{-1}\cdot h^0(\alpha)\cdot H(x)$ becomes a path. Since $G_1(x_{i-1})=gh^1(x_{i-1})$ and $H_1(x_i)=\iota h^1(x_i)$ hold, $h(x, \alpha)$ becomes a homotopy pseudo-orbit. 
\end{proof}

 \begin{figure}
\setlength{\unitlength}{1mm}
\begin{picture}(100, 90)(25, 0)

\put(30, 0){\line(1, 0){30}}
\put(60, 0){\line(0, 1){30}}
\put(30, 0){\line(0, 1){30}}
\put(30, 30){\line(1, 0){30}}
\put(90, 0){\line(1, 0){30}}
\put(120, 0){\line(0, 1){30}}
\put(90, 0){\line(0, 1){30}}
\put(90, 30){\line(1, 0){30}}
\put(0, 50){\line(1, 0){30}}
\put(30, 50){\line(0, 1){30}}
\put(0, 50){\line(0, 1){30}}
\put(0, 80){\line(1, 0){30}}
\put(60, 50){\line(1, 0){30}}
\put(90, 50){\line(0, 1){30}}
\put(60, 50){\line(0, 1){30}}
\put(60, 80){\line(1, 0){30}}
\put(120, 50){\line(1, 0){30}}
\put(150, 50){\line(0, 1){30}}
\put(120, 50){\line(0, 1){30}}
\put(120, 80){\line(1, 0){30}}

\qbezier(15, 57)(14, 58)(15, 59)
\qbezier(15, 59)(16, 60)(15, 61)
\qbezier(15, 61)(14, 62)(15, 63)
\qbezier(15, 63)(16, 64)(15, 65)
\qbezier(15, 65)(14, 66)(15, 67)
\qbezier(15, 67)(16, 68)(15, 69)
\qbezier(15, 69)(14, 70)(15, 71)
\qbezier(15, 71)(15.5, 72)(15, 73)
\qbezier(37, 7)(36.5, 8)(37, 9)
\qbezier(37, 9)(38, 10)(37, 11)
\qbezier(37, 11)(36, 12)(37, 13)
\qbezier(37, 13)(38, 14)(37, 15)
\qbezier(37, 15)(36, 16)(37, 17)
\qbezier(37, 17)(38, 18)(37, 19)
\qbezier(37, 19)(36, 20)(37, 21)
\qbezier(37, 21)(38, 22)(37, 23)
\qbezier(53, 7)(52, 8)(53, 9)
\qbezier(53, 9)(54, 10)(53, 11)
\qbezier(53, 11)(52, 12)(53, 13)
\qbezier(53, 13)(54, 14)(53, 15)
\qbezier(53, 15)(52, 16)(53, 17)
\qbezier(53, 17)(54, 18)(53, 19)
\qbezier(53, 19)(52, 20)(53, 21)
\qbezier(53, 21)(54, 22)(53, 23)
\qbezier(75, 57)(74, 58)(75, 59)
\qbezier(75, 59)(76, 60)(75, 61)
\qbezier(75, 61)(74, 62)(75, 63)
\qbezier(75, 63)(76, 64)(75, 65)
\qbezier(75, 65)(74, 66)(75, 67)
\qbezier(75, 67)(76, 68)(75, 69)
\qbezier(75, 69)(74, 70)(75, 71)
\qbezier(75, 71)(75.5, 72)(75, 73)
\qbezier(97, 7)(96.5, 8)(97, 9)
\qbezier(97, 9)(98, 10)(97, 11)
\qbezier(97, 11)(96, 12)(97, 13)
\qbezier(97, 13)(98, 14)(97, 15)
\qbezier(97, 15)(96, 16)(97, 17)
\qbezier(97, 17)(98, 18)(97, 19)
\qbezier(97, 19)(96, 20)(97, 21)
\qbezier(97, 21)(98, 22)(97, 23)
\qbezier(113, 7)(112, 8)(113, 9)
\qbezier(113, 9)(114, 10)(113, 11)
\qbezier(113, 11)(112, 12)(113, 13)
\qbezier(113, 13)(114, 14)(113, 15)
\qbezier(113, 15)(112, 16)(113, 17)
\qbezier(113, 17)(114, 18)(113, 19)
\qbezier(113, 19)(112, 20)(113, 21)
\qbezier(113, 21)(114, 22)(113, 23)
\qbezier(135, 57)(134, 58)(135, 59)
\qbezier(135, 59)(136, 60)(135, 61)
\qbezier(135, 61)(134, 62)(135, 63)
\qbezier(135, 63)(136, 64)(135, 65)
\qbezier(135, 65)(134, 66)(135, 67)
\qbezier(135, 67)(136, 68)(135, 69)
\qbezier(135, 69)(134, 70)(135, 71)
\qbezier(135, 71)(135.5, 72)(135, 73)

\qbezier(37, 7)(38, 8)(39, 7)
\qbezier(39, 7)(40, 6)(41, 7)
\qbezier(41, 7)(42, 8)(43, 7)
\qbezier(43, 7)(44, 6)(45, 7)
\qbezier(45, 7)(46, 8)(47, 7)
\qbezier(47, 7)(48, 6)(49, 7)
\qbezier(49, 7)(50, 8)(51, 7)
\qbezier(51, 7)(52, 6)(53, 7)
\qbezier(37, 23)(38, 24)(39, 23)
\qbezier(39, 23)(40, 22)(41, 23)
\qbezier(41, 23)(42, 24)(43, 23)
\qbezier(43, 23)(44, 22)(45, 23)
\qbezier(45, 23)(46, 24)(47, 23)
\qbezier(47, 23)(48, 22)(49, 23)
\qbezier(49, 23)(50, 24)(51, 23)
\qbezier(51, 23)(52, 22)(53, 23)
\qbezier(97, 7)(98, 8)(99, 7)
\qbezier(99, 7)(100, 6)(101, 7)
\qbezier(101, 7)(102, 8)(103, 7)
\qbezier(103, 7)(104, 6)(105, 7)
\qbezier(105, 7)(106, 8)(107, 7)
\qbezier(107, 7)(108, 6)(109, 7)
\qbezier(109, 7)(110, 8)(111, 7)
\qbezier(111, 7)(112, 6)(113, 7)
\qbezier(97, 23)(98, 24)(99, 23)
\qbezier(99, 23)(100, 22)(101, 23)
\qbezier(101, 23)(102, 24)(103, 23)
\qbezier(103, 23)(104, 22)(105, 23)
\qbezier(105, 23)(106, 24)(107, 23)
\qbezier(107, 23)(108, 22)(109, 23)
\qbezier(109, 23)(110, 24)(111, 23)
\qbezier(111, 23)(112, 22)(113, 23)

\put(21.3, 65){\line(-2, -1){6}}
\put(81.3, 65){\line(-2, -1){6}}
\put(141.3, 65){\line(-2, -1){6}}
\put(26, 16){\line(3, -1){10.5}}
\put(65, 9){\line(-3, 1){11.5}}
\put(86, 16){\line(3, -1){10.5}}
\put(125, 9){\line(-3, 1){11.5}}

\put(36.5, 16){\vector(1, -2){1}}
\put(45, 7){\vector(2, 1){1}}
\put(52.5, 16){\vector(1, -2){1}}
\put(45, 23){\vector(2, 1){1}}
\put(96.5, 16){\vector(1, -2){1}}
\put(105, 7){\vector(2, 1){1}}
\put(112.5, 16){\vector(1, -2){1}}
\put(105, 23){\vector(2, 1){1}}
\put(14.5, 66){\vector(1, -2){1}}
\put(74.5, 66){\vector(1, -2){1}}
\put(134.5, 66){\vector(1, -2){1}}
\put(35.6, 26){\vector(1, -2){1}}
\put(35.6, 10){\vector(1, -2){1}}
\put(54.4, 26){\vector(-1, -2){1}}
\put(54.4, 10){\vector(-1, -2){1}}
\put(95.6, 26){\vector(1, -2){1}}
\put(95.6, 10){\vector(1, -2){1}}
\put(114.4, 26){\vector(-1, -2){1}}
\put(114.4, 10){\vector(-1, -2){1}}

\qbezier(37, 7)(26, 32)(15, 57)
\qbezier(37, 23)(26, 48)(15, 73)
\qbezier(53, 7)(64, 32)(75, 57)
\qbezier(53, 23)(64, 48)(75, 73)
\qbezier(97, 7)(86, 32)(75, 57)
\qbezier(97, 23)(86, 48)(75, 73)
\qbezier(113, 7)(124, 32)(135, 57)
\qbezier(113, 23)(124, 48)(135, 73)

\put(45, 15){\circle{6}}
\put(105, 15){\circle{6}}
\put(45, 17.9){\vector(-1, 0){1}}
\put(105, 17.9){\vector(-1, 0){1}}

\put(15, 57){\circle*{1.5}}
\put(15, 73){\circle*{1.5}}
\put(37, 7){\circle*{1.5}}
\put(37, 23){\circle*{1.5}}
\put(53, 7){\circle*{1.5}}
\put(53, 23){\circle*{1.5}}
\put(75, 57){\circle*{1.5}}
\put(75, 73){\circle*{1.5}}
\put(97, 7){\circle*{1.5}}
\put(97, 23){\circle*{1.5}}
\put(113, 7){\circle*{1.5}}
\put(113, 23){\circle*{1.5}}
\put(135, 57){\circle*{1.5}}
\put(135, 73){\circle*{1.5}}

\small
\put(43, 32){$X^0$}
\put(103, 32){$X^0$}
\put(13, 82){$X^1$}
\put(73, 82){$X^1$}
\put(133, 82){$X^1$}

\put(43, 25){$\alpha_i$}
\put(43, 3){$\alpha'_i$}
\put(101, 25){$\alpha_{i+1}$}
\put(101, 3){$\alpha'_{i+1}$}
\put(22, 65){$\beta_{i-1}$}
\put(82, 65){$\beta_i$}
\put(142, 65){$\beta_{i+1}$}
\put(3, 52){$\beta_{i-1}(1)=x'_{i-1}$}
\put(67, 52){$\beta_i(1)=x'_i$}
\put(124, 52){$\beta_{i+1}(1)=x'_{i+1}$}
\put(3, 75){$\beta_{i-1}(0)=x_{i-1}$}
\put(67, 75){$\beta_i(0)=x_i$}
\put(124, 75){$\beta_{i+1}(0)=x_{i+1}$}
\put(20, 37){$\sigma$}
\put(30, 41){$\sigma$}
\put(58, 41){$\iota$}
\put(68, 37){$\iota$}
\put(80, 37){$\sigma$}
\put(90, 41){$\sigma$}
\put(118, 41){$\iota$}
\put(128, 37){$\iota$}
\put(13, 16){$\sigma(\beta_{i-1})$}
\put(66, 7){$\iota(\beta_i)$}
\put(77, 16){$\sigma(\beta_i)$}
\put(126, 7){$\iota(\beta_{i+1})$}

\end{picture}
\vspace{1cm}
\begin{center}
Figure 5a. A homotopy $\beta$ : homotopies $\beta_i$ in $X^1$ are drawn as paths.
\end{center}
\end{figure}

 \begin{figure}
\setlength{\unitlength}{1.2mm}

\begin{picture}(100, 50)(25, 0)

\qbezier(20, 5)(21, 6.25)(20, 7.5)
\qbezier(20, 7.5)(19, 8.75)(20, 10)
\qbezier(20, 10)(21, 11.25)(20, 12.5)
\qbezier(20, 12.5)(19, 13.75)(20, 15)
\qbezier(20, 15)(21, 16.25)(20, 17.5)
\qbezier(20, 17.5)(19, 18.75)(20, 20)
\qbezier(20, 20)(21, 21.25)(20, 22.5)
\qbezier(20, 22.5)(19, 23.75)(20, 25)
\qbezier(20, 25)(21, 26.25)(20, 27.5)
\qbezier(20, 27.5)(19, 28.75)(20, 30)

\qbezier(40, 20)(41, 21.25)(40, 22.5)
\qbezier(40, 22.5)(39, 23.75)(40, 25)
\qbezier(40, 25)(41, 26.25)(40, 27.5)
\qbezier(40, 27.5)(39, 28.75)(40, 30)
\qbezier(40, 30)(41, 31.25)(40, 32.5)
\qbezier(40, 32.5)(39, 33.75)(40, 35)
\qbezier(40, 35)(41, 36.25)(40, 37.5)
\qbezier(40, 37.5)(39, 38.75)(40, 40)
\qbezier(40, 40)(41, 41.25)(40, 42.5)
\qbezier(40, 42.5)(39, 43.75)(40, 45)

\qbezier(70, 5)(71, 6.25)(70, 7.5)
\qbezier(70, 7.5)(69, 8.75)(70, 10)
\qbezier(70, 10)(71, 11.25)(70, 12.5)
\qbezier(70, 12.5)(69, 13.75)(70, 15)
\qbezier(70, 15)(71, 16.25)(70, 17.5)
\qbezier(70, 17.5)(69, 18.75)(70, 20)
\qbezier(70, 20)(71, 21.25)(70, 22.5)
\qbezier(70, 22.5)(69, 23.75)(70, 25)
\qbezier(70, 25)(71, 26.25)(70, 27.5)
\qbezier(70, 27.5)(69, 28.75)(70, 30)

\qbezier(90, 20)(91, 21.25)(90, 22.5)
\qbezier(90, 22.5)(89, 23.75)(90, 25)
\qbezier(90, 25)(91, 26.25)(90, 27.5)
\qbezier(90, 27.5)(89, 28.75)(90, 30)
\qbezier(90, 30)(91, 31.25)(90, 32.5)
\qbezier(90, 32.5)(89, 33.75)(90, 35)
\qbezier(90, 35)(91, 36.25)(90, 37.5)
\qbezier(90, 37.5)(89, 38.75)(90, 40)
\qbezier(90, 40)(91, 41.25)(90, 42.5)
\qbezier(90, 42.5)(89, 43.75)(90, 45)

\qbezier(140, 20)(141, 21.25)(140, 22.5)
\qbezier(140, 22.5)(139, 23.75)(140, 25)
\qbezier(140, 25)(141, 26.25)(140, 27.5)
\qbezier(140, 27.5)(139, 28.75)(140, 30)
\qbezier(140, 30)(141, 31.25)(140, 32.5)
\qbezier(140, 32.5)(139, 33.75)(140, 35)
\qbezier(140, 35)(141, 36.25)(140, 37.5)
\qbezier(140, 37.5)(139, 38.75)(140, 40)
\qbezier(140, 40)(141, 41.25)(140, 42.5)
\qbezier(140, 42.5)(139, 43.75)(140, 45)

\qbezier(120, 5)(121, 6.25)(120, 7.5)
\qbezier(120, 7.5)(119, 8.75)(120, 10)
\qbezier(120, 10)(121, 11.25)(120, 12.5)
\qbezier(120, 12.5)(119, 13.75)(120, 15)
\qbezier(120, 15)(121, 16.25)(120, 17.5)
\qbezier(120, 17.5)(119, 18.75)(120, 20)
\qbezier(120, 20)(121, 21.25)(120, 22.5)
\qbezier(120, 22.5)(119, 23.75)(120, 25)
\qbezier(120, 25)(121, 26.25)(120, 27.5)
\qbezier(120, 27.5)(119, 28.75)(120, 30)

\qbezier(40, 45)(41, 43)(43, 43.5)
\qbezier(43, 43.5)(45, 44)(46, 42)
\qbezier(46, 42)(47, 40)(49, 40.5)
\qbezier(49, 40.5)(51, 41)(52, 39)
\qbezier(52, 39)(53, 37)(55, 37.5)
\qbezier(55, 37.5)(57, 38)(58, 36)
\qbezier(58, 36)(59, 34)(61, 34.5)
\qbezier(61, 34.5)(63, 35)(64, 33)
\qbezier(64, 33)(65, 31)(67, 31.5)
\qbezier(67, 31.5)(69, 32)(70, 30)

\qbezier(40, 20)(41, 18)(43, 18.5)
\qbezier(43, 18.5)(45, 19)(46, 17)
\qbezier(46, 17)(47, 15)(49, 15.5)
\qbezier(49, 15.5)(51, 16)(52, 14)
\qbezier(52, 14)(53, 12)(55, 12.5)
\qbezier(55, 12.5)(57, 13)(58, 11)
\qbezier(58, 11)(59, 9)(61, 9.5)
\qbezier(61, 9.5)(63, 10)(64, 8)
\qbezier(64, 8)(65, 6)(67, 6.5)
\qbezier(67, 6.5)(69, 7)(70, 5)

\qbezier(90, 45)(91, 43)(93, 43.5)
\qbezier(93, 43.5)(95, 44)(96, 42)
\qbezier(96, 42)(97, 40)(99, 40.5)
\qbezier(99, 40.5)(101, 41)(102, 39)
\qbezier(102, 39)(103, 37)(105, 37.5)
\qbezier(105, 37.5)(107, 38)(108, 36)
\qbezier(108, 36)(109, 34)(111, 34.5)
\qbezier(111, 34.5)(113, 35)(114, 33)
\qbezier(114, 33)(115, 31)(117, 31.5)
\qbezier(117, 31.5)(119, 32)(120, 30)

\qbezier(90, 20)(91, 18)(93, 18.5)
\qbezier(93, 18.5)(95, 19)(96, 17)
\qbezier(96, 17)(97, 15)(99, 15.5)
\qbezier(99, 15.5)(101, 16)(102, 14)
\qbezier(102, 14)(103, 12)(105, 12.5)
\qbezier(105, 12.5)(107, 13)(108, 11)
\qbezier(108, 11)(109, 9)(111, 9.5)
\qbezier(111, 9.5)(113, 10)(114, 8)
\qbezier(114, 8)(115, 6)(117, 6.5)
\qbezier(117, 6.5)(119, 7)(120, 5)

\put(40.2, 20){\circle*{1.5}}
\put(40, 45){\circle*{1.5}}
\put(90.2, 20){\circle*{1.5}}
\put(90, 45){\circle*{1.5}}
\put(140.2, 20){\circle*{1.5}}
\put(140, 45){\circle*{1.5}}
\put(20.5, 5.5){\circle*{1.5}}
\put(20, 30.5){\circle*{1.5}}
\put(70.5, 5.5){\circle*{1.5}}
\put(70, 30.5){\circle*{1.5}}
\put(120.5, 5.5){\circle*{1.5}}
\put(120, 30.5){\circle*{1.5}}

\put(40.5, 31.5){\vector(0, -1){1}}
\put(69.5, 18.8){\vector(0, -1){1}}
\put(90.5, 31.5){\vector(0, -1){1}}
\put(119.5, 18.8){\vector(0, -1){1}}
\put(55, 37.5){\vector(3, 1){1}}
\put(55, 12.5){\vector(3, 1){1}}
\put(105, 37.5){\vector(3, 1){1}}
\put(105, 12.5){\vector(3, 1){1}}

\put(38.5, 19.5){\vector(2, 1){1}}
\put(38.8, 24.7){\vector(3, 1){1}}
\put(38.8, 29.7){\vector(3, 1){1}}
\put(38.8, 34.7){\vector(3, 1){1}}
\put(38.8, 39.7){\vector(3, 1){1}}
\put(38.5, 44.5){\vector(2, 1){1}}
\put(88.5, 19.5){\vector(2, 1){1}}
\put(88.8, 24.7){\vector(3, 1){1}}
\put(88.8, 29.7){\vector(3, 1){1}}
\put(88.8, 34.7){\vector(3, 1){1}}
\put(88.8, 39.7){\vector(3, 1){1}}
\put(88.5, 44.5){\vector(2, 1){1}}
\put(138.5, 19.5){\vector(3, 1){1}}
\put(138.8, 24.7){\vector(3, 1){1}}
\put(138.8, 29.7){\vector(3, 1){1}}
\put(138.8, 34.7){\vector(3, 1){1}}
\put(138.8, 39.7){\vector(3, 1){1}}
\put(138.5, 44.5){\vector(3, 1){1}}

\put(55, 25){\circle{7}}
\put(105, 25){\circle{7}}
\put(55, 28.5){\vector(-1, 0){1}}
\put(105, 28.5){\vector(-1, 0){1}}

\qbezier(20, 5)(28, 16)(40, 20)
\qbezier(20, 30)(28, 41)(40, 45)
\qbezier(70, 5)(78, 16)(90, 20)
\qbezier(70, 30)(78, 41)(90, 45)
\qbezier(120, 5)(128, 16)(140, 20)
\qbezier(120, 30)(128, 41)(140, 45)

\small
\put(27, 27){$\beta_{i-1}$}
\put(79, 27){$\beta_i$}
\put(127, 27){$\beta_{i+1}$}
\put(42, 32){$\sigma(\beta_{i-1})$}
\put(62, 17){$\iota(\beta_i)$}
\put(92, 32){$\sigma(\beta_i)$}
\put(109, 17){$\iota(\beta_{i+1})$}
\put(54, 40){$\alpha_i$}
\put(104, 40){$\alpha_{i+1}$}
\put(52, 9){$\alpha'_i$}
\put(102, 9){$\alpha'_{i+1}$}
\put(13, 44){$\beta_{i-1}(0)=x_{i-1}$}
\put(70, 44){$\beta_i(0)=x_i$}
\put(114, 44){$\beta_{i+1}(0)=x_{i+1}$}
\put(24, 7){$\beta_{i-1}(1)=x'_{i-1}$}
\put(74, 7){$\beta_i(1)=x'_i$}
\put(124, 7){$\beta_{i+1}(1)=x'_{i+1}$}

\linethickness{0.2pt}
\qbezier(20, 10)(28, 21)(40, 25)
\qbezier(20, 15)(28, 26)(40, 30)
\qbezier(20, 20)(28, 31)(40, 35)
\qbezier(20, 25)(28, 36)(40, 40)
\qbezier(70, 10)(78, 21)(90, 25)
\qbezier(70, 15)(78, 26)(90, 30)
\qbezier(70, 20)(78, 31)(90, 35)
\qbezier(70, 25)(78, 36)(90, 40)
\qbezier(120, 10)(128, 21)(140, 25)
\qbezier(120, 15)(128, 26)(140, 30)
\qbezier(120, 20)(128, 31)(140, 35)
\qbezier(120, 25)(128, 36)(140, 40)

\end{picture}
\vspace{1cm}
\begin{center}
Figure 5b. A homotopy $\beta$ : homotopies $\beta_i$ in $X^1$ are drawn as families of arrows.
\end{center}
\end{figure}

In Theorems~\ref{thm:exp-shadowing} and \ref{thm:hyp-shadowing} we will prove a closing lemma for homotopy pseudo-orbits. To state these we need

\begin{dfn} 
\label{hpo-homotopy}
Two homotopy pseudo-orbits $(x, \alpha)$ and $(x', \alpha')$ are said to be {\it homotopic} if there is a sequence $\beta=(\beta_i)$ of paths $\beta_i : [0, 1]\to X^1$ of bounded length with $\beta_i(0)=x_i$ and $\beta_i(1)=x'_i$ so that the path $\alpha_i \cdot \iota(\beta_i)$ is homotopic to the path $\sigma(\beta_{i-1})\cdot \alpha'_i$.
\end{dfn}

See Figures 5a and 5b, where a homotopy pseudo-orbit $(x, \alpha)$ is homotopic to another homotopy pseudo-orbit $(x', \alpha')$ by a homotopy $\beta=(\beta_i)$. Note that the two diagrams in each figure are commutative up to homotopies.

\end{section}

\begin{section}{Shadowing and its uniqueness for expanding systems}
\label{sec:exp-shadowing}

In this section we consider the case of one-sided orbits. When we use the term orbit in this section without further modification, we mean a one-sided orbit.

\subsection{Homotopy shadowing theorem}
\label{sub:exp-shadowing}

This subsection is devoted to the proof of the following theorem. A corresponding statement in the case of hyperbolic systems can be found in Theorem~\ref{thm:hyp-shadowing}.

\begin{thm} 
\label{thm:exp-shadowing}
Every homotopy pseudo-orbit $(x, \alpha)$ of an expanding system $\iota, \sigma : X^1\to X^0$ is homotopic to an orbit.
\end{thm}

\begin{proof}
We write $(x, \alpha)=((x_i)_{i\geq 0}, (\alpha_i)_{i\geq 1})$ where the index for the points $x_i$ starts at $i=0$ and the index for the paths $\alpha_i$ starts at $i=1$. In what follows we will inductively define a sequence of homotopy pseudo-orbits $((x^n_i)_{i\geq 0}, (\alpha^n_i)_{i\geq 1})$. Set $x^0_i\equiv x_i$ and $\alpha^0_i\equiv \alpha_i$. Suppose that a homotopy pseudo-orbit $((x^n_i)_{i\geq 0}, (\alpha^n_i)_{i\geq 1})$ is defined. This means that $x^n_i\in X^1$ and $\alpha^n_i : [0, 1]\to X^0$ satisfy $\alpha^n_i(0)=\sigma(x^n_{i-1})$ and $\alpha^n_i(1)=\iota(x^n_i)$. Then, since $\sigma$ is a covering and $\alpha^n_i(0)=\sigma(x^n_{i-1})$, there exists a unique lift $\beta^n_{i-1} : [0, 1]\to X^1$ of $\alpha^n_i$ by $\sigma$ so that $\beta^n_{i-1}(0)=x^n_{i-1}$ by the path lifting property. Put $\alpha^{n+1}_i\equiv \iota(\beta^n_i)$ and $x^{n+1}_i\equiv \beta^n_i(1)$. Then, we have $\sigma(x^{n+1}_{i-1})=\sigma(\beta^n_{i-1}(1))=\alpha^n_i(1)=\iota(x^n_i)=\iota(\beta^n_i(0))=\alpha^{n+1}_i(0)$ and $\iota (x^{n+1}_i)=\iota(\beta^n_i(1))=\alpha^{n+1}_i(1)$. This means that $((x^{n+1}_i)_{i\geq 0}, (\alpha^{n+1}_i)_{i\geq 1})$ is a homotopy pseudo-orbit.

 \begin{figure}
\setlength{\unitlength}{0.8mm}

\begin{picture}(100, 125)(40, 0)

\put(0, 0){\line(1, 0){40}}
\put(0, 0){\line(0, 1){50}}
\put(40, 0){\line(0, 1){50}}
\put(0, 50){\line(1, 0){40}}
\put(30, 60){\line(1, 0){38}}
\put(30, 60){\line(0, 1){50}}
\put(68, 60){\line(0, 1){50}}
\put(30, 110){\line(1, 0){38}}
\put(60, 0){\line(1, 0){40}}
\put(60, 0){\line(0, 1){50}}
\put(100, 0){\line(0, 1){50}}
\put(60, 50){\line(1, 0){40}}
\put(90, 60){\line(1, 0){38}}
\put(90, 60){\line(0, 1){50}}
\put(128, 60){\line(0, 1){50}}
\put(90, 110){\line(1, 0){38}}
\put(120, 0){\line(1, 0){40}}
\put(120, 0){\line(0, 1){50}}
\put(160, 0){\line(0, 1){50}}
\put(120, 50){\line(1, 0){40}}
\put(150, 60){\line(1, 0){38}}
\put(150, 60){\line(0, 1){50}}
\put(188, 60){\line(0, 1){50}}
\put(150, 110){\line(1, 0){38}}

\put(21.5, 24){\vector(-1, -2){1}}
\put(21.5, 18){\vector(-1, -2){1}}
\put(21.5, 14){\vector(-1, -2){1}}
\put(21.5, 8){\vector(-1, -2){1}}
\put(21.5, 33){\vector(-1, -2){1}}
\put(78.5, 47.5){\vector(2, -3){1}}
\put(78.5, 33){\vector(1, -2){1}}
\put(78.5, 24){\vector(1, -2){1}}
\put(78.5, 18){\vector(1, -2){1}}
\put(78.5, 8){\vector(1, -2){1}}
\put(81.5, 24){\vector(-1, -2){1}}
\put(81.5, 18){\vector(-1, -2){1}}
\put(81.5, 14){\vector(-1, -2){1}}
\put(81.5, 8){\vector(-1, -2){1}}
\put(82, 34){\vector(-1, -2){1}}
\put(138.5, 47.5){\vector(2, -3){1}}
\put(138.5, 33){\vector(1, -2){1}}
\put(138.5, 24){\vector(1, -2){1}}
\put(138.5, 18){\vector(1, -2){1}}
\put(138.5, 8){\vector(1, -2){1}}
\put(141.5, 24){\vector(-1, -2){1}}
\put(141.5, 18){\vector(-1, -2){1}}
\put(141.5, 14){\vector(-1, -2){1}}
\put(141.5, 8){\vector(-1, -2){1}}
\put(142, 34){\vector(-1, -2){1}}
\put(189.5, 61){\vector(2, -3){1}}
\put(189.5, 49){\vector(1, -2){1}}
\put(189.5, 41){\vector(1, -2){1}}
\put(189.5, 36){\vector(1, -2){1}}
\put(189.5, 28){\vector(1, -2){1}}

\put(20, 5){\circle*{2}}
\put(20, 7){\circle*{0.5}}
\put(20, 8){\circle*{0.5}}
\put(20, 9){\circle*{0.5}}
\put(20, 11){\circle*{1.3}}
\put(20, 15){\circle*{1.3}}
\put(20, 21){\circle*{1.3}}
\put(20, 30){\circle*{2}}
\put(50, 67){\circle*{2}}
\put(50, 69){\circle*{0.5}}
\put(50, 70.5){\circle*{0.5}}
\put(50, 72){\circle*{0.5}}
\put(50, 74){\circle*{1.3}}
\put(50, 78){\circle*{1.3}}
\put(50, 84){\circle*{1.3}}
\put(50, 93){\circle*{2}}
\put(80, 5){\circle*{2}}
\put(80, 7){\circle*{0.5}}
\put(80, 8){\circle*{0.5}}
\put(80, 9){\circle*{0.5}}
\put(80, 11){\circle*{1.3}}
\put(80, 15){\circle*{1.3}}
\put(80, 21){\circle*{1.3}}
\put(80, 30){\circle*{2}}
\put(80, 45){\circle*{2}}
\put(110, 67){\circle*{2}}
\put(110, 69){\circle*{0.5}}
\put(110, 70.5){\circle*{0.5}}
\put(110, 72){\circle*{0.5}}
\put(110, 74){\circle*{1.3}}
\put(110, 78){\circle*{1.3}}
\put(110, 84){\circle*{1.3}}
\put(110, 93){\circle*{2}}
\put(140, 5){\circle*{2}}
\put(140, 7){\circle*{0.5}}
\put(140, 8){\circle*{0.5}}
\put(140, 9){\circle*{0.5}}
\put(140, 11){\circle*{1.3}}
\put(140, 15){\circle*{1.3}}
\put(140, 21){\circle*{1.3}}
\put(140, 30){\circle*{2}}
\put(140, 45){\circle*{2}}
\put(170, 67){\circle*{2}}
\put(170, 69){\circle*{0.5}}
\put(170, 70.5){\circle*{0.5}}
\put(170, 72){\circle*{0.5}}
\put(170, 74){\circle*{1.3}}
\put(170, 78){\circle*{1.3}}
\put(170, 84){\circle*{1.3}}
\put(170, 93){\circle*{2}}

\linethickness{0.3pt}

\qbezier(20, 21)(35, 52.5)(50, 84)
\qbezier(20, 15)(35, 46.5)(50, 78)
\qbezier(20, 11)(35, 42.5)(50, 74)
\qbezier(50, 74)(65, 45)(80, 15)
\qbezier(50, 78)(65, 49.5)(80, 21)
\qbezier(50, 84)(65, 57)(80, 30)
\qbezier(80, 21)(95, 52.5)(110, 84)
\qbezier(80, 15)(95, 46.5)(110, 78)
\qbezier(80, 11)(95, 42.5)(110, 74)
\qbezier(110, 74)(125, 45)(140, 15)
\qbezier(110, 78)(125, 49)(140, 21)
\qbezier(110, 84)(125, 57)(140, 30)
\qbezier(140, 21)(155, 52.5)(170, 84)
\qbezier(140, 15)(155, 46.5)(170, 78)
\qbezier(140, 11)(155, 42.5)(170, 74)
\qbezier(170, 74)(185, 45)(190, 35)
\qbezier(170, 78)(185, 49)(190, 40)
\qbezier(170, 84)(185, 57)(190, 48)

\qbezier(20, 31)(19, 30)(20, 29)
\qbezier(20, 29)(21, 28)(20, 27)
\qbezier(20, 27)(19, 26)(20, 25)
\qbezier(20, 25)(21, 24)(20, 23)
\qbezier(20, 23)(19, 22)(20, 21)
\qbezier(20, 21)(21, 20)(20, 19)
\qbezier(20, 19)(19, 18)(20, 17)
\qbezier(20, 17)(21, 16)(20, 15)
\qbezier(20, 15)(19, 14)(20, 13)
\qbezier(20, 13)(21, 12)(20, 11)
\qbezier(50, 92)(51, 91)(50, 90)
\qbezier(50, 90)(49, 89)(50, 88)
\qbezier(50, 88)(51, 87)(50, 86)
\qbezier(50, 86)(49, 85)(50, 84)
\qbezier(50, 84)(51, 83)(50, 82)
\qbezier(50, 82)(49, 81)(50, 80)
\qbezier(50, 80)(51, 79)(50, 78)
\qbezier(50, 78)(49, 77)(50, 76)
\qbezier(50, 76)(51, 75)(50, 74)
\qbezier(80, 31)(79, 30)(80, 29)
\qbezier(80, 29)(81, 28)(80, 27)
\qbezier(80, 27)(79, 26)(80, 25)
\qbezier(80, 25)(81, 24)(80, 23)
\qbezier(80, 23)(79, 22)(80, 21)
\qbezier(80, 21)(81, 20)(80, 19)
\qbezier(80, 19)(79, 18)(80, 17)
\qbezier(80, 17)(81, 16)(80, 15)
\qbezier(80, 15)(79, 14)(80, 13)
\qbezier(80, 13)(81, 12)(80, 11)
\qbezier(110, 92)(111, 91)(110, 90)
\qbezier(110, 90)(109, 89)(110, 88)
\qbezier(110, 88)(111, 87)(110, 86)
\qbezier(110, 86)(109, 85)(110, 84)
\qbezier(110, 84)(111, 83)(110, 82)
\qbezier(110, 82)(109, 81)(110, 80)
\qbezier(110, 80)(111, 79)(110, 78)
\qbezier(110, 78)(109, 77)(110, 76)
\qbezier(110, 76)(111, 75)(110, 74)
\qbezier(140, 31)(139, 30)(140, 29)
\qbezier(140, 29)(141, 28)(140, 27)
\qbezier(140, 27)(139, 26)(140, 25)
\qbezier(140, 25)(141, 24)(140, 23)
\qbezier(140, 23)(139, 22)(140, 21)
\qbezier(140, 21)(141, 20)(140, 19)
\qbezier(140, 19)(139, 18)(140, 17)
\qbezier(140, 17)(141, 16)(140, 15)
\qbezier(140, 15)(139, 14)(140, 13)
\qbezier(140, 13)(141, 12)(140, 11)
\qbezier(170, 92)(171, 91)(170, 90)
\qbezier(170, 90)(169, 89)(170, 88)
\qbezier(170, 88)(171, 87)(170, 86)
\qbezier(170, 86)(169, 85)(170, 84)
\qbezier(170, 84)(171, 83)(170, 82)
\qbezier(170, 82)(169, 81)(170, 80)
\qbezier(170, 80)(171, 79)(170, 78)
\qbezier(170, 78)(169, 77)(170, 76)
\qbezier(170, 76)(171, 75)(170, 74)

\qbezier(40, 88)(44, 88)(49.5, 88)
\qbezier(40, 80)(45, 80)(49.5, 80)
\qbezier(40, 75.5)(44.5, 75.5)(50, 75.5)
\qbezier(100, 88)(104, 88)(109.5, 88)
\qbezier(100, 80)(104, 80)(109.5, 80)
\qbezier(100, 75.5)(104.5, 75.5)(110, 75.5)
\qbezier(160, 88)(164, 88)(169.5, 88)
\qbezier(160, 80)(164, 80)(169.5, 80)
\qbezier(160, 75.5)(164.5, 75.5)(170, 75.5)
\qbezier(19.5, 25)(15, 25)(13, 25)
\qbezier(19.5, 19)(15, 19)(13, 19)
\qbezier(19.5, 13)(15, 13)(13, 13)
\qbezier(79.5, 25)(75, 25)(68, 25)
\qbezier(79.5, 19)(75, 19)(68, 19)
\qbezier(79.5, 13)(75, 13)(68, 13)
\qbezier(139.5, 25)(135, 25)(128, 25)
\qbezier(139.5, 19)(135, 19)(128, 19)
\qbezier(139.5, 13)(135, 13)(128, 13)

\small
\put(48, 113){$X^1$}
\put(108, 113){$X^1$}
\put(168, 113){$X^1$}
\put(18, 53){$X^0$}
\put(78, 53){$X^0$}
\put(138, 53){$X^0$}
\put(34, 87){$\beta^0_0$}
\put(34, 80){$\beta^1_0$}
\put(34, 74){$\beta^2_0$}
\put(94, 87){$\beta^0_1$}
\put(94, 80){$\beta^1_1$}
\put(94, 74){$\beta^2_1$}
\put(154, 87){$\beta^0_2$}
\put(154, 80){$\beta^1_2$}
\put(154, 74){$\beta^2_2$}
\put(7, 24){$\alpha^1_0$}
\put(7, 18){$\alpha^2_0$}
\put(7, 12){$\alpha^3_0$}
\put(74, 37){$\alpha^0_1$}
\put(62, 24){$\alpha^1_1$}
\put(62, 18){$\alpha^2_1$}
\put(62, 12){$\alpha^3_1$}
\put(134, 37){$\alpha^0_2$}
\put(122, 24){$\alpha^1_2$}
\put(122, 18){$\alpha^2_2$}
\put(122, 12){$\alpha^3_2$}
\put(52, 94){$x^0_0$}
\put(112, 94){$x^0_1$}
\put(172, 94){$x^0_2$}
\put(12.5, 3){$c^{\infty}_0$}
\put(72.5, 3){$c^{\infty}_1$}
\put(132.5, 3){$c^{\infty}_2$}
\put(42.5, 66){$x^{\infty}_0$}
\put(102.5, 66){$x^{\infty}_1$}
\put(162.5, 66){$x^{\infty}_2$}

\put(36.5, 53){$\iota$}
\put(96.5, 53){$\iota$}
\put(156.5, 53){$\iota$}
\put(63.5, 53){$\sigma$}
\put(123.5, 53){$\sigma$}
\put(183.5, 53){$\sigma$}

\linethickness{1pt}

\qbezier(20, 30)(35, 61.5)(50, 93)
\qbezier(20, 5)(35, 36)(50, 67)
\qbezier(50, 67)(65, 36)(80, 5)
\qbezier(50, 93)(65, 68)(80, 45)
\qbezier(80, 30)(95, 61.5)(110, 93)
\qbezier(80, 5)(95, 36)(110, 67)
\qbezier(110, 67)(125, 36)(140, 5)
\qbezier(110, 93)(125, 68)(140, 45)
\qbezier(140, 30)(155, 61.5)(170, 93)
\qbezier(140, 5)(155, 36)(170, 67)
\qbezier(170, 67)(185, 37)(190, 27)
\qbezier(170, 93)(185, 68)(190, 60)

\qbezier(80, 45)(81, 44)(80, 43)
\qbezier(80, 43)(79, 42)(80, 41)
\qbezier(80, 41)(81, 40)(80, 39)
\qbezier(80, 39)(79, 38)(80, 37)
\qbezier(80, 37)(81, 36)(80, 35)
\qbezier(80, 35)(79, 34)(80, 33)
\qbezier(80, 33)(81, 32)(80, 31)
\qbezier(140, 45)(141, 44)(140, 43)
\qbezier(140, 43)(139, 42)(140, 41)
\qbezier(140, 41)(141, 40)(140, 39)
\qbezier(140, 39)(139, 38)(140, 37)
\qbezier(140, 37)(141, 36)(140, 35)
\qbezier(140, 35)(139, 34)(140, 33)
\qbezier(140, 33)(141, 32)(140, 31)

\end{picture}
\vspace{1cm}
\begin{center}
Figure 6a. Pullbacks of homotopies : homotopies $\beta^n_i$ are drawn as paths.
\end{center}

\end{figure}

 \begin{figure}
\setlength{\unitlength}{1mm}

\begin{picture}(100, 60)(15, -5)

\put(10, 0){\circle*{2}}
\put(50, 0){\circle*{2}}
\put(90, 0){\circle*{2}}
\put(10, 30){\circle*{2}}
\put(50, 30){\circle*{2}}
\put(90, 30){\circle*{2}}
\put(50, 50){\circle*{2}}
\put(90, 50){\circle*{2}}
\put(10, 20){\circle*{1.2}}
\put(50, 20){\circle*{1.2}}
\put(90, 20){\circle*{1.2}}
\put(10, 14.5){\circle*{1.2}}
\put(50, 14.5){\circle*{1.2}}
\put(90, 14.5){\circle*{1.2}}
\put(10, 12){\circle*{1.2}}
\put(50, 12){\circle*{1.2}}
\put(90, 12){\circle*{1.2}}

\put(10, 4){\circle*{0.5}}
\put(10, 6){\circle*{0.5}}
\put(10, 8){\circle*{0.5}}
\put(50, 4){\circle*{0.5}}
\put(50, 6){\circle*{0.5}}
\put(50, 8){\circle*{0.5}}
\put(90, 4){\circle*{0.5}}
\put(90, 6){\circle*{0.5}}
\put(90, 8){\circle*{0.5}}

\put(48, 49.6){\vector(3, 1){1}}
\put(88, 49.6){\vector(3, 1){1}}
\put(130, 50){\vector(3, 1){1}}
\put(48, 30){\vector(1, 0){1}}
\put(88, 30){\vector(1, 0){1}}
\put(130, 30){\vector(1, 0){1}}
\put(48, 20){\vector(1, 0){1}}
\put(88, 20){\vector(1, 0){1}}
\put(130, 20){\vector(1, 0){1}}
\put(48, 14.5){\vector(1, 0){1}}
\put(88, 14.5){\vector(1, 0){1}}
\put(130, 14.5){\vector(1, 0){1}}
\put(48, 0){\vector(1, 0){1}}
\put(88, 0){\vector(1, 0){1}}
\put(130, 0){\vector(1, 0){1}}

\linethickness{0.3pt}
\qbezier(10, 20)(25, 28)(50, 30)
\qbezier(50, 20)(65, 28)(90, 30)
\qbezier(90, 20)(105, 28)(130, 30)
\qbezier(10, 14.5)(25, 19)(50, 20)
\qbezier(50, 14.5)(65, 19)(90, 20)
\qbezier(90, 14.5)(105, 19)(130, 20)
\qbezier(10, 12)(25, 14)(50, 14.5)
\qbezier(50, 12)(65, 14)(90, 14.5)
\qbezier(90, 12)(105, 14)(130, 14.5)

\qbezier(10, 30)(11, 29)(10, 28)
\qbezier(10, 28)(9, 27)(10, 26)
\qbezier(10, 26)(11, 25)(10, 24)
\qbezier(10, 24)(9, 23)(10, 22)
\qbezier(10, 22)(11, 21)(10, 20)
\qbezier(10, 20)(9, 19)(10, 18)
\qbezier(10, 18)(11, 17)(10, 16)
\qbezier(10, 16)(9, 15)(10, 14)
\qbezier(10, 14)(11, 13)(10, 12)
\qbezier(50, 30)(51, 29)(50, 28)
\qbezier(50, 28)(49, 27)(50, 26)
\qbezier(50, 26)(51, 25)(50, 24)
\qbezier(50, 24)(49, 23)(50, 22)
\qbezier(50, 22)(51, 21)(50, 20)
\qbezier(50, 20)(49, 19)(50, 18)
\qbezier(50, 18)(51, 17)(50, 16)
\qbezier(50, 16)(49, 15)(50, 14)
\qbezier(50, 14)(51, 13)(50, 12)
\qbezier(90, 30)(91, 29)(90, 28)
\qbezier(90, 28)(89, 27)(90, 26)
\qbezier(90, 26)(91, 25)(90, 24)
\qbezier(90, 24)(89, 23)(90, 22)
\qbezier(90, 22)(91, 21)(90, 20)
\qbezier(90, 20)(89, 19)(90, 18)
\qbezier(90, 18)(91, 17)(90, 16)
\qbezier(90, 16)(89, 15)(90, 14)
\qbezier(90, 14)(91, 13)(90, 12)

\linethickness{1pt}

\qbezier(10, 30)(25, 45)(50, 50)
\qbezier(50, 30)(65, 45)(90, 50)
\qbezier(90, 30)(105, 45)(130, 50)
\qbezier(10, 0)(20, 0)(130, 0)

\qbezier(50, 50)(51, 49)(50, 48)
\qbezier(50, 48)(49, 47)(50, 46)
\qbezier(50, 46)(51, 45)(50, 44)
\qbezier(50, 44)(49, 43)(50, 42)
\qbezier(50, 42)(51, 41)(50, 40)
\qbezier(50, 40)(49, 39)(50, 38)
\qbezier(50, 38)(51, 37)(50, 36)
\qbezier(50, 36)(49, 35)(50, 34)
\qbezier(50, 34)(51, 33)(50, 32)
\qbezier(50, 32)(49, 31)(50, 30)
\qbezier(90, 50)(91, 49)(90, 48)
\qbezier(90, 48)(89, 47)(90, 46)
\qbezier(90, 46)(91, 45)(90, 44)
\qbezier(90, 44)(89, 43)(90, 42)
\qbezier(90, 42)(91, 41)(90, 40)
\qbezier(90, 40)(89, 39)(90, 38)
\qbezier(90, 38)(91, 37)(90, 36)
\qbezier(90, 36)(89, 35)(90, 34)
\qbezier(90, 34)(91, 33)(90, 32)
\qbezier(90, 32)(89, 31)(90, 30)

\small
\put(28, 33){$\beta^0_0$}
\put(28, 21.5){$\beta^1_0$}
\put(28, 15){$\beta^2_0$}
\put(68, 33){$\beta^0_1$}
\put(68, 21.5){$\beta^1_1$}
\put(68, 15){$\beta^2_1$}
\put(108, 33){$\beta^0_2$}
\put(108, 21.5){$\beta^1_2$}
\put(108, 15){$\beta^2_2$}
\put(11, 25){$\alpha^1_0$}
\put(11, 17.5){$\alpha^2_0$}
\put(11, 12.5){$\alpha^3_0$}
\put(51, 40){$\alpha^0_1$}
\put(51, 25){$\alpha^1_1$}
\put(51, 17.5){$\alpha^2_1$}
\put(51, 12.5){$\alpha^3_1$}
\put(91, 40){$\alpha^0_2$}
\put(91, 25){$\alpha^1_2$}
\put(91, 17.5){$\alpha^2_2$}
\put(91, 12.5){$\alpha^3_2$}
\put(28, 47){$x^0_0$}
\put(68, 47){$x^0_1$}
\put(108, 47){$x^0_2$}
\put(28, 2){$x^{\infty}_0$}
\put(68, 2){$x^{\infty}_1$}
\put(108, 2){$x^{\infty}_2$}
\put(8, -5){$c^{\infty}_0$}
\put(48, -5){$c^{\infty}_1$}
\put(88, -5){$c^{\infty}_2$}

\end{picture}
\vspace{1cm}
\begin{center}
Figure 6b. Pullbacks of homotopies : homotopies $\beta^n_i$ are drawn as families of arrows.
\end{center}

\end{figure}

Let $l^n(\alpha)$ be the length of a path $\alpha$ in $X^n$. Here we need

\begin{lmm} 
\label{lmm:exp-estimate}
There exists a constant $C\geq 0$ so that 
$$l^0(\alpha_i^n)\leq \frac{C}{\lambda^n}$$
holds for all $n\geq 0$ and $i\geq 1$, where $\lambda>1$ is as in Definition~\ref{dfn:expsyst}.
\end{lmm}

\begin{proof}
By the definition of homotopy pseudo-orbit, there is a constant $C\geq 0$ so that $l^0(\alpha^0_i)\leq C$ for all $i\geq 1$. Recall that $\sigma(\beta^n_{i-1})=\alpha^n_i$ and $\alpha^{n+1}_i= \iota(\beta^n_i)$. Hence, the expanding property of $\iota, \sigma : X^1\to X^0$ implies the desired estimate.
\end{proof}

Since $\alpha^{n+1}_i(0)=\iota(\beta^n_i(0))=\iota(x^n_i)=\alpha^n_i(1)$, we can concatenate the paths $\alpha^n_i$ ($n=0, 1, 2, \ldots$) for each $i\geq 1$ as follows. Let $I_n=[1-\frac{1}{2^n}, 1-\frac{1}{2^{n+1}}]$ and define $\alpha^{\infty}_i|_{I_n} : I_n\to X^0$ as $\alpha^{\infty}_i(t)\equiv \alpha^n_i(2^{n+1}(t-1+\frac{1}{2^n}))$ for $t\in I_n$ and $n\geq 0$. See Figures 6a and 6b. This construction gives a continuous map $\alpha^{\infty}_i : [0, 1)\to X^0$. Lemma~\ref{lmm:exp-estimate} implies that the map naturally extends to $\alpha^{\infty}_i : [0, 1]\to X^0$ since $\alpha_i^{\infty}(1)\equiv \lim_{n\to \infty}\alpha^n_i(1)$ exists. A similar construction implies that one can concatenate the paths $\beta^n_i$ ($n=0, 1, 2, \ldots$) to get a continuous map $\beta^{\infty}_i : [0, 1]\to X^1$. Put $x^{\infty}_i\equiv \lim_{n\to \infty}x^n_i=\lim_{n\to \infty}\beta^n_i(0)$. By the definition of $\beta^n=(\beta_i^n)_{i\geq 0}$, one easily sees that $\beta_i^n(0)=x_i^n$ and $\beta_i^n(1)=x_i^{n+1}$ hold. From these we have $\iota(x^{\infty}_i)=\alpha^{\infty}_i(1)$ and $\sigma(x_{i-1}^{\infty})=\alpha^{\infty}_i(0)$. Since $l^0(\alpha^n_i)\to 0$ when $n\to \infty$ by Lemma~\ref{lmm:exp-estimate}, we conclude $\iota(x^{\infty}_i)=\alpha^{\infty}_i(1)=\alpha^{\infty}_i(0)=\sigma(x_{i-1}^{\infty})$. This means that $(x^{\infty}, c^{\infty})=((x^{\infty}_i)_{i\geq 0}, (c^{\infty}_i)_{i\geq 1})$ becomes an orbit, where $c^{\infty}_i(t)\equiv \alpha^{\infty}_i(0)=\alpha^{\infty}_i(1)$ is a constant homotopy.

Again by the definition of $\beta^n$, we have $\alpha_i^n\cdot \iota(\beta_i^n)=\sigma(\beta^n_{i-1})\cdot \alpha_i^{n+1}$. By using this several times, one gets
\begin{align*}
\alpha^0_i
= & \sigma(\beta^0_{i-1})\cdot \alpha^1_i\cdot \iota(\beta^0_i)^{-1}\\
= & \sigma(\beta^0_{i-1})\sigma(\beta^1_{i-1})\cdot \alpha^2_i\cdot \iota (\beta_i^1)^{-1}\iota(\beta^0_i)^{-1}\\
= & \cdots \\
= & \sigma(\beta^0_{i-1})\cdots \sigma(\beta^{n-1}_{i-1})\cdot \alpha^n_i\cdot \iota (\beta_i^{n-1})^{-1}\cdots \iota(\beta^0_i)^{-1}\\
= & \sigma(\beta^0_{i-1}\cdots \beta^{n-1}_{i-1})\cdot \alpha^n_i\cdot \iota (\beta^0_i\cdots \beta^{n-1}_i)^{-1}.
\end{align*}
Letting $n\to \infty$, we get $\alpha^0_i=\sigma(\beta_{i-1}^{\infty})\cdot c^{\infty}_i\cdot \iota(\beta_i^{\infty})^{-1}$. We also have $\beta^{\infty}_i(0)=x^0_i$ and $\beta^{\infty}_i(1)=x^{\infty}_i$. Thus, the homotopy pseudo-orbit $(x, \alpha)$ is homotopic to the orbit $(x^{\infty}, c^{\infty})$. This completes the proof of Theorem~\ref{thm:exp-shadowing}.
\end{proof}

\subsection{Uniqueness of shadowing} 
\label{sub:exp-unique}

In this subsection we prove two results. First we show that if two orbits of an expanding system are homotopic, then they are equal. In particular, it follows that the shadowing orbit found in Theorem~\ref{thm:exp-shadowing} is unique. Based on this fact, we finish the proof of Theorem~\ref{thm:exp-funct}.

\begin{prp} 
\label{prp:exp-unique}
Let $(x, \alpha)$ and $(x', \alpha')$ be homotopy pseudo-orbits of an expanding system $\iota, \sigma : X^1\to X^0$ with $l^0(\alpha_i)\leq C$ and $l^0(\alpha'_i)\leq C'$ for all $i\geq 1$. If they are homotopic, then there is a sequence of homotopies $\beta=(\beta_i)_{i\geq 1}$ so that 
$$l^1(\beta_i)\leq \frac{\lambda(C+C')}{\lambda-1}.$$
\end{prp}

\begin{proof}
Given a path $\alpha$ in $X^0$, let us use the notation $\|\alpha\|$ to denote the infimum of lengths of paths $\gamma$ homotopic to $\alpha$ relative to endpoints. Note that the infimum is realized since we assume that $X^0$ and $X^1$ are complete length spaces.

Since the two homotopy pseudo-orbits $(x, \alpha)$ and $(x', \alpha')$ are homotopic, there exists a sequence of paths $\beta_i : [0, 1]\to X^1$ so that $\sigma(\beta_{i-1})\cdot \alpha_i'$ is homotopic to $\alpha_i\cdot \iota(\beta_i)$ with $l^1(\beta_i)\leq C''$ for some constant $C''\geq 0$. This implies
$$\|\sigma(\beta_{i-1})\|\leq l^0(\alpha_i)+\|\iota(\beta_i)\|+l^0(\alpha_i')\leq C+\|\iota(\beta_i)\|+C'.$$
It follows from the homotopy lifting property for $\sigma$ that if $\gamma$ is a path homotopic to $\sigma(\beta_i)$ relative to endpoints, then $\gamma$ has a lift to a path $\tilde{\gamma}$ in $X^1$ which is homotopic to $\beta_i$ relative to endpoints. Thus, $\iota(\tilde{\gamma})$ is homotopic to $\iota(\beta_i)$ and $l^0(\iota(\tilde{\gamma}))\leq \frac{1}{\lambda}l^0(\gamma)$. Hence, we have
$$\|\iota(\beta_i)\|\leq \frac{1}{\lambda}\|\sigma(\beta_i)\|.$$

Combining the above two inequalities gives
$$\|\sigma(\beta_{i-1})\|\leq \|\iota(\beta_i)\|+C+C'\leq \frac{1}{\lambda}\|\sigma(\beta_i)\|+C+C'.$$
Let us put $h(x)\equiv\frac{1}{\lambda}x+C+C'$. We can rewrite the above inequality as $\|\sigma(\beta_{i-1})\|\leq h(\|\sigma(\beta_i)\|)$. Applying this inequality repeatedly gives $\|\sigma(\beta_i)\|\leq h^n(\|\sigma(\beta_{i+n})\|)$. Since $\|\sigma(\beta_{i+n})\|\leq C''$ and since $h$ is monotone increasing, we have $\|\sigma(\beta_i)\|\leq h^n(C'')$. The function $h$ has a unique attractive fixed point at $x=\lambda(C+C')/(\lambda-1)$. Letting $n$ go to infinity gives
$$\|\sigma(\beta_i)\|\leq \frac{\lambda(C+C')}{\lambda-1}.$$
This completes the proof. 
\end{proof}

\begin{cor} 
\label{cor:hom-exp} 
If two orbits $(x_i)_{i\geq 0}$ and $(x'_i)_{i\geq 0}$ of an expanding system are homotopic, then they are equal.
\end{cor}

\begin{proof}
One can view an orbit as a homotopy pseudo-orbit with homotopies of length zero. Thus, we apply the previous result with $C=C'=0$ to conclude that the homotopies $\beta_i$ between $x_i$ and $x_i'$ are homotopic to some homotopies of length zero. In particular, this means that the endpoints $x_i$ and $x_i'$ of the homotopies $\beta_i$ are equal. Thus, we are done. 
\end{proof}

\begin{proof}[Proof of Theorem~\ref{thm:exp-funct}.]
Take an orbit $x\in X^{+\infty}$. This defines a homotopy pseudo-orbit $h(x)$ of $\iota, g : Y^1\to Y^0$ by Lemma~\ref{lmm:orbhpo}. Thanks to Theorem~\ref{thm:exp-shadowing} and Corollary~\ref{cor:hom-exp}, there exists a unique orbit $y$ of $\iota, g : Y^1\to Y^0$ which is homotopic to $h(x)$. Define $h^{\infty} : X^{+\infty}\to Y^{+\infty}$ by $h^{\infty}(x)\equiv y$. Then, one can easily verify $h^{\infty} \hat{f}=\hat{g}h^{\infty}$.

We next show the continuity of $h^{\infty}$. We replace the metric of $X^1$ by the pullback of $d^0$ by $\sigma$ which we again denote by $d^1$. Since $\sigma$ is a local homeomorphism, this does not change the topology of $X^1$. Take two orbits $x=(x_i)_{i\geq 0}$ and $\tilde{x}=(\tilde{x}_i)_{i\geq 0}$ of $\iota, \sigma : X^1\to X^0$. Write $(y_i)_{i\geq 0}\equiv h^{\infty}(x)$ and $(\tilde{y}_i)_{i\geq 0}\equiv h^{\infty}(\tilde{x})$. By the construction of the shadowing orbit in Theorem~\ref{thm:exp-shadowing}, we have $y_0=\lim_{n\to \infty}\beta^n_0(0)$ and $\tilde{y}_0=\lim_{n\to \infty}\tilde{\beta}^n_0(0)$. Then, for any $\varepsilon>0$ there exists $N\geq 0$ so that
$$d^1(y_0, \beta^N_0(0))<\varepsilon \quad \mathrm{and} \quad d^1(\tilde{y}_0, \tilde{\beta}^N_0(0))<\varepsilon$$ 
hold. For this $N$, we choose $\delta>0$ sufficiently small so that $d^1(x_i, \tilde{x}_i)<\delta$ ($0\leq i\leq N$) implies $d^1(h^1(x_i), h^1(\tilde{x}_i))<\varepsilon$ ($0\leq i\leq N$) by the continuity of $h^1$. In particular, $d^1(\beta^0_N(0), \tilde{\beta}^0_N(0))=d^1(h^1(x_N), h^1(\tilde{x}_N)) <\varepsilon$. The definition of an expanding system implies $d^1(\beta^1_{N-1}(0), \tilde{\beta}^1_{N-1}(0))<\varepsilon/\lambda$. Applying this repeatedly, one has 
$$d^1(\beta^N_0(0), \tilde{\beta}^N_0(0)) < \frac{\varepsilon}{\lambda^N} < \varepsilon.$$ 
Combining this with the above two estimates gives $d^1(y_0, \tilde{y}_0)<3 \varepsilon$.
\end{proof}

\end{section}

\begin{section}{Shadowing and its uniqueness for hyperbolic systems}
\label{sec:hyp-shadowing}

In this section we consider the case of bi-infinite orbits. When we use the term orbit in this section without further modification, we mean a bi-infinite orbit.

\subsection{Homotopy shadowing theorem}
\label{sub:hyp-shadowing}

This subsection is devoted to the proof of the following theorem. A corresponding statement in the case of expanding systems can be found in Theorem~\ref{thm:exp-shadowing}.

\begin{thm} 
\label{thm:hyp-shadowing}
Every homotopy pseudo-orbit $(z, \alpha)$ of a hyperbolic system $\iota, f : X^1\to X^0$ is homotopic to an orbit.
\end{thm}

For the proof of the theorem, we first need the following lemma. Let $V$ be a vertical-like submanifold of degree one in $X^{n+1}$ and $H$ be a horizontal-like submanifold of degree one in $X^{n+1}$. Then, $f(H)$ becomes a horizontal-like submanifold of degree $d$ in $X^n$ and $\iota(V)$ becomes a vertical-like submanifold of degree one in $X^n$ by Lemma~\ref{lmm:submfd} and Proposition~\ref{prp:induction}. Suppose that points $p_h\in H$ and $p_v\in V$ are given. Let $\gamma$ be a path from $f(p_h)$ to $\iota(p_v)$ in $X^n$. We know that the intersection $f(H)\cap \iota(V)$ consists of $d$ distinct points by Lemma~\ref{lmm:intersection} and Proposition~\ref{prp:induction}.

\begin{lmm} 
\label{lmm:unique-intersect}
There is a unique point $p\in f(H)\cap \iota(V)$ for which there exist a path $u$ from $f(p_h)$ to $p$ in $f(H)$ and a path $s$ from $p$ to $\iota(p_v)$ in $\iota(V)$ so that the concatenation $u\cdot s$ is homotopic to $\gamma$ in $X^n$. The paths are unique up to homotopy relative to endpoints. 
\end{lmm}

\begin{proof}
We may assume that $\iota(V)$ is a straight vertical submanifold in $X^n$. Then, $\pi_x^n(\iota(V))$ becomes one point. Since $\pi_x^n : f(H)\to M^n_x$ is a covering, one can take a unique lift $u$ of $\pi_x^n(\gamma)$ in $f(H)$ starting from $f(p_h)$. Let $p\in f(H)\cap \iota(V)$ be the other endpoint of $u$. Then, for any path $s$ in $\iota(V)$ from $p$ to $\iota(p_v)$, we see that $u\cdot s\sim \gamma$.

Choose two distinct points $p_1$, $p_2\in f(H)\cap \iota(V)$ among $d$ points. Let $v$ be any path from $p_1$ to $p_2$ in $\iota(V)$, and let $h$ be any path from $p_2$ to $p_1$ in $f(H)$. Then, the concatenation of the two paths $v \cdot h$ (which forms a closed path in $X^n$) is non-trivial in the fundamental group $\pi_1(X^n)$. Suppose that $s\cdot v\cdot h\cdot u\sim \gamma$. Then, $s\cdot u\sim s\cdot v\cdot h\cdot u$ and this implies that $1\sim v\cdot h$, which is a contradiction. Thus, the uniqueness of the point $p$ follows.
\end{proof}

\begin{proof}[Proof of Theorem~\ref{thm:hyp-shadowing}] 
Given a homotopy pseudo-orbit $(z, \alpha)$ we replace it by another homotopy pseudo-orbit $(z', \alpha')$ which is homotopic to $(z, \alpha)$ where the connecting homotopies $\alpha'$ become shorter. We show that this process converges. The limit will be an actual orbit which ``shadows'' the original homotopy pseudo-orbit $(z, \alpha)$.

\vspace{0.3cm}

\noindent
{\bf Step I:} {\it Starting Condition.}

\vspace{0.1cm}

Suppose that a homotopy pseudo-orbit $(z, \alpha)$ of $\iota, f : X^1\to X^0$ is given, where $z=(z_i)_{i\in \ZZ}$ is a bi-infinite sequence of points $z_i$ in $X^1_i\equiv X^1$ and $\alpha=(\alpha_i)_{i\in \ZZ}$ is a bi-infinite sequence of paths $\alpha_i$ in $X^0_i$ so that $l^0(\alpha_i)\leq C'$. Set $z^1_i\equiv z_i$ and $\alpha^0_i\equiv \alpha_i$. We let $H^1_i$ be the connected component of $\iota^{-1}(M^0_x(\pi^0_y(\iota (z^1_i))))$ containing $z^1_i$, and we let $V^1_i$ be the connected component of $f^{-1}(M^0_y(\pi^0_x(f(z^1_i))))$ containing $z^1_i$.

 \begin{figure}
\setlength{\unitlength}{1mm}
\begin{picture}(100, 150)(20, 0)

\put(-5, 0){\line(1, 0){40}}
\put(35, 0){\line(0, 1){40}}
\put(-5, 0){\line(0, 1){40}}
\put(-5, 40){\line(1, 0){40}}

\put(55, 0){\line(1, 0){40}}
\put(95, 0){\line(0, 1){40}}
\put(55, 0){\line(0, 1){40}}
\put(55, 40){\line(1, 0){40}}

\put(115, 0){\line(1, 0){40}}
\put(155, 0){\line(0, 1){40}}
\put(115, 0){\line(0, 1){40}}
\put(115, 40){\line(1, 0){40}}

\put(25, 50){\line(1, 0){40}}
\put(65, 50){\line(0, 1){40}}
\put(25, 50){\line(0, 1){40}}
\put(25, 90){\line(1, 0){40}}

\put(85, 50){\line(1, 0){40}}
\put(125, 50){\line(0, 1){40}}
\put(85, 50){\line(0, 1){40}}
\put(85, 90){\line(1, 0){40}}

\put(55, 100){\line(1, 0){40}}
\put(95, 100){\line(0, 1){40}}
\put(55, 100){\line(0, 1){40}}
\put(55, 140){\line(1, 0){40}}

\linethickness{1pt}

\put(-5, 15){\line(1, 0){40}}
\put(25, 65){\line(1, 0){40}}
\put(55, 115){\line(1, 0){40}}
\put(125, 0){\line(0, 1){40}}
\qbezier(120, 50)(115, 85)(110, 90)
\qbezier(60, 100)(65, 135)(70, 140)
\qbezier(90, 0)(85, 35)(80, 40)
\qbezier(95, 10)(75, 10)(65, 20)
\qbezier(65, 20)(60, 25)(55, 25)
\qbezier(95, 30)(75, 30)(65.5, 20.5)
\qbezier(64.5, 19.5)(60, 15)(55, 15)

\put(62.5, 115){\circle*{1.7}}
\put(55, 65){\circle*{1.7}}
\put(32.5, 65){\circle*{1.7}}
\put(25, 15){\circle*{1.7}}
\put(88.5, 10.3){\circle*{1.7}}
\put(72, 15){\circle*{1.7}}
\put(83, 34){\circle*{1.7}}
\put(125, 25){\circle*{1.7}}
\put(118.5, 60.3){\circle*{1.7}}
\put(113, 84){\circle*{1.7}}

\put(34, 49){\vector(-1, -1){8}}
\put(56, 49){\vector(1, -1){8}}
\put(94, 49){\vector(-1, -1){8}}
\put(116, 49){\vector(1, -1){8}}
\put(64, 99){\vector(-1, -1){8}}
\put(86, 99){\vector(1, -1){8}}

\small
\put(33, 44){$\iota$}
\put(63, 44.5){$f$}
\put(93, 44){$\iota$}
\put(123, 44.5){$f$}
\put(63, 94){$\iota$}
\put(93, 94.5){$f$}
\put(22, 10){$\iota(a^1_{i-1})$}
\put(-2, 25){$M_x(\pi_y(\iota(a^1_{i-1})))$}
\put(8, 23){\vector(-1, -2){4}}
\put(28, 68){$f^{-1}(\zeta^0_i)$}
\put(54, 68){$a^1_{i-1}$}
\put(84, 5){$\zeta^0_i$}
\put(16, 64){$H^1_{i-1}$}
\put(96, 9){$f(H^1_{i-1})$}
\put(73, 42){$\iota(V^1_i)$}
\put(108, 92){$V^1_i$}
\put(128, 5){$M_y(\pi_x(f(a^1_i)))$}
\put(136, 9){\vector(-2, 1){10}}
\put(127, 24){$f(a^1_i)$}
\put(73, 33){$\iota(a^1_i)$}
\put(62, 10){$f(a^1_{i-1})$}
\put(108, 83){$a^1_i$}
\put(104, 59){$\iota^{-1}(\zeta^0_i)$}
\put(63.1, 110.2){$a^2_{i-1}=(f\iota)^{-1}(\zeta^0_i)$}
\put(71.5, 105.2){$=(\iota f)^{-1}(\zeta^0_i)$}
\put(46, 114){$H^2_{i-1}$}
\put(68, 142){$V^2_{i-1}$}
\put(50, 142){$X^2_{i-1}$}
\put(20, 92){$X^1_{i-1}$}
\put(-10, 42){$X^0_{i-1}$}
\put(51, 42){$X^0_i$}
\put(124, 92){$X^1_i$}
\put(151, 42){$X^1_{i+1}$}

\end{picture}
\vspace{1cm}
\begin{center}
Figure 7. Shadowing points for homotopy shadowing in the hyperbolic case.
\end{center}
\end{figure}

Thanks to Lemma~\ref{lmm:intersection} we know that $f(H^1_{i-1})$ and $\iota(V^1_i)$ have $d$ distinct intersection points in $X^0_i$. By Lemma~\ref{lmm:unique-intersect}, there exists a unique point $\zeta^0_i$ among them so that the concatenation of a path $u_i^0$ from $f(z^1_{i-1})$ to $\zeta^0_i$ in $f(H^1_{i-1})$ and a path $s_i^0$ from $\zeta^0_i$ to $\iota(z^1_i)$ in $\iota(V^1_i)$ is homotopic to $\alpha^0_i$ in $X^0_i$. Let us put $z^2_i\equiv (f\iota)^{-1}(\zeta^0_{i+1})=(\iota f)^{-1}(\zeta^0_{i+1})\in X^2_i$. Then, $\tilde{s}^1_i\equiv \iota^{-1}s^0_i$ becomes a path from $\iota^{-1}(\zeta^0_i)=f(z^2_{i-1})$ to $z^1_i$ in $V^1_i$ and $\tilde{u}^1_i\equiv f^{-1}u^0_{i+1}$ becomes a path from $z^1_i$ to $f^{-1}(\zeta^0_{i+1})=\iota(z^2_i)$ in $H^1_i$ (see Figure 7). One may choose a representative for $s^0_i$ in its homotopy class so that
$$l^0_y(\iota\tilde{s}^1_i)=l^0_y(s^0_i)\leq C+1$$
holds by the fact that $M^0_y$ is simply connected and ($\ast$). Since $\alpha^0_i$ is homotopic to $u^0_i\cdot s^0_i$, we see that $u^0_i$ is homotopic to $\alpha^0_i\cdot (s^0_i)^{-1}$. Thus, one may choose a representative for $u^0_i$ in its homotopy class so that 
$$l^0_x(f\tilde{u}^1_i)=l^0_x(u^0_{i+1})\leq l^0_x(\alpha^0_{i+1})+l^0_x(s^0_{i+1})\leq l^0(\alpha^0_{i+1})+l^0_y(s^0_{i+1})\leq C'+C+1$$ 
holds, where we used the fact that for any path $\gamma$ in a vertical-like submanifold $V$ in $X^0$ we have $l^0_x(\gamma)\leq l^0_y(\gamma)$.

 \begin{figure}
\setlength{\unitlength}{1mm}
\begin{picture}(100, 130)(25, 0)

\put(50, 70){\line(1, 0){50}}
\put(100, 70){\line(0, 1){50}}
\put(50, 120){\line(1, 0){50}}
\put(50, 70){\line(0, 1){50}}

\put(10, 0){\line(1, 0){50}}
\put(60, 0){\line(0, 1){50}}
\put(10, 0){\line(0, 1){50}}
\put(10, 50){\line(1, 0){50}}

\put(90, 0){\line(1, 0){50}}
\put(140, 0){\line(0, 1){50}}
\put(90, 0){\line(0, 1){50}}
\put(90, 50){\line(1, 0){50}}

\linethickness{0.5pt}

\put(50, 105){\line(1, 0){50}}
\qbezier(95, 70)(90, 110)(80, 120)
\qbezier(55, 0)(50, 40)(40, 50)
\qbezier(60, 10)(30, 10)(20, 20)
\qbezier(20, 20)(15, 25)(10, 25)
\qbezier(60, 30)(30, 30)(20.5, 20.5)
\qbezier(19.5, 19.5)(15, 15)(10, 15)
\qbezier(97, 0)(102, 40)(112, 50)
\qbezier(90, 10)(120, 10)(130, 20)
\qbezier(130, 20)(135, 25)(140, 25)
\qbezier(90, 30)(120, 30)(129.5, 20.5)
\qbezier(130.6, 19.4)(135, 15)(140, 15)

\linethickness{1.3pt}

\qbezier(63, 105)(72, 105)(88, 105)
\qbezier(88, 105)(92, 91)(93.5, 80.3)
\qbezier(48, 35)(52, 21)(53.5, 10.3)
\qbezier(53.5, 10.3)(43, 10)(33, 13)
\qbezier(98.5, 10.3)(102, 31)(107, 42)
\qbezier(98.5, 10.3)(120, 11)(130, 20)
\qbezier(130, 20)(132, 22)(134, 23)

\put(88, 105){\circle*{1.7}}
\put(48, 35){\circle*{1.7}}
\put(63, 105){\circle*{1.7}}
\put(53.5, 10.3){\circle*{1.7}}
\put(93.5, 80.3){\circle*{1.7}}
\put(33, 13){\circle*{1.7}}
\put(98.5, 10.3){\circle*{1.7}}
\put(134, 23){\circle*{1.7}}
\put(107, 42){\circle*{1.7}}

\put(104, 91){\vector(-3, 1){12}}
\put(47, 116){\vector(3, -1){29}}
\put(60, 69){\vector(-1, -2){9}}
\put(90, 69){\vector(1, -2){9}}
\put(10, 69){\vector(1, -2){9}}
\put(140, 69){\vector(-1, -2){9}}

\small
\put(23, 115){$\tilde{u}^1_i\equiv f^{-1}u^0_{i+1}$}
\put(105, 90){$\tilde{s}^1_i\equiv \iota^{-1}s^0_i$}
\put(89, 107){$a^1_i$}
\put(77, 122){$V^1_i$}
\put(101, 104){$H^1_i$}
\put(34, 52){$\iota(V^1_i)$}
\put(61, 9){$f(H^1_{i-1})$}
\put(108, 52){$\iota(V^1_{i+1})$}
\put(78, 29){$f(H^1_i)$}
\put(40, 6){$u^0_i$}
\put(53, 22){$s^0_i$}
\put(114, 8){$u^0_{i+1}$}
\put(92.5, 22){$s^0_{i+1}$}
\put(48, 13){$\zeta^0_i$}
\put(100, 13){$\zeta^0_{i+1}$}
\put(20, 8){$f(a^1_{i-1})$}
\put(38, 34){$\iota(a^1_i)$}
\put(127, 26){$f(a^1_i)$}
\put(93.5, 41){$\iota(a^1_{i+1})$}
\put(79, 79){$\iota^{-1}(\zeta^0_i)$}
\put(55, 100){$f^{-1}(\zeta^0_{i+1})$}
\put(57, 59){$\iota$}
\put(90, 59){$f$}
\put(137, 59){$\iota$}
\put(10, 59){$f$}

\end{picture}
\vspace{1cm}
\begin{center}
Figure 8. Lifts of paths for homotopy shadowing in the hyperbolic case.
\end{center}
\end{figure}

We define $V^2_{i-1}$ to be the connected component of $f^{-1}(V^1_i)$ containing $z^2_{i-1}$ and $H^2_{i-1}$ to be the connected component of $\iota^{-1}(H^1_{i-1})$ containing $z^2_{i-1}$ (see Figure 7 again). Consider the concatenation $\alpha^1_i\equiv \tilde{s}^1_i\cdot \tilde{u}^1_i$. Then, $\alpha^1_i$ is a path from $f(z^2_{i-1})$ to $\iota(z^2_i)$ via $z^1_i$ in $X^1_i$ (see Figure 8).

\vspace{0.3cm}

\noindent
{\bf Step II:} {\it Induction Step.}

\vspace{0.1cm}

Now we let $n\geq 1$. Suppose we are given a sequence of points $(z^n_i)_{i\in \ZZ}$ in $X^n_i\equiv X^n$, a sequence of horizontal-like submanifolds $(H^n_i)_{i\in \ZZ}$ and vertical-like submanifolds $(V^n_i)_{i\in \ZZ}$ in $X^n_i$ through $z^n_i$. We moreover assume that a sequence of paths $(\alpha^{n-1}_i)_{i\in \ZZ}$ from $f(z^n_{i-1})$ to $\iota(z^n_i)$ is defined. Now, we try to define $(z^{n+1}_i)_{i\in \ZZ}$, $(H^{n+1}_i)_{i\in \ZZ}$, $(V^{n+1}_i)_{i\in \ZZ}$ and $(\alpha^n_i)_{i\in \ZZ}$. 

Thanks to Lemma~\ref{lmm:intersection} and Proposition~\ref{prp:induction}  we know that $f(H^n_{i-1})$ and $\iota(V^n_i)$ has $d$ distinct intersection points in $X^{n-1}_i$. By Lemma~\ref{lmm:unique-intersect}, there exists a unique point $\zeta^{n-1}_i$ among the $d$ the intersection points as well as two paths $u^{n-1}_i$ from $\iota^{-1}(\zeta^{n-2}_i)$ to $\zeta^{n-1}_i$ in $f(H^n_{i-1})$ and $s^{n-1}_i$ from $\zeta^{n-1}_i$ to $f^{-1}(\zeta^{n-2}_{i+1})$ in $\iota(V^n_i)$ so that the concatenation $u^{n-1}_i\cdot s^{n-1}_i$ is homotopic to $\alpha^{n-1}_i$. We put $\tilde{u}^n_i\equiv f^{-1}u^{n-1}_{i+1}$ and $\tilde{s}^n_i\equiv \iota^{-1}s^{n-1}_i$. 

\begin{lmm} 
\label{lmm:s-initial}
One may choose a representative for $s^{n-1}_i$ in its homotopy class so that
$$l_y^0(\iota^{n-1}(s^{n-1}_i))=l_y^0(\iota^n(\tilde{s}^n_i))\leq C_0\equiv C+1$$ 
holds for $n\geq 1$.
\end{lmm}

\begin{proof}
This is due to the fact that $M^0_y$ is simply connected and ($\ast$) as in Step I.
\end{proof}

By using this lemma we also have the following crucial claim.

\begin{lmm} 
\label{lmm:u-initial}
One may choose a representative for $u^{n-1}_i$ in its homotopy class so that
$$l^0_x(f^{n-1}(u^{n-1}_i))=l^0_x(f^n(\tilde{u}^n_{i-1}))\leq C_1\equiv \frac{2C_0}{\lambda-1}+(C'+C+1)$$ 
holds for $n\geq 1$.
\end{lmm}

\begin{proof}
Recall that one has already seen $l^0_x(u^0_i)\leq l^0_x(\alpha^0_i)+l^0_x(s^0_i)\leq C'+C+1$. Since $f\tilde{u}^1_{i-1}=u^0_i$, this shows the claim for $n=1$.

By definition, the two concatenated paths $u^{n-1}_i\cdot s^{n-1}_i$ and $\tilde{s}^{n-1}_i\cdot \tilde{u}^{n-1}_i$ are homotopic in $X^{n-1}_i$, so $f^{n-1}(u^{n-1}_i\cdot s^{n-1}_i)=f^{n-1}u^{n-1}_i\cdot f^{n-1}s^{n-1}_i$ and $f^{n-1}(\tilde{s}^{n-1}_i\cdot \tilde{u}^{n-1}_i)=f^{n-1}\tilde{s}^{n-1}_i\cdot f^{n-1}\tilde{u}^{n-1}_i$ are homotopic in $X^0_{i+n-1}$ for $n\geq 1$. Thus, 
\begin{align*}
& f^{n-1}u^{n-1}_{i} \cdot f^{n-1}s^{n-1}_{i} \cdot f^{n-2}s^{n-2}_{i+1} \cdots fs^1_{i+n-2} \\
\sim \ & f^{n-1}\tilde{s}^{n-1}_{i}\cdot f^{n-1}\tilde{u}^{n-1}_{i}\cdot
    f^{n-2}s^{n-2}_{i+1}\cdots fs^1_{i+n-2} \\
= \ & f^{n-1}\tilde{s}^{n-1}_{i}\cdot f^{n-2}u^{n-2}_{i+1}\cdot
    f^{n-2}s^{n-2}_{i+1}\cdots fs^1_{i+n-2} \\
\sim \ & f^{n-1}\tilde{s}^{n-1}_{i}\cdot f^{n-2}\tilde{s}^{n-2}_{i+1}\cdot f^{n-2}\tilde{u}^{n-2}_{i+1}\cdot f^{n-3}s^{n-3}_{i+2}\cdots fs^1_{i+n-2} \\
= \ & f^{n-1}\tilde{s}^{n-1}_{i}\cdot f^{n-2}\tilde{s}^{n-2}_{i+1}\cdot f^{n-3}
  u^{n-3}_{i+2}\cdot f^{n-3}s^{n-3}_{i+2}\cdots fs^1_{i+n-2} \\
\sim \ & \cdots \\
\sim \ & f^{n-1}\tilde{s}^{n-1}_{i}\cdots f\tilde{s}^1_{i+n-2}\cdot u^0_{i+n-1}.
\end{align*}
In particular,
$$(f^{n-1}\tilde{s}^{n-1}_i \cdots f\tilde{s}^1_{i+n-2})\cdot u^0_{i+n-1}\cdot (f^{n-1}s^{n-1}_i \cdots fs^1_{i+n-2})^{-1}\sim f^{n-1}u^{n-1}_{i}=f^n\tilde{u}^n_{i-1}$$
in $X^0_{i+n-1}$. Since $s^{n-1}_i$ is taken so that $l^0_y(\iota^n\tilde{s}^n_i)\leq C_0\equiv C+1$ by the previous Lemma~\ref{lmm:s-initial}, it follows that
\begin{align*}
     & \ l^0_x(f^n \tilde{u}^n_{i-1}) \\
\leq & \ l^0_x((f^{n-1}\tilde{s}^{n-1}_i\cdots f\tilde{s}^1_{i+n-2})\cdot 
         u^0_{i+n-1}\cdot (f^{n-1}s^{n-1}_i\cdots fs^1_{i+n-2})^{-1})\\
=    & \ l^0_x(f^{n-1}\tilde{s}^{n-1}_i) + \cdots +l^0_x(f\tilde{s}^1_{i+n-2})+
         l^0_x(u^0_{i+n-1})+l^0_x(f^{n-1}s^{n-1}_i)+\cdots 
         +l^0_x(fs^1_{i+n-2}) \\
\leq & \ l^0_y(f^{n-1}\tilde{s}^{n-1}_i) + \cdots +l^0_y(f\tilde{s}^1_{i+n-2})+
         l^0_x(u^0_{i+n-1})+l^0_y(f^{n-1}s^{n-1}_i)+\cdots 
         +l^0_y(fs^1_{i+n-2}) \\
\leq & \ \frac{C_0}{\lambda^{n-1}} + \cdots + \frac{C_0}{\lambda}
         +(C'+C+1) +
         \frac{C_0}{\lambda^{n-1}} + \cdots + \frac{C_0}{\lambda} \\
\leq & \ \frac{2C_0}{\lambda-1}+(C'+C+1)
\end{align*}
holds, where we used the fact that for any path $\gamma$ in a vertical-like submanifold $V$ in $X^0$ we have $l^0_x(\gamma)\leq l^0_y(\gamma)$. Thus one obtains $l^0_x(f^n\tilde{u}^n_{i-1})\leq C_1$ for some constant $C_1\geq 0$ independent of $i\in\ZZ$ and $n\geq 1$. This completes the proof.
\end{proof}

\begin{cor} 
\label{cor:hyp-estimate}
For all $m\geq 0$ we have
$$l^0((\iota f)^m\tilde{u}^{2m}_i)\leq \frac{2C_1}{\lambda^m} \quad {\it and} \quad l^0((\iota f)^m\iota \tilde{s}^{2m+1}_i)\leq \frac{2C_0}{\lambda^m}.$$
\end{cor}

\begin{proof}
Lemmas~\ref{lmm:u-initial} and \ref{lmm:s-initial} together with Corollary~\ref{cor:expansion} imply 
\begin{equation*}
l^0_x((\iota f)^m\tilde{u}^{2m}_i)\leq \frac{C_1}{\lambda^m}\quad {\it and} \quad l^0_y((\iota f)^m\iota \tilde{s}^{2m+1}_i)\leq \frac{C_0}{\lambda^m}.
\end{equation*}
Since $(\iota f)^m\tilde{u}^{2m}_i$ is contained in a horizontal-like submanifold and $(\iota f)^m\iota \tilde{s}^{2m+1}_i$ is contained in a vertical-like submanifold, we have $l^0_y((\iota f)^m\tilde{u}^{2m}_i)\leq l^0_x((\iota f)^m\tilde{u}^{2m}_i)$ and $l^0_x((\iota f)^m\iota \tilde{s}^{2m+1}_i)\leq l^0_y((\iota f)^m\iota \tilde{s}^{2m+1}_i)$. The conclusion then follows.
\end{proof}

Let us put $z^{n+1}_{i-1}\equiv (f\iota)^{-1}(\zeta^{n-1}_i)=(\iota f)^{-1}(\zeta^{n-1}_i)$. Moreover, we define $V^{n+1}_{i-1}$ to be the connected component of $f^{-1}(V^n_i)$ containing $z^{n+1}_{i-1}$ and $H^{n+1}_{i-1}$ to be the connected component of $\iota^{-1}(H^n_{i-1})$ containing $z^{n+1}_{i-1}$. Consider the concatenation $\alpha^n_i\equiv \tilde{s}^n_i\cdot \tilde{u}^n_i$. Then, $\alpha^n_i$ becomes a path from $f(z^{n+1}_{i-1})$ to $\iota(z^{n+1}_i)$ via $z^n_i$.

We can continue the construction above to obtain a sequence of points $(z^{n+1}_i)_{i\in \ZZ}$, a sequence of horizontal-like submanifolds $(H^{n+1}_i)_{i\in \ZZ}$, a sequence of vertical-like submanifolds $(V^{n+1}_i)_{i\in \ZZ}$ and a sequence of paths $(\alpha^n_i)_{i\in \ZZ}$ for each $n\geq 0$. This inductive construction also defines a sequence of homotopies between homotopy pseudo-orbits. More precisely,

\begin{lmm} 
\label{lmm:hom-iota}
The sequence of paths $(\tilde{u}^n_i)_{i\in \ZZ}$ gives a homotopy between two homotopy pseudo-orbits $((z^n_i)_{i\in \ZZ}, (\alpha^{n-1}_i)_{i\in \ZZ})$ and $((\iota z^{n+1}_i)_{i\in \ZZ}, (\iota \alpha^n_i)_{i\in \ZZ})$ for each $n\geq 1$.
\end{lmm}

\begin{proof}
Put $\delta_i=\tilde{u}^n_i$, $\gamma_i=\alpha^{n-1}_i=\tilde{s}^{n-1}_i\cdot \tilde{u}^{n-1}_i$ and $\gamma'_i=\iota(\tilde{s}^n_i\cdot \tilde{u}^n_i)=\iota(\tilde{s}^n_i)\cdot \iota(\tilde{u}^n_i)$. Then, $\delta_i(0)=\tilde{u}^n_i(0)=z^n_i$ and $\delta_i(1)=\tilde{u}^n_i(1)=\iota(z^{n+1}_i)$. Moreover, $\gamma_i\cdot \iota(\delta_i)= \tilde{s}^{n-1}_i\cdot \tilde{u}^{n-1}_i\cdot \iota(\tilde{u}^n_i) \sim u^{n-1}_i\cdot s^{n-1}_i \cdot \iota(\tilde{u}^n_i) = f(\tilde{u}^n_{i-1})\cdot \iota(\tilde{s}^n_i)\cdot \iota(\tilde{u}^n_i) = f(\delta_{i-1})\cdot \gamma'_i$. This completes the proof.
\end{proof}

Similarly one can show

\begin{lmm} 
\label{lmm:hom-f}
The sequence of paths $(\tilde{s}^n_i)_{i\in \ZZ}$ gives a homotopy between two homotopy pseudo-orbits $((z^n_i)_{i\in \ZZ}, (\alpha^{n-1}_i)_{i\in \ZZ})$ and $((f z^{n+1}_{i-1})_{i\in \ZZ}, (f\alpha^n_{i-1})_{i\in \ZZ})$ for each $n\geq 1$.
\end{lmm}

 \begin{figure}
\setlength{\unitlength}{1mm}
\begin{picture}(100, 150)(25, 0)

\put(25, 0){\line(1, 0){40}}
\put(65, 0){\line(0, 1){40}}
\put(25, 0){\line(0, 1){40}}
\put(25, 40){\line(1, 0){40}}
\put(85, 30){\line(1, 0){40}}
\put(125, 30){\line(0, 1){40}}
\put(85, 30){\line(0, 1){40}}
\put(85, 70){\line(1, 0){40}}
\put(25, 60){\line(1, 0){40}}
\put(65, 60){\line(0, 1){40}}
\put(25, 60){\line(0, 1){40}}
\put(25, 100){\line(1, 0){40}}
\put(85, 90){\line(1, 0){40}}
\put(125, 90){\line(0, 1){40}}
\put(85, 90){\line(0, 1){40}}
\put(85, 130){\line(1, 0){40}}

\put(83, 43){\vector(-1, -1){16}}
\put(67, 72){\vector(1, -1){16}}
\put(83, 103){\vector(-1, -1){16}}
\put(67, 132){\vector(1, -1){16}}

\put(30, 5){\circle*{1.5}}
\put(45, 5){\circle*{1.5}}
\put(45, 20){\circle*{1.5}}
\put(53, 20){\circle*{1.5}}
\put(53, 28){\circle*{1.5}}
\put(60, 35){\circle*{1.5}}
\put(55, 30){\circle*{0.5}}
\put(56.5, 31.5){\circle*{0.5}}
\put(58, 33){\circle*{0.5}}
\put(105, 35){\circle*{1.5}}
\put(105, 50){\circle*{1.5}}
\put(45, 80){\circle*{1.5}}
\put(53, 80){\circle*{1.5}}
\put(113, 110){\circle*{1.5}}
\put(113, 118){\circle*{1.5}}
\put(75, 137){\circle*{0.8}}
\put(75, 139){\circle*{0.8}}
\put(75, 141){\circle*{0.8}}

\put(74, 38){$\iota$}
\put(75, 67){$f$}
\put(74, 98){$\iota$}
\put(75, 127){$f$}
\put(43, 43){$X^1_i$}
\put(43, 103){$X^3_{i-1}$}
\put(103, 73){$X^2_i$}
\put(103, 133){$X^4_{i-1}$}
\put(28, 8){$z^1_i$}
\put(47, 4){$\iota z^2_i$}
\put(33, 20){$\iota f z^3_{i-1}$}
\put(52, 16){$\iota f \iota z^4_{i-1}$}
\put(37, 30){$\iota f\iota f z^5_{i-2}$}
\put(60, 44){$w_i^{(\infty)}$}
\put(62.5, 42.5){\vector(-1, -3){2}}
\put(23, 22){\vector(1, -1){16}}
\put(67, 6){\vector(-3, 1){21}}
\put(23, 34){\vector(2, -1){26.5}}
\put(67, 20){\vector(-3, 1){13}}

\small
\put(18, 22){$\tilde{u}^1_i$}
\put(68, 4){$\iota \tilde{s}^2_i$}
\put(11, 34){$\iota f\tilde{u}^3_{i-1}$}
\put(68, 18){$\iota f\iota \tilde{s}^4_{i-1}$}
\put(100, 42){$\tilde{s}^2_i$}
\put(46, 83){$\tilde{u}^3_{i-1}$}
\put(105, 113){$\tilde{s}^4_{i-1}$}

\put(107, 34){$z^2_i$}
\put(107, 49){$fz^3_{i-1}$}
\put(40, 75){$z^3_{i-1}$}
\put(53, 75){$\iota z^4_{i-1}$}
\put(114, 106){$z^4_{i-1}$}
\put(114, 121){$fz^5_{i-2}$}

\linethickness{1pt}
\put(30, 5){\line(1, 0){15}}
\put(45, 5){\line(0, 1){15}}
\put(45, 20){\line(1, 0){8}}
\put(53, 20){\line(0, 1){8}}
\put(105, 35){\line(0, 1){15}}
\put(45, 80){\line(1, 0){8}}
\put(113, 110){\line(0, 1){8}}

\end{picture}
\vspace{1cm}
\begin{center}
Figure 9. Pushing down homotopies along a zigzag path.
\end{center}
\end{figure}

By these two Lemmas, we get a sequence of homotopy pseudo-orbits which are successively homotopic:
\begin{align*}
((z^1_i)_{i\in \ZZ}, (\alpha^0_i)_{i\in \ZZ})
& \sim ((\iota z^2_i)_{i\in \ZZ}, (\iota \alpha^1_i)_{i\in \ZZ}) \\
& \sim ((\iota f z^3_{i-1})_{i\in \ZZ}, (\iota f\alpha^2_{i-1})_{i\in \ZZ}) \\
& \sim ((\iota f\iota z^4_{i-1})_{i\in \ZZ}, (\iota f\iota \alpha^3_{i-1})_{i\in \ZZ}) \\
& \sim ((\iota f\iota f z^5_{i-2})_{i\in \ZZ}, (\iota f\iota f\alpha^4_{i-2})_{i\in \ZZ}) \\
& \sim ((\iota f\iota f\iota z^6_{i-2})_{i\in \ZZ}, (\iota f\iota f\iota\alpha^5_{i-2})_{i\in \ZZ}) \\
& \sim \cdots,
\end{align*}
where the first homotopy is given by $(\tilde{u}^1_i)_{i\in \ZZ}$, the second homotopy is given by $(\iota \tilde{s}^2_i)_{i\in \ZZ}$, the third homotopy is given by $(\iota f\tilde{u}^3_{i-1})_{i\in \ZZ}$, the fourth homotopy is given by $(\iota f \iota \tilde{s}^4_{i-1})_{i\in \ZZ}$, and so on (see Figure 9). We let
$$
w_i^{(1)}\equiv z^1_i, \ 
w_i^{(2)}\equiv \iota z^2_i, \ 
w_i^{(3)}\equiv \iota f z^3_{i-1}, \ 
w_i^{(4)}\equiv \iota f \iota z^4_{i-1}, \ 
w_i^{(5)}\equiv \iota f \iota f z^5_{i-2}, \ 
\cdots$$
be the sequence of points in $X^1_i$ and
$$
\delta^{(1)}_i\equiv \alpha^0_i, \ 
\delta^{(2)}_i\equiv \iota \alpha^1_i, \ 
\delta^{(3)}_i\equiv \iota f \alpha^2_{i-1}, \ 
\delta^{(4)}_i\equiv \iota f\iota \alpha^3_{i-1}, \ 
\delta^{(5)}_i\equiv \iota f\iota f \alpha^4_{i-2}, \
\cdots$$
be the sequence of paths in $X^0_i$ in the sequence of homotopy pseudo-orbits above. It then follows from the inductive construction that $(z^1, \alpha^0)=((z^1_i)_{i\in \ZZ}, (\alpha^0_i)_{i\in\ZZ})$ is homotopic to $(w^{(m)}, \delta^{(m)})=((w^{(m)}_i)_{i\in \ZZ}, (\delta^{(m)}_i)_{i\in\ZZ})$ for any $m\geq 1$.

\vspace{0.3cm}

\noindent
{\bf Step III:} {\it Convergence.}

\vspace{0.1cm}

We build a homotopy by concatenating an infinite sequence of homotopies. First recall that the homotopy pseudo-orbit $(w^{(1)}, \delta^{(1)})$ is homotopic to $(w^{(2)}, \delta^{(2)})$ by $(\tilde{u}^1_i)_{i\in\ZZ}$, $(w^{(2)}, \delta^{(2)})$ is homotopic to $(w^{(3)}, \delta^{(3)})$ by $(\iota \tilde{s}^2_i)_{i\in\ZZ}$, $(w^{(3)}, \delta^{(3)})$ is homotopic to $(w^{(4)}, \delta^{(4)})$ by $(\iota f\tilde{u}^3_{i-1})_{i\in\ZZ}$, $(w^{(4)}, \delta^{(4)})$ is homotopic to $(w^{(5)}, \delta^{(5)})$ by $(\iota f \iota \tilde{s}^4_{i-1})_{i\in\ZZ}$, etc. Let $I_n=[1-\frac{1}{2^n}, 1-\frac{1}{2^{n+1}}]$. By rescaling the domains, one may assume that the domain of $\tilde{u}^1_i$ is $I_0$, the domain of $\iota \tilde{s}^2_i$ is $I_1$, the domain of $\iota f\tilde{u}^3_{i-1}$ is $I_2$, the domain of definition of $\iota f \iota \tilde{s}^4_{i-1}$ is $I_3$, etc. Then, the concatenation of the homotopies:
$$(\tilde{u}^1_i) \cdot (\iota \tilde{s}^2_i) \cdot (\iota f\tilde{u}^3_{i-1}) \cdot (\iota f \iota \tilde{s}^4_{i-1}) \cdot (\iota f \iota f \tilde{u}^5_{i-2}) \cdots$$
appearing above is defined on $[0, 1)$ and continuously extends to $[0, 1]$ by Corollary~\ref{cor:hyp-estimate}. In particular, the sequence of points $w_i^{(m)}$ ($m=1, 2, \ldots$) converges to some $w^{(\infty)}_i$ in $X^1_i$.

Since $\alpha^n_i\equiv \tilde{s}^n_i\cdot \tilde{u}^n_i$, one can again apply Corollary~\ref{cor:hyp-estimate} to show that $l^0(\delta^{(m)}_i)\to 0$ as $m\to \infty$. It follows that $\iota (w^{(\infty)}_{i+1})=f(w^{(\infty)}_i)$, and the orbit $w^{(\infty)}\equiv (w^{(\infty)}_i)_{i\in \ZZ}$ with constant homotopies is homotopic to $((z^1_i)_{i\in \ZZ}, (\alpha^0_i)_{i\in \ZZ})$. This means that the homotopy pseudo-orbit $(z, \alpha)$ is homotopic to the orbit $w^{(\infty)}$. 
\end{proof}

\subsection{Uniqueness of shadowing} 
\label{sub:hyp-unique}

In this subsection we prove two results. First we show that if two orbits of a hyperbolic system are homotopic, then they are equal. In particular, it follows that the shadowing orbit found in Theorem~\ref{thm:hyp-shadowing} is unique. Based on this fact we finish the proof of Theorem~\ref{thm:hyp-funct}. We note that the uniqueness result Proposition~\ref{prp:hom-hyp} is essential in the procedure to construct a Hubbard tree for a complex H\'enon map (see Lemma 5.14 of~\cite{I2}).

\begin{prp}
\label{prp:hom-hyp}
If two orbits of a hyperbolic system are homotopic, then they are equal.
\end{prp}

\begin{proof}
Let $\iota, f : X^1\to X^0$ be a hyperbolic system, where $X^n=M^n_x\times M^n_y$. Fix $y_0\in M^1_y$ and set $\tau=\tau_{y_0} : M_x^1\to X^1$ by $\tau(x)\equiv (x, y_0)$. Then, $\iota, \sigma : M^1_x\to M^0_x$ becomes a expanding system, where $\sigma\equiv \pi^0_x\circ f\circ \tau$.

Now, take two orbits $z=(z_i)_{i\in \ZZ}$ and $z'=(z'_i)_{i\in \ZZ}$ of the hyperbolic system which are homotopic by a homotopy $\beta=(\beta_i)_{i\in \ZZ}$. Then, $\pi^1_x(z)=(\pi^1_x(z_i))_{i\in \ZZ}$ and $\pi^1_x(z')=(\pi^1_x(z'_i))_{i\in \ZZ}$ become two orbits of the expanding system $\iota, \sigma : M^1_x\to M^0_x$ which are homotopic by the homotopy $\pi^1_x(\beta)=(\pi^1_x(\beta_i))_{i\in \ZZ}$. By the discussion on the uniqueness of the shadowing orbit for expanding systems, one can take the homotopy $\beta$ so that $l^1(\pi^1_x(\beta_i))=0$. It follows that each path $\beta_i$ is contained in a vertical manifold of the form $M^1_y(x_i)$.

Given a path $\delta=\delta(t)$ ($0\leq t\leq 1$) in $X^0_i$, we define $L_y(\delta)$ to be the infimum of $l^0_y(\delta')$ where $\delta'$ is homotopic to $\delta$ relative to endpoints. By the assumption, we have $f(\beta_{i-1})\sim \iota(\beta_i)$ in $X^0_i$ relative to endpoints with $l^1(\beta_i)\leq C''$ for some constant $C''\geq 0$. Then, this relation gives 
$$L_y(\iota(\beta_i))=L_y(f(\beta_{i-1})).$$ 
For any $\varepsilon>0$ we let $s$ be a path which is homotopic to $\iota(\beta_i)$ relative to endpoints with the property $L_y(\iota(\beta_i))+\varepsilon \geq l^0_y(s)$. Then, $f\iota^{-1}(s)$ is homotopic to $f(\beta_i)$ relative to endpoints. It follows from Corollary~\ref{cor:expansion} that $\lambda\cdot l^0_y(f\iota^{-1}(s))\leq l^0_y(s)$, so 
$$\lambda\cdot L_y(f(\beta_i))\leq L_y(\iota(\beta_i))+\varepsilon.$$
Combining these estimates we may conclude that
$$L_y(\iota(\beta_i))\leq \frac{1}{\lambda}(L_y(\iota(\beta_{i-1}))+\varepsilon).$$ 

Put $h(x)\equiv \frac{1}{\lambda}(x+\varepsilon)$. Then, the above inequality can be rewritten as $L_y(\iota(\beta_i))\leq h(L_y(\iota(\beta_{i-1})))$. By applying this repeatedly we obtain $L_y(\iota(\beta_i))\leq h^n(L_y(\iota(\beta_{i-n})))$. Since $L_y(\iota(\beta_{i-n}))\leq C''$ and since $h$ is monotone increasing, we get $L_y(\iota(\beta_i))\leq h^n(C'')$. The function $h$ has a unique attractive fixed point at $x=\varepsilon/(\lambda-1)$. Letting $n$ go to infinity gives
$$L_y(\iota(\beta_i))\leq \frac{\varepsilon}{\lambda-1}.$$
Since $\varepsilon>0$ is arbitrary, we conclude $L_y(\iota(\beta_i))=0$. This means that one may take $(\beta_i)_{i\in \ZZ}$ to be constant homotopies. Thus, we are done. 
\end{proof}

\begin{proof}[Proof of Theorem~\ref{thm:hyp-funct}.]
Take an orbit $x\in X^{\pm\infty}$. This defines a homotopy pseudo-orbit $h(x)$ of $\iota, g : Y^1\to Y^0$ by Lemma~\ref{lmm:orbhpo}. Thanks to Theorem ~\ref{thm:hyp-shadowing} and Proposition~\ref{prp:hom-hyp}, there exists a unique orbit $y$ of $\iota, g : Y^1\to Y^0$ which is homotopic to $h(x)$. Define $h^{\infty} : X^{\pm\infty}\to Y^{\pm\infty}$ by $h^{\infty}(x)\equiv y$. Then, one can easily verify $h^{\infty} \hat{f}=\hat{g}h^{\infty}$.

We next prove the continuity of $h^{\infty}$. Take two orbits $x=(x_i)_{i\in\ZZ}$ and $\tilde{x}=(\tilde{x}_i)_{i\in\ZZ}$ of $\iota, f : X^1\to X^0$ and let $h(x)=((z^1_i)_{i\in \ZZ}, (\alpha^0_i)_{i\in \ZZ})$ and $h(\tilde{x})=((\tilde{z}^1_i)_{i\in \ZZ}, (\tilde{\alpha}^0_i)_{i\in \ZZ})$ be the corresponding homotopy pseudo-orbits. Let $(w_i)_{i\in \ZZ}\equiv h^{\infty}(x)$ and $(\tilde{w}_i)_{i\in \ZZ}\equiv h^{\infty}(\tilde{x})$ be their unique shadowing orbits. As in Step II of the proof of the shadowing theorem, these orbits are obtained as $w_0=\lim_{n\to \infty}(\iota f)^n z_0^{2n+1}$ and $\tilde{w}_0=\lim_{n\to \infty}(\iota f)^n \tilde{z}_0^{2n+1}$. Then, for any $\varepsilon>0$ there exists $N>0$ such that 
$$d^1(w_0, (\iota f)^Nz_0^{2N+1})<\varepsilon \quad \mathrm{and} \quad d^1(\tilde{w}_0, (\iota f)^N\tilde{z}_0^{2N+1})<\varepsilon$$
hold. For this $N$, we can choose $\delta>0$ so that $d^1(x_i, \tilde{x}_i)<\delta$ implies $d^1(z^1_i, \tilde{z}^1_i)<\varepsilon$ ($-N\leq i\leq N$) by the continuity of $h^1$. It follows that $d_x^{2N+1}(z_0^{2N+1}, \tilde{z}_0^{2N+1})<\varepsilon$ and $d_y^{2N+1}(z_0^{2N+1}, \tilde{z}_0^{2N+1})<\varepsilon$ hold. Lemma~\ref{lmm:expansion} together with Proposition~\ref{prp:induction} imply that $d_x^1((\iota f)^Nz_0^{2N+1}, (\iota f)^N\tilde{z}_0^{2N+1}) < \varepsilon/\lambda^N$ and $d_y^1((\iota f)^Nz_0^{2N+1}, (\iota f)^N\tilde{z}_0^{2N+1}) < \varepsilon/\lambda^N$. Hence, 
$$d^1((\iota f)^Nz_0^{2N+1}, (\iota f)^N\tilde{z}_0^{2N+1}) < \frac{2\varepsilon}{\lambda^N} < 2\varepsilon.$$
Combining this with the above two estimates we obtain $d^1(w_0, \tilde{w}_0)<4 \varepsilon$ for $x, \tilde{x}\in X^{\pm\infty}$ with $d^1(x_i, \tilde{x}_i)<\delta$ ($-N\leq i\leq N$). This shows the continuity of $h^{\infty}$ and hence finishes the proof. 
\end{proof}

\end{section}

\begin{section}{Associated expanding systems of hyperbolic systems}
\label{sec:associated}

We start with the following theorem which is similar to Theorems~\ref{thm:exp-equiv} and \ref{thm:hyp-equiv}.

\begin{thm}
\label{thm:exphyp-equiv}
A homotopy equivalence between a hyperbolic system $\mathcal{X}=(X^0, X^1; \iota, f)$ and an expanding system $\mathcal{Y}=(Y^0, Y^1; \iota, g)$ induces a topological conjugacy between $\hat{f} : X^{\pm\infty}\to X^{\pm\infty}$ and $\hat{g} : Y^{\pm\infty}\to Y^{\pm\infty}$.
\end{thm}

The proof of this theorem is given in Section~\ref{sec:functorial}.

A specific situation where we can apply the above theorem is the following. Given a hyperbolic system $\mathcal{X}=(M^0_x\times M^0_y, M^1_x\times M^1_y; \iota, f)$ one can associate an expanding system as follows. Choose a point $y_0\in M^1_y$ and put 
$$\sigma_{y_0}\equiv \pi^0_x\circ f\circ \tau_{y_0} : M^1_x \longrightarrow M^0_x$$
and
$$\iota_{y_0}\equiv \pi^0_x\circ \iota \circ \tau_{y_0} : M^1_x\longrightarrow M^0_x,$$
where $\tau_{y_0} : M^1_x\to M^1_x\times M^1_y$ is given by $\tau_{y_0}(x)=(x, y_0)$. It is easy to see that $(M^0_x, M^1_x; \iota_{y_0}, \sigma_{y_0})$ becomes an expanding system.

\begin{dfn} 
\label{dfn:associated}
We call $(M^0_x, M^1_x; \iota_{y_0}, \sigma_{y_0})$ an {\it associated expanding system} of $\mathcal{X}$. 
\end{dfn}

We first examine the dependence on the choice of $y_0\in M^1_y$ in this definition.

\begin{prp} 
\label{prp:independence}
The homotopy equivalence class of an associated expanding system does not depend on the choice of $y_0\in M^1_y$.
\end{prp}

\begin{proof}
Let $(M^0_x\times M^0_y, M^1_x\times M^1_y; \iota, f)$ be a hyperbolic system. Take $y_i\in M^1_y$ and put $\mathcal{Y}_i=(Y_i^0, Y_i^1; \iota_i, \sigma_i)\equiv (M^0_x, M^0_y; \iota_{y_i}, \sigma_{y_i})$ for $i=0, 1$. We will show that $\mathcal{Y}_0$ and $\mathcal{Y}_1$ are homotopy equivalent.

Since $M^1_y$ is assumed to be connected, there exists a path $y : [0, 1]\to M^1_y$ such that $y(0)=y_0$ and $y(1)=y_1$. Let $h^0 : Y^0_0\to Y^0_1$, $k^0 : Y^0_1\to Y^0_0$, $h^1 : Y^1_0\to Y^1_1$ and $k^1 : Y^1_1\to Y^1_0$ be the identity maps. Define $G_t\equiv \sigma_{y(t)}$ and $H_t\equiv \iota_{y(t)}$, then we see that $h=(h^0, h^1; G, H)$ is a homotopy semi-conjugacy from $\mathcal{Y}_0$ to $\mathcal{Y}_1$. Similarly, define $G'_t\equiv \sigma_{y(1-t)}$ and $H'_t\equiv \iota_{y(1-t)}$, then we see that $k=(k^0, k^1; G', H')$ is a homotopy semi-conjugacy from $\mathcal{Y}_1$ to $\mathcal{Y}_0$. 

We will show that $kh=({\rm id}_{Y^0_0}, {\rm id}_{Y^1_0}; k^0G\cdot G'\cdot h^1, k^0H\cdot H'h^1)$ is homotopic to ${\rm id}_{\mathcal{Y}_0}=({\rm id}_{Y^0_0}, {\rm id}_{Y^1_0}; \sigma_0, \iota_0)$ in the sense of Definition~\ref{dfn:homotopic}. Let $S_t\equiv {\rm id}_{Y^1_0}$ and $T_s\equiv {\rm id}_{Y^0_0}$. We compute $\sigma_0{\rm id}_{Y^1_0}(x)\cdot \sigma_0(x)^{-1}=\sigma_0(x)\cdot\sigma_0(x)=\sigma_0(x)$ and $(k^0G(x)\cdot G'h^1(x))^{-1}\cdot {\rm id}_{Y^0_0}\sigma_0(x)=(\sigma_{y(t)}(x)\cdot \sigma_{y(1-t)}(x))^{-1}\cdot\sigma_0(x)=\sigma_{y(t)}(x)\cdot \sigma_{y(1-t)}(x)$, which are homotopic. This means that the condition (i) in Definition~\ref{dfn:homotopic} is verified. Similarly one can verify the condition (ii) in Definition~\ref{dfn:homotopic}. Hence,  $kh$ is homotopic to ${\rm id}_{\mathcal{Y}_0}$.

Similarly $hk$ is shown to be homotopic to the identity semi-conjugacy of $\mathcal{Y}_1$. This means that $\mathcal{Y}_0$ and $\mathcal{Y}_1$ are homotopy equivalent, and hence finishes the proof. 
\end{proof}

It then follows from Theorem~\ref{thm:exp-equiv} that {\it the topological conjugacy class of the shift map on the orbit space of an associated expanding system $(M^0_x, M^1_x; \iota_{y_0}, \sigma_{y_0})$ does not depend on $y_0\in M^1_y$.} Hence, we may drop $y_0$ in the notation of an associated expanding system.

\begin{thm} 
\label{thm:associated-equiv}
A hyperbolic system $\mathcal{X}=(M_x^0\times M_y^0, M_x^1\times M_y^1; \iota, f)$ with $M^0_y$ being contractible and its associated expanding system $\mathcal{Y}= (M_x^0, M_x^1; \iota, \sigma)$ are homotopy equivalent.
\end{thm} 

\begin{proof}
Let $\mathcal{X}=(X^0, X^1; \iota, f)\equiv (M_x^0\times M_y^0, M_x^1\times M_y^1; \iota, f)$ be a hyperbolic system and $\mathcal{Y}=(Y^0, Y^1; \iota, \sigma)\equiv (M_x^0, M_x^1; \iota, \sigma)$ be its associated expanding system, where $\sigma\equiv \pi^0_x\circ f\circ \tau^1$ and $\tau^1 : M_x^1\to M_x^1\times M_y^1$ is defined by $\tau^1(x)\equiv (x, y^1)$ for $y^1\in M_y^1$. We also define $\tau^0 : M_x^0\to M_x^0\times M_y^0$ by $\tau^0(x)\equiv (x, y^0)$ for $y^0\in M^0_y$. 

Let $h : \mathcal{X}\to \mathcal{Y}$ be a homotopy semi-conjugacy given by $h=(\pi^0_x, \pi^1_x; \widetilde{G}, \widetilde{H})$, where $\widetilde{G}=\widetilde{G}_t$ and $\widetilde{H}=\widetilde{H}_t$ are constant homotopies so that $\widetilde{G}_0=\pi^0_x f$, $\widetilde{G}_1=\sigma \pi^1_x$, $\widetilde{H}_0=\pi^0_x\iota$ and $\widetilde{H}_1=\iota\pi^1_x$ hold. Similarly, we let $k : \mathcal{Y}\to \mathcal{X}$ be a homotopy semi-conjugacy given by $k=(\tau^0, \tau^1; \widetilde{G}', \widetilde{H}')$, where $\widetilde{G}'_t\equiv U^0_tf\tau^1$ and $\widetilde{H}'_t\equiv U^0_t\iota\tau^1$. Here, $U^n_t : M_x^n\times M_y^n\to M_x^n\times M_y^n$ ($0\leq t\leq 1$) is a homotopy of the form $U^n_t(x, y)=(x, u^n_t(y))$ so that $U^n_0(x, y)=(x, y^n)$ and $U^n_1(x, y)=(x, y)$ hold (note that such a homotopy exists since $M^n_y$ is contractible). We then have $\widetilde{G}'_0=\tau^0 \sigma$, $\widetilde{G}'_1=f \tau^1$, $\widetilde{H}'_0=\tau^0\iota$ and $\widetilde{H}'_1=\iota\tau^1$. 

By the definition of the composition of homotopy semi-conjugacies, one has
$$kh=(\tau^0\pi^0_x, \tau^1\pi^1_x; \tau^0\widetilde{G}\cdot \widetilde{G}'\pi^1_x, \tau^0\widetilde{H}\cdot \widetilde{H}'\pi^1_x).$$
We first show that $kh$ is homotopic to the identity semi-conjugacy ${\rm id}_{\mathcal{X}}$ of $\mathcal{X}$ as a homotopy semi-conjugacy. Put $S_t\equiv U^1_t$ and $T_s\equiv U^0_s$. Then, $S_0=\tau^1\pi^1_x$, $S_1={\rm id}_{X^1}$, $T^0=\tau^0\pi^0_x$ and $T_1={\rm id}_{X^0}$ hold. Moreover, we compute $fS\cdot f=fS=fU^1_t=U^0_1fU^1_t$ and $(\tau^0\widetilde{G}\cdot \widetilde{G}'\pi^1_x)^{-1}\cdot Tf=(\widetilde{G}')^{-1}\pi^1_x\cdot \tau^0(\widetilde{G})^{-1}\cdot Tf=(U^0_s)^{-1}f\tau^1\pi^1_x\cdot \tau^0\pi^0_xf\cdot U^0_sf=(U^0_s)^{-1}fU^1_0\cdot U^0_0fU^1_t\cdot U^0_sfU^1_1$. Consider the homotopy: 
$$U^0_sfU^1_t(z) : [0, 1]\times [0, 1]\to X^0,$$
then we see that $U^0_sfU^1_0\cdot U^0_1f U^1_t \sim U^0_0f U^1_t\cdot U^0_s fU^1_1$. This implies 
$$fS\cdot f\sim (\tau^0\widetilde{G}\cdot \widetilde{G}'\pi^1_x)^{-1}\cdot Tf,$$
which is the condition (i) in Definition~\ref{dfn:homotopic} of the homotopy from $kh$ to ${\rm id}_{\mathcal{X}}$. Similarly the condition (ii) is verified.

Next we show that
$$hk=(\pi^0_x\tau^0, \pi^1_x\tau^1; \pi^0_x\widetilde{G}'\cdot\widetilde{G}\tau^1, \pi^0_x\widetilde{H}'\cdot\widetilde{H}\tau^1)$$
is homotopic to the identity semi-conjugacy ${\rm id}_{\mathcal{Y}}=({\rm id}_{Y^0}, {\rm id}_{Y^1}; \sigma, \iota)$ of $\mathcal{Y}$ as a homotopy semi-conjugacy. Put $S_t\equiv {\rm id}_{Y^1}$ and $T_s\equiv {\rm id}_{Y^0}$. Then, $S_0=\pi^1_x\tau^1$, $S_1={\rm id}_{Y^1}$, $T_0=\pi^0_x\tau^0$ and $T_1={\rm id}_{Y^0}$ hold. Moreover, we compute $\sigma S\cdot \sigma=\sigma {\rm id}_{Y^1}\cdot \sigma=\sigma$ and $(\pi^0_x\widetilde{G}'\cdot \widetilde{G}\tau^1)^{-1}\cdot {\rm id}_{Y^0}\sigma=(\widetilde{G})^{-1}\tau^1\cdot\pi^0_x(\widetilde{G}')^{-1}\cdot {\rm id}_{Y^0}\sigma=\sigma\pi^1_x\tau^1\cdot \pi^0_x(U^0_t)^{-1}f\tau^1\cdot {\rm id}_{Y^0}\sigma=\sigma\cdot\sigma\cdot\sigma=\sigma$. This implies the condition (i) in Definition~\ref{dfn:homotopic} of the homotopy from $hk$ to ${\rm id}_{\mathcal{Y}}$. Similarly the condition (ii) is verified. 
This finishes the proof.
\end{proof}

As an immediate consequence we get the third main result of this paper.

\begin{cor} 
\label{cor:associated-conj}
Let $\mathcal{X}=(X^0, X^1; \iota, f)=(M_x^0\times M_y^0, M_x^1\times M_y^1; \iota, f)$ be a hyperbolic system with $M^0_y$ being contractible and $\mathcal{Y}=(Y^0, Y^1; \iota, \sigma)$ be its associated expanding system. Then, $\hat{f} : X^{\pm\infty}\to X^{\pm\infty}$ and $\hat{\sigma} : Y^{\pm\infty}\to Y^{\pm\infty}$ are topologically conjugate.
\end{cor} 

\begin{proof}
This follows from Theorems~\ref{thm:exphyp-equiv} and \ref{thm:associated-equiv}.
\end{proof}

Application to the construction of Hubbard trees for H\'enon maps is given in~\cite{I2}.

\end{section}

\begin{section}{Functorial properties of homotopy semi-conjugacies}
\label{sec:functorial}

In this section we prove some results which deal with properties of compositions of homotopy semi-conjugacies and uniqueness of the conjugacies determined by homotopy semi-conjugacies. This will conclude the proofs of Theorems~\ref{thm:exp-equiv}, \ref{thm:hyp-equiv} and \ref{thm:exphyp-equiv}. These results apply in both the expanding case and the hyperbolic case.

 \begin{figure}
\setlength{\unitlength}{1.2mm}

\begin{picture}(100, 80)(30, 0)

\qbezier(50, 70)(51, 68)(53, 68.5)
\qbezier(53, 68.5)(55, 69)(56, 67)
\qbezier(56, 67)(57, 65)(59, 65.5)
\qbezier(59, 65.5)(61, 66)(62, 64)
\qbezier(62, 64)(63, 62)(65, 62.5)
\qbezier(65, 62.5)(67, 63)(68, 61)
\qbezier(68, 61)(69, 59)(71, 59.5)
\qbezier(71, 59.5)(73, 60)(74, 58)
\qbezier(74, 58)(75, 56)(77, 56.5)
\qbezier(77, 56.5)(79, 57)(80, 55)
\qbezier(80, 55)(81, 53)(83, 53.5)
\qbezier(83, 53.5)(85, 54)(86, 52)
\qbezier(86, 52)(87, 50)(89, 50.5)
\qbezier(89, 50.5)(91, 51)(92, 49)
\qbezier(92, 49)(93, 47)(95, 47.5)
\qbezier(95, 47.5)(97, 48)(98, 46)
\qbezier(98, 46)(99, 44)(101, 44.5)
\qbezier(101, 44.5)(103, 45)(104, 43)
\qbezier(104, 43)(105, 41)(107, 41.5)
\qbezier(107, 41.5)(109, 42)(110, 40)

\qbezier(50, 35)(51, 33)(53, 33.5)
\qbezier(53, 33.5)(55, 34)(56, 32)
\qbezier(56, 32)(57, 30)(59, 30.5)
\qbezier(59, 30.5)(61, 31)(62, 29)
\qbezier(62, 29)(63, 27)(65, 27.5)
\qbezier(65, 27.5)(67, 28)(68, 26)
\qbezier(68, 26)(69, 24)(71, 24.5)
\qbezier(71, 24.5)(73, 25)(74, 23)
\qbezier(74, 23)(75, 21)(77, 21.5)
\qbezier(77, 21.5)(79, 22)(80, 20)
\qbezier(80, 20)(81, 18)(83, 18.5)
\qbezier(83, 18.5)(85, 19)(86, 17)
\qbezier(86, 17)(87, 15)(89, 15.5)
\qbezier(89, 15.5)(91, 16)(92, 14)
\qbezier(92, 14)(93, 12)(95, 12.5)
\qbezier(95, 12.5)(97, 13)(98, 11)
\qbezier(98, 11)(99, 9)(101, 9.5)
\qbezier(101, 9.5)(103, 10)(104, 8)
\qbezier(104, 8)(105, 6)(107, 6.5)
\qbezier(107, 6.5)(109, 7)(110, 5)

\put(50, 35){\line(0, 1){35}}
\put(80, 20){\line(0, 1){35}}
\put(110, 5){\line(0, 1){35}}

\put(50, 70){\circle*{1.5}}
\put(50, 35){\circle*{1.5}}
\put(80, 55){\circle*{1.5}}
\put(80, 20){\circle*{1.5}}
\put(110, 40){\circle*{1.5}}
\put(110, 5){\circle*{1.5}}

\put(50, 52.5){\vector(0, -1){1}}
\put(80, 39.5){\vector(0, -1){1}}
\put(110, 22.5){\vector(0, -1){1}}

\put(48.3, 34.5){\vector(3, 1){1}}
\put(49, 39.5){\vector(3, 1){1}}
\put(49, 44.7){\vector(3, 1){1}}
\put(49, 49.7){\vector(3, 1){1}}
\put(49, 54.7){\vector(3, 1){1}}
\put(49, 59.7){\vector(3, 1){1}}
\put(49, 64.7){\vector(3, 1){1}}
\put(48.3, 69.5){\vector(3, 1){1}}
\put(139, 34.7){\vector(3, 1){1}}
\put(139, 39.7){\vector(3, 1){1}}
\put(139, 44.7){\vector(3, 1){1}}
\put(139, 49.7){\vector(3, 1){1}}
\put(139, 54.7){\vector(3, 1){1}}
\put(139, 59.7){\vector(3, 1){1}}
\put(139, 64.7){\vector(3, 1){1}}
\put(139, 69.7){\vector(3, 1){1}}

\put(65, 45){\circle{8}}
\put(95, 30){\circle{8}}
\put(65.2, 49.2){\vector(-1, 0){1}}
\put(95.2, 34.2){\vector(-1, 0){1}}

\qbezier(20, 5)(30, 30)(50, 35)
\qbezier(20, 40)(30, 65)(50, 70)
\qbezier(110, 5)(120, 30)(140, 35)
\qbezier(110, 40)(120, 65)(140, 70)

\small
\put(29, 40){$S(x_{i-1})$}
\put(26, 66){$h^1(x_{i-1})$}
\put(29, 17){$k^1(x_{i-1})$}
\put(119, 40){$S(x_i)$}
\put(119, 66){$h^1(x_i)$}
\put(119, 17){$k^1(x_i)$}
\put(60, 67){$G(x_{i-1})^{-1}$}
\put(93, 50){$H(x_i)$}
\put(57, 20){$G'(x_{i-1})^{-1}$}
\put(89, 7){$H'(x_i)$}
\put(52, 53){$gS(x_{i-1})$}
\put(65.2, 34){$Tf(x_{i-1})=T\iota(x_i)$}
\put(100, 23){$\iota S(x_i)$}

\linethickness{0.2pt}

\qbezier(20, 10)(30, 35)(50, 40)
\qbezier(20, 15)(30, 40)(50, 45)
\qbezier(20, 20)(30, 45)(50, 50)
\qbezier(20, 25)(30, 50)(50, 55)
\qbezier(20, 30)(30, 55)(50, 60)
\qbezier(20, 35)(30, 60)(50, 65)
\qbezier(110, 10)(120, 35)(140, 40)
\qbezier(110, 15)(120, 40)(140, 45)
\qbezier(110, 20)(120, 45)(140, 50)
\qbezier(110, 25)(120, 50)(140, 55)
\qbezier(110, 30)(120, 55)(140, 60)
\qbezier(110, 35)(120, 60)(140, 65)

\end{picture}
\vspace{1cm}
\begin{center}
Figure 10. Homotopy pseudo-orbits $h(x)$ and $k(x)$ are homotopic.
\end{center}
\end{figure}

Let $h=(h^0, h^1; G, H)$ and $k=(k^0, k^1; G', H')$ be two homotopy semi-conjugacies from a multivalued dynamical system $\mathcal{X}=(X^0, X^1; \iota, f)$ to another multivalued dynamical system $\mathcal{Y}=(Y^0, Y^1; \iota, g)$. 

\begin{lmm} 
\label{lmm:twohsc-orbit}
Let $h$ and $k$ be two homotopy semi-conjugacies from $\mathcal{X}$ to $\mathcal{Y}$ which are homotopic. Then, for any orbit $x$ of $\mathcal{X}$, the homotopy pseudo-orbits $h(x)$ and $k(x)$ are homotopic.
\end{lmm}

This situation is described in Figure 10. Note that the two diagrams in the figure are commutative up to homotopies. More generally, we show the following ``tongue twister'': two homotopy semi-conjugacies which are homotopic send a homotopy pseudo-orbit to two homotopy pseudo-orbits which are homotopic. That is,

\begin{lmm} 
\label{lmm:twohsc-hpo}
Let $h$ and $k$ be two homotopy semi-conjugacies from $\mathcal{X}$ to $\mathcal{Y}$ which are homotopic. Then, for any homotopy pseudo-orbit $(x, \alpha)$ of $\mathcal{X}$, the homotopy pseudo-orbits $h(x, \alpha)$ and $k(x, \alpha)$ are homotopic.
\end{lmm}

 \begin{figure}
\setlength{\unitlength}{1.2mm}

\begin{picture}(100, 80)(30, 0)

\qbezier(50, 70)(51, 68)(53, 68.5)
\qbezier(53, 68.5)(55, 69)(56, 67)
\qbezier(56, 67)(57, 65)(59, 65.5)
\qbezier(59, 65.5)(61, 66)(62, 64)
\qbezier(62, 64)(63, 62)(65, 62.5)
\qbezier(65, 62.5)(67, 63)(68, 61)
\qbezier(68, 61)(69, 59)(71, 59.5)
\qbezier(71, 59.5)(73, 60)(74, 58)
\qbezier(74, 58)(75, 56)(77, 56.5)
\qbezier(77, 56.5)(79, 57)(80, 55)
\qbezier(80, 55)(81, 53)(83, 53.5)
\qbezier(83, 53.5)(85, 54)(86, 52)
\qbezier(86, 52)(87, 50)(89, 50.5)
\qbezier(89, 50.5)(91, 51)(92, 49)
\qbezier(92, 49)(93, 47)(95, 47.5)
\qbezier(95, 47.5)(97, 48)(98, 46)
\qbezier(98, 46)(99, 44)(101, 44.5)
\qbezier(101, 44.5)(103, 45)(104, 43)
\qbezier(104, 43)(105, 41)(107, 41.5)
\qbezier(107, 41.5)(109, 42)(110, 40)

\qbezier(50, 35)(51, 33)(53, 33.5)
\qbezier(53, 33.5)(55, 34)(56, 32)
\qbezier(56, 32)(57, 30)(59, 30.5)
\qbezier(59, 30.5)(61, 31)(62, 29)
\qbezier(62, 29)(63, 27)(65, 27.5)
\qbezier(65, 27.5)(67, 28)(68, 26)
\qbezier(68, 26)(69, 24)(71, 24.5)
\qbezier(71, 24.5)(73, 25)(74, 23)
\qbezier(74, 23)(75, 21)(77, 21.5)
\qbezier(77, 21.5)(79, 22)(80, 20)
\qbezier(80, 20)(81, 18)(83, 18.5)
\qbezier(83, 18.5)(85, 19)(86, 17)
\qbezier(86, 17)(87, 15)(89, 15.5)
\qbezier(89, 15.5)(91, 16)(92, 14)
\qbezier(92, 14)(93, 12)(95, 12.5)
\qbezier(95, 12.5)(97, 13)(98, 11)
\qbezier(98, 11)(99, 9)(101, 9.5)
\qbezier(101, 9.5)(103, 10)(104, 8)
\qbezier(104, 8)(105, 6)(107, 6.5)
\qbezier(107, 6.5)(109, 7)(110, 5)

\put(50, 35){\line(0, 1){35}}
\put(70, 25){\line(0, 1){35}}
\put(90, 15){\line(0, 1){35}}
\put(110, 5){\line(0, 1){35}}

\put(50, 70){\circle*{1.5}}
\put(50, 35){\circle*{1.5}}
\put(70, 59.5){\circle*{1.5}}
\put(70, 24.5){\circle*{1.5}}
\put(90, 50.5){\circle*{1.5}}
\put(90, 15.5){\circle*{1.5}}
\put(110, 40){\circle*{1.5}}
\put(110, 5){\circle*{1.5}}

\put(50, 52.5){\vector(0, -1){1}}
\put(70, 42.5){\vector(0, -1){1}}
\put(90, 32.5){\vector(0, -1){1}}
\put(110, 22.5){\vector(0, -1){1}}

\put(48.3, 34.6){\vector(3, 1){1}}
\put(49, 39.8){\vector(3, 1){1}}
\put(49, 44.8){\vector(3, 1){1}}
\put(49, 49.8){\vector(3, 1){1}}
\put(49, 54.8){\vector(3, 1){1}}
\put(49, 59.8){\vector(3, 1){1}}
\put(49, 64.8){\vector(3, 1){1}}
\put(48.3, 69.6){\vector(3, 1){1}}
\put(139, 34.8){\vector(3, 1){1}}
\put(139, 39.8){\vector(3, 1){1}}
\put(139, 44.8){\vector(3, 1){1}}
\put(139, 49.8){\vector(3, 1){1}}
\put(139, 54.8){\vector(3, 1){1}}
\put(139, 59.8){\vector(3, 1){1}}
\put(139, 64.8){\vector(3, 1){1}}
\put(139, 69.8){\vector(3, 1){1}}

\put(60, 47.5){\circle{7}}
\put(100, 27.5){\circle{7}}
\put(60.2, 51){\vector(-1, 0){1}}
\put(100.2, 31){\vector(-1, 0){1}}

\qbezier(20, 5)(30, 30)(50, 35)
\qbezier(20, 40)(30, 65)(50, 70)
\qbezier(110, 5)(120, 30)(140, 35)
\qbezier(110, 40)(120, 65)(140, 70)

\small
\put(29, 42){$S(x_{i-1})$}
\put(26, 66){$h^1(x_{i-1})$}
\put(29, 17){$k^1(x_{i-1})$}
\put(119, 42){$S(x_i)$}
\put(119, 66){$h^1(x_i)$}
\put(119, 17){$k^1(x_i)$}
\put(57, 68){$G(x_{i-1})^{-1}$}
\put(77, 58){$h^0(\alpha_i)$}
\put(98, 48){$H(x_i)$}
\put(52, 23){$G'(x_{i-1})^{-1}$}
\put(76, 14){$k^0(\alpha_i)$}
\put(95, 5){$H'(x_i)$}
\put(51.5, 54){$gS(x_{i-1})$}
\put(100.5, 19){$\iota S(x_i)$}
\put(71, 44){$Tf(x_{i-1})$}
\put(80.5, 29){$T\iota(x_i)$}
\put(85, 70){$T_s(\alpha_i(t))$ fills up this.}
\qbezier(82, 37.5)(86, 47.5)(94, 67.5)

\linethickness{0.2pt}

\qbezier(20, 10)(30, 35)(50, 40)
\qbezier(20, 15)(30, 40)(50, 45)
\qbezier(20, 20)(30, 45)(50, 50)
\qbezier(20, 25)(30, 50)(50, 55)
\qbezier(20, 30)(30, 55)(50, 60)
\qbezier(20, 35)(30, 60)(50, 65)
\qbezier(110, 10)(120, 35)(140, 40)
\qbezier(110, 15)(120, 40)(140, 45)
\qbezier(110, 20)(120, 45)(140, 50)
\qbezier(110, 25)(120, 50)(140, 55)
\qbezier(110, 30)(120, 55)(140, 60)
\qbezier(110, 35)(120, 60)(140, 65)

\end{picture}
\vspace{1cm}
\begin{center}
Figure 11. Homotopy pseudo-orbits $h(x, \alpha)$ and $k(x, \alpha)$ are homotopic.
\end{center}
\end{figure}

\begin{proof}
We will show that  $h(x, \alpha)=(h^1(x), G(x)^{-1}\cdot h^0(\alpha)\cdot H(x))$ is homotopic to $k(x, \alpha)=(k^1(x), G'(x)^{-1}\cdot k^0(\alpha)\cdot H'(x))$. By the definition of homotopy between two homotopy semi-conjugacies, we have
$$gS(x_{i-1})\cdot G'(x_{i-1})^{-1}\sim G(x_{i-1})^{-1}\cdot Tf(x_{i-1})$$
and 
$$H(x_i)\cdot \iota S(x_i)\sim T\iota (x_i)\cdot H'(x_i).$$
Define $T_s(\alpha_i(t)) : [0, 1]\times [0, 1]\to Y^0$. Then, we see $T_0(\alpha_i)=h^0\alpha_i$, $T_1(\alpha_i)=k^0\alpha_i$, $T_s(\alpha_i(0))=Tf(x_{i-1})$ and $T_s(\alpha_i(1))=T\iota (x_i)$. Thus, 
$$h^0(\alpha_i)\cdot T\iota(x_i)\sim Tf(x_{i-1})\cdot k^0(\alpha_i).$$
By combining these three homotopies, we have
\begin{align*}
(G(x_{i-1})^{-1}\cdot h^0(\alpha_i)\cdot H(x_i))\cdot \iota S(x_i) 
& \sim G(x_{i-1})^{-1}\cdot h^0(\alpha_i)\cdot T\iota(x_i)\cdot H'(x_i) \\
& \sim G(x_{i-1})^{-1}\cdot Tf(x_{i-1})\cdot k^0(\alpha_i)\cdot H'(x_i) \\
& \sim gS(x_{i-1})\cdot (G'(x_{i-1})^{-1}\cdot k^0(\alpha_i)\cdot H'(x_i)).
\end{align*}
See Figure 11. Note that the left-most and the right-most diagrams in the figure are commutative up to homotopies. This shows that $h(x, \alpha)$ is homotopic to $k(x, \alpha)$ by a sequence of homotopies $(\beta_i(t))\equiv (S_t(x_i))$.
\end{proof}

For a homotopy semi-conjugacy $h$ from a multivalued dynamical system $\mathcal{X}$ to an expanding or a hyperbolic system $\mathcal{Y}$, let $h^{\infty} : X^{\infty}\to Y^{\infty}$ be the corresponding semi-conjugacy obtained via Theorem~\ref{thm:exp-funct} or Theorem~\ref{thm:hyp-funct}.

\begin{cor} 
\label{cor:coincide}
If two homotopy semi-conjugacies $h$ and $k$ from $\mathcal{X}$ to an expanding or a hyperbolic system $\mathcal{Y}$ are homotopic, then $h^{\infty}=k^{\infty}$.
\end{cor}

\begin{proof}
Take an orbit $x\in X^{\infty}$. Then, by Lemma~\ref{lmm:twohsc-orbit}, we see that the homotopy pseudo-orbits $h(x)$ and $k(x)$ are homotopic. Since the corresponding shadowing orbit $h^{\infty}(x)$ for $h(x)$ is homotopic to $h(x)$ and $k^{\infty}(x)$ for $k(x)$ is homotopic to $k(x)$, we see that $h^{\infty}(x)$ is homotopic to $k^{\infty}(x)$. By the uniqueness of the shadowing orbit Corollary~\ref{cor:hom-exp} for the expanding case and Proposition~\ref{prp:hom-hyp} for the hyperbolic case, we have $h^{\infty}(x)=k^{\infty}(x)$.
\end{proof}

\begin{cor} 
\label{cor:identity}
If a homotopy semi-conjugacy $h$ from an expanding or a hyperbolic system $\mathcal{X}$ to itself is homotopic to the identity semi-conjugacy ${\rm id}_{\mathcal{X}}$, then $h^{\infty}$ is equal to the identity map ${\rm id}_{X^{\infty}}$ on $X^{\infty}$.
\end{cor}

\begin{proof}
This readily follows from the definition of the identity semi-conjugacy.
\end{proof}

We next show the second ``tongue twister'': a homotopy semi-conjugacy takes two homotopy pseudo-orbits which are homotopic to two homotopy pseudo-orbits which are homotopic. That is,

\begin{lmm} 
\label{lmm:hom-preserved}
Let $h : \mathcal{X}\to \mathcal{Y}$ be a homotopy semi-conjugacy. If a homotopy pseudo-orbit $(x, \alpha)$ of $\mathcal{X}$ is homotopic to $(x', \alpha')$ by a homotopy $\beta$, then the homotopy pseudo-orbit $h(x, \alpha)$ is homotopic to $h(x', \alpha')$ by $h^1\beta$.
\end{lmm}

 \begin{figure}
\setlength{\unitlength}{1.2mm}

\begin{picture}(100, 90)(30, 0)

\qbezier(50, 70)(51, 68)(53, 68.5)
\qbezier(53, 68.5)(55, 69)(56, 67)
\qbezier(56, 67)(57, 65)(59, 65.5)
\qbezier(59, 65.5)(61, 66)(62, 64)
\qbezier(62, 64)(63, 62)(65, 62.5)
\qbezier(65, 62.5)(67, 63)(68, 61)
\qbezier(68, 61)(69, 59)(71, 59.5)
\qbezier(71, 59.5)(73, 60)(74, 58)
\qbezier(74, 58)(75, 56)(77, 56.5)
\qbezier(77, 56.5)(79, 57)(80, 55)
\qbezier(80, 55)(81, 53)(83, 53.5)
\qbezier(83, 53.5)(85, 54)(86, 52)
\qbezier(86, 52)(87, 50)(89, 50.5)
\qbezier(89, 50.5)(91, 51)(92, 49)
\qbezier(92, 49)(93, 47)(95, 47.5)
\qbezier(95, 47.5)(97, 48)(98, 46)
\qbezier(98, 46)(99, 44)(101, 44.5)
\qbezier(101, 44.5)(103, 45)(104, 43)
\qbezier(104, 43)(105, 41)(107, 41.5)
\qbezier(107, 41.5)(109, 42)(110, 40)

\qbezier(50, 35)(51, 33)(53, 33.5)
\qbezier(53, 33.5)(55, 34)(56, 32)
\qbezier(56, 32)(57, 30)(59, 30.5)
\qbezier(59, 30.5)(61, 31)(62, 29)
\qbezier(62, 29)(63, 27)(65, 27.5)
\qbezier(65, 27.5)(67, 28)(68, 26)
\qbezier(68, 26)(69, 24)(71, 24.5)
\qbezier(71, 24.5)(73, 25)(74, 23)
\qbezier(74, 23)(75, 21)(77, 21.5)
\qbezier(77, 21.5)(79, 22)(80, 20)
\qbezier(80, 20)(81, 18)(83, 18.5)
\qbezier(83, 18.5)(85, 19)(86, 17)
\qbezier(86, 17)(87, 15)(89, 15.5)
\qbezier(89, 15.5)(91, 16)(92, 14)
\qbezier(92, 14)(93, 12)(95, 12.5)
\qbezier(95, 12.5)(97, 13)(98, 11)
\qbezier(98, 11)(99, 9)(101, 9.5)
\qbezier(101, 9.5)(103, 10)(104, 8)
\qbezier(104, 8)(105, 6)(107, 6.5)
\qbezier(107, 6.5)(109, 7)(110, 5)

\put(50, 35){\line(0, 1){35}}
\put(70, 25){\line(0, 1){35}}
\put(90, 15){\line(0, 1){35}}
\put(110, 5){\line(0, 1){35}}

\put(50, 70){\circle*{1.5}}
\put(50, 35){\circle*{1.5}}
\put(70, 59.5){\circle*{1.5}}
\put(70, 24.5){\circle*{1.5}}
\put(90, 50.5){\circle*{1.5}}
\put(90, 15.5){\circle*{1.5}}
\put(110, 40){\circle*{1.5}}
\put(110, 5){\circle*{1.5}}

\put(50, 52.5){\vector(0, -1){1}}
\put(70, 42.5){\vector(0, -1){1}}
\put(90, 32.5){\vector(0, -1){1}}
\put(110, 22.5){\vector(0, -1){1}}

\put(48.3, 34.6){\vector(3, 1){1}}
\put(49, 39.8){\vector(3, 1){1}}
\put(49, 44.8){\vector(3, 1){1}}
\put(49, 49.8){\vector(3, 1){1}}
\put(49, 54.8){\vector(3, 1){1}}
\put(49, 59.8){\vector(3, 1){1}}
\put(49, 64.8){\vector(3, 1){1}}
\put(48.3, 69.6){\vector(3, 1){1}}
\put(139, 34.8){\vector(3, 1){1}}
\put(139, 39.8){\vector(3, 1){1}}
\put(139, 44.8){\vector(3, 1){1}}
\put(139, 49.8){\vector(3, 1){1}}
\put(139, 54.8){\vector(3, 1){1}}
\put(139, 59.8){\vector(3, 1){1}}
\put(139, 64.8){\vector(3, 1){1}}
\put(139, 69.8){\vector(3, 1){1}}

\put(80, 37.5){\circle{7}}
\put(80.2, 41){\vector(-1, 0){1}}

\qbezier(20, 5)(30, 30)(50, 35)
\qbezier(20, 40)(30, 65)(50, 70)
\qbezier(110, 5)(120, 30)(140, 35)
\qbezier(110, 40)(120, 65)(140, 70)

\small
\put(29, 42){$h^1(\beta_{i-1})$}
\put(29, 68){$h^1(x_{i-1})$}
\put(29, 17){$h^1(x'_{i-1})$}
\put(119, 42){$h^1(\beta_i)$}
\put(119, 66){$h^1(x_i)$}
\put(119, 17){$h^1(x'_i)$}
\put(57, 68){$G(x_{i-1})^{-1}$}
\put(77, 58){$h^0(\alpha_i)$}
\put(98, 48){$H(x_i)$}
\put(52, 23){$G(x'_{i-1})^{-1}$}
\put(76, 14){$h^0(\alpha'_i)$}
\put(95, 5){$H(x'_i)$}
\put(51.5, 54){$gh^1(\beta_{i-1})$}
\put(100.5, 24){$\iota S(x_i)$}
\put(71, 44){$h^0f(\beta_{i-1})$}
\put(79.5, 29){$h^0\iota(\beta_i)$}
\put(70, 82){$G_s(\beta_{i-1}(t))$ fills up this.}
\put(85, 72){$H_s(\beta^1(t))$ fills up this.}
\qbezier(62, 46)(70, 63)(78, 80)
\qbezier(100, 30)(96, 50)(92, 70)

\linethickness{0.2pt}

\qbezier(20, 10)(30, 35)(50, 40)
\qbezier(20, 15)(30, 40)(50, 45)
\qbezier(20, 20)(30, 45)(50, 50)
\qbezier(20, 25)(30, 50)(50, 55)
\qbezier(20, 30)(30, 55)(50, 60)
\qbezier(20, 35)(30, 60)(50, 65)
\qbezier(110, 10)(120, 35)(140, 40)
\qbezier(110, 15)(120, 40)(140, 45)
\qbezier(110, 20)(120, 45)(140, 50)
\qbezier(110, 25)(120, 50)(140, 55)
\qbezier(110, 30)(120, 55)(140, 60)
\qbezier(110, 35)(120, 60)(140, 65)

\end{picture}
\vspace{1cm}
\begin{center}
Figure 12. Homotopy pseudo-orbits $h(x, \alpha)$ and $h(x', \alpha')$ are homotopic.
\end{center}
\end{figure}

\begin{proof}
What we have to check is 
$$(G(x_{i-1})^{-1}\cdot h^0\alpha_i \cdot H(x_i))\cdot \iota (h^1\beta_i)\sim g(h^1\beta_{i-1})\cdot (G(x_i')^{-1}\cdot h^0\alpha'_i \cdot H(x'_i)).$$

From the assumption $\alpha_i\cdot \iota(\beta_i)\sim f(\beta_{i-1})\cdot\alpha_i'$ we have 
$$h^0\alpha_i\cdot h^0\iota(\beta_i)\sim h^0f(\beta_{i-1})\cdot h^0\alpha_i'.$$
Define $G_s(\beta_{i-1}(t)) : [0, 1]\times[0, 1]\to Y^0$. Then, we see that $G_0(\beta_{i-1})=h^0f(\beta_{i-1})$, $G_s(\beta_{i-1}(0))=G(x_{i-1})^{-1}$, $G_1(\beta_{i-1})=gh^1(\beta_{i-1})$ and $G_s(\beta_{i-1}(1))=G(x'_{i-1})^{-1}$ hold. Thus, we obtain
$$gh^1(\beta_{i-1})\cdot G(x_{i-1}')^{-1}\sim G(x_{i-1})^{-1}\cdot h^0f(\beta_{i-1}).$$
Define $H_s(\beta_i(t)) : [0, 1]\times[0, 1]\to Y^0$. Then, similarly we get
$$h^0\iota (\beta_i)\cdot H(x_i')\sim H(x_i)\cdot \iota h^1(\beta_i).$$
Combining these homotopies we see
\begin{align*}
(G(x_{i-1})^{-1}\cdot h^0\alpha_i \cdot H(x_i))\cdot \iota h^1(\beta_i)
\sim & G(x_{i-1})^{-1}\cdot h^0\alpha_i \cdot h^0\iota (\beta_i)\cdot H(x_i') \\
\sim & G(x_{i-1})^{-1}\cdot h^0f(\beta_{i-1})\cdot h^0\alpha_i'\cdot H(x_i') \\
\sim & gh^1(\beta_{i-1})\cdot (G(x_{i-1}')^{-1}\cdot h^0\alpha_i'\cdot H(x_i')).
\end{align*}
See Figure 12. Note that the middle diagram in the figure is commutative up to homotopy. Thus, we are done.
\end{proof}

Lemma~\ref{lmm:hom-preserved} implies that the composition of homotopy semi-conjugacies induces the composition of corresponding semi-conjugacies. More precisely,

\begin{cor} 
\label{cor:composition}
Let $h$ be a homotopy semi-conjugacy from $\mathcal{X}$ to an expanding or a hyperbolic system $\mathcal{Y}$ and let $k$ be from $\mathcal{Y}$ to an expanding or a hyperbolic system $\mathcal{Z}$. Then, we have $(hk)^{\infty}=h^{\infty}k^{\infty}$.
\end{cor}

\begin{proof}
Take $x\in X^{\infty}$. Then, $(h^1(x), G(x)^{-1}\cdot H(x))$ becomes a homotopy pseudo-orbit of $\mathcal{Y}$ by Lemma~\ref{lmm:orbhpo}. By the shadowing theorem, there exists an orbit $y\equiv h^{\infty}(x)\in Y^{\infty}$ so that 
$$y\sim (h^1(x), G(x)^{-1}\cdot H(x)).$$ 
It then follows from Lemma~\ref{lmm:hom-preserved} that 
$$(k^1(y), G'(y)^{-1}\cdot H'(y))\sim (k^1h^1(x), G'(h^1(x))^{-1}\cdot k^0(G(x)^{-1}\cdot H(x))\cdot H'(h^1(x))).$$
Again by shadowing theorem, there exists an orbit $z\equiv k^{\infty}(y)\in Z^{\infty}$ so that 
$$z\sim(k^1(y), G'(y)^{-1}\cdot H'(y)).$$
Thus, by the definition of the composition $kh$, we see
\begin{align*}
z \sim & (k^1h^1(x), G'(h^1(x))^{-1}\cdot k^0(G(x)^{-1}\cdot H(x))\cdot H'(h^1(x))) \\
\sim & (k^1h^1(x), (k^0G(x)\cdot G'(h^1(x)))^{-1}\cdot k^0H(x)\cdot H'(h^1(x))) \\
\equiv & (kh)(x).
\end{align*}
The homotopy pseudo-orbit $(kh)(x)$ can be shadowed by an orbit $(kh)^{\infty}(x)\in Z^{\infty}$ which is homotopic to $(kh)(x)$. Thus, the orbit $z=k^{\infty}(y)=k^{\infty}(h^{\infty}(x))$ is homotopic to $(kh)^{\infty}(x)$. By the uniqueness of the shadowing orbit Corollary~\ref{cor:hom-exp} for the expanding case and Proposition~\ref{prp:hom-hyp} for the hyperbolic case, we have $k^{\infty}(h^{\infty}(x))=(kh)^{\infty}(x)$.
\end{proof}

Finally we show

\begin{proof}[Proof of Theorems~\ref{thm:exp-equiv}, \ref{thm:hyp-equiv} and \ref{thm:exphyp-equiv}.]
Let $h$ be a homotopy semi-conjugacy from $\mathcal{X}$ to $\mathcal{Y}$ and $k$ be a homotopy semi-conjugacy from $\mathcal{Y}$ to $\mathcal{X}$ so that $kh$ is homotopic to the identity semi-conjugacy ${\rm id}_{\mathcal{X}}$ of $\mathcal{X}$ and $hk$ is homotopic to the identity semi-conjugacy ${\rm id}_{\mathcal{Y}}$ of $\mathcal{Y}$. Corollary~\ref{cor:identity} says $(hk)^{\infty}=({\rm id}_{\mathcal{X}})^{\infty}={\rm id}_{X^{\infty}}$. Corollary~\ref{cor:composition} implies $(hk)^{\infty}=h^{\infty}k^{\infty}$, so $h^{\infty}k^{\infty}={\rm id}_{X^{\infty}}$. Similarly we have $k^{\infty}h^{\infty}={\rm id}_{Y^{\infty}}$. This finishes the proof.
\end{proof}

\begin{proof}[Proof of Theorem~\ref{thm:exp}.]
This result readily follows from Theorems~\ref{thm:exp-equiv} and \ref{thm:hyp-equiv} by letting $X^0=Y^0=U$, $X^1=Y^1=U\cap f^{-1}(U)$, $f : U\cap  f^{-1}(U)\to U$ be the restriction of $f : M\to M$ and $\iota : U\cap f^{-1}(U)\to U$ be the inclusion.
\end{proof}

\begin{proof}[Proof of Theorem~\ref{thm:hyp}.]
Similarly this readily follows from Corollary~\ref{cor:associated-conj}.
\end{proof}

\end{section}

\end{document}